\documentclass[aos]{imsart}
\RequirePackage[OT1]{fontenc}
\RequirePackage{amsthm,amsmath}
\RequirePackage[colorlinks,citecolor=blue,urlcolor=blue]{hyperref}

\usepackage[latin1]{inputenc}
\usepackage{lipsum}
\usepackage{amsmath}
\usepackage{bm}
\usepackage{amsfonts}
\usepackage{amsthm}
\usepackage{cleveref}
\usepackage[shortlabels]{enumitem}
\usepackage{txfonts}
\usepackage{stmaryrd}
\usepackage{amssymb}
\usepackage{makeidx}
\usepackage{appendix}
\usepackage{afterpage}
\usepackage{mathrsfs}
\usepackage{natbib}
\usepackage{relsize}
\usepackage{xcolor}

\setcitestyle{authoryear,open={(},close={)}}
\usepackage{graphicx}
\usepackage{subcaption}
\usepackage{epstopdf}
\usepackage{float}
\newtheorem{assumption}{Assumption}
\newtheorem{proposition}{Proposition}[section]
\newtheorem{theorem}{Theorem}[section]
\newtheorem{lemma}{Lemma}[section]
\newtheorem{corollary}{Corollary}[section]

\newtheorem{Example}{Example}

\numberwithin{equation}{section}
\numberwithin{proposition}{section}
\numberwithin{corollary}{section}
\numberwithin{theorem}{section}
\numberwithin{lemma}{section}
\newcommand{\ds}{\displaystyle}
\begin{document}
\begin{frontmatter}
\title{Improved Gaussian Mean Matrix Estimators in high-dimensional data
} \runtitle{On Improved Estimator of Mean Matrix in high-dimensional data}
\begin{aug}
\author{\fnms{Arash} \snm{A. Foroushani } \thanksref{t2,m2}
\ead[label=e2]{aghaeif@uwindsor.ca}}
\and
\author{\fnms{S\'ev\'erien} \snm{Nkurunziza} \thanksref{t2,m2}
\ead[label=e3]{severien@uwindsor.ca}}

\thankstext{t2}{Supported by Natural Sciences and Engineering Research Council of Canada}
\runauthor{Arash A. Foroushani and S\'ev\'erien Nkurunziza}

\affiliation{ University of Windsor, Mathematics and Statistics department\\
	401 Sunset Avenue, Windsor, ON, N9B 3P4\\
	Email:{\color{blue}aghaeif@uwindsor.ca} E-mail:{\color{blue}severien@uwindsor.ca}\thanksmark{m2}}

\end{aug}
\maketitle

\begin{abstract}
   In this paper, we introduce a class of improved estimators for the mean parameter matrix of a multivariate normal 
   distribution with an unknown variance-covariance matrix. In particular, the main results of \cite{ChetelatWells}[{\em Improved Multivariate Normal Mean Estimation with Unknown Covariance when $p$ is Greater than $n$. Annals of Statistics, 2012, {\bf 40}(6), 3137--3160}] are established in their full generalities and we provide the corrected version of their Theorem~2. Specifically, we generalize the existing results in three ways. First, we consider a parameter matrix estimation problem which enclosed as a special case the one about the vector parameter.   
   Second, we propose a class of James-Stein matrix estimators and, we establish a necessary and a sufficient condition for any member of the proposed class to have a finite risk function. Third, we present the conditions for the proposed class of estimators to dominate the maximum likelihood estimator. On the top of these interesting contributions, 
   the additional novelty consists in the fact that, we extend the methods suitable for the vector parameter case and 
   the derived results hold in the classical case as well as in the context of high and ultra-high dimensional data.

\end{abstract}



\end{frontmatter}
\section{Introduction}\label{sec:intro}
Inspired by the work in \cite{ChetelatWells}, we consider an estimation problem concerning the mean matrix of a Gaussian random matrix in the context where the variance-covariance matrix is unknown. In the realm of multivariate statistical analysis, the estimation of the matrix mean plays a pivotal role, especially when dealing with covariance or correlation matrices. Traditional methods of matrix mean estimation have been effective for low-dimensional datasets. However, with the advent of big data and the increasing dimensionality of datasets in fields such as genomics~(see \cite{Tzeng}, \cite{Pardy}), finance~(see \cite{Fan}), and neuro-imaging data analysis~(see \cite{ZhouLiZhu}), these traditional techniques often fall short. High-dimensional data, where the number of variables can exceed the number of observations, presents unique challenges, including increased computational complexity and the curse of dimensionality.

High-dimensional data analysis plays a pivotal role in statistical inference, enabling researchers and analysts to extract valuable insights, build precise models, and draw meaningful conclusions from intricate and extensive datasets. By addressing the distinctive challenges posed by high-dimensional data, where the number of features exceeds the number of observations, this approach equips us with essential tools to gain a comprehensive understanding across diverse real-world applications. These applications span various fields, such as text mining, image and video analysis, genomic, and financial analytic.

The need for improved matrix mean estimation techniques in the context of high-dimensional data has become increasingly evident. Such techniques aim to provide more accurate and computationally efficient estimates, even when faced with the intricacies of high-dimensional spaces. By leveraging advancements in optimization, regularization, and matrix theory, researchers are developing innovative approaches to address the challenges posed by high-dimensional multivariate models. These enhanced methods not only promise better statistical properties but also open the door to more sophisticated analyses in various scientific and technological domains. This exploration delves into the cutting-edge methodologies for improved matrix mean estimation in high-dimensional settings, elucidating their theoretical underpinnings, practical applications, and potential impact on multivariate data analysis.


In this paper, we improve the results in \cite{ChetelatWells} in three ways. First, we propose a class of matrix James-Stein estimators in the context of high-dimensional data. In the special case of parameter vector, the proposed class of estimators yields the class of estimators given in \cite{ChetelatWells}. Second, we derive a necessary and sufficient condition for the members of the proposed class of  James-Stein estimators to have a finite risk. For the special case of vector estimators, the established result provides a corrected version of Theorem~2 of \cite{ChetelatWells}. Third, we establish a sufficient condition for the proposed class of James-Stein estimators to dominate the classical maximum likelihood estimator~(MLE).

The remainder of this paper is structured as follows. In \Cref{Sec:fundamental}, we present the statistical model and the primary results. Namely, we present in this section crucial propositions and lemma which play a vital role in deriving the main results of this paper.  
In  \Cref{Sec:mainresult}, we present the main results of this paper.
In \Cref{sec:simulation}, we present some simulation results which corroborate the main findings of this paper. Further, in \Cref{sec:conclusion},  we present some concluding remarks. Finally, for the convenience of the reader, technical results and some proofs are given in \Cref{sec:append} and \Cref{sec:appendMainres}.



%


   \section{Statistical model and some fundamental results}\label{Sec:fundamental}
In this section, we present the statistical model and set up some notations used in this paper. We also present some fundamental results which are useful in deriving the main results of this paper.  To set up some notations, let $\bm{I}_{q}$ denote $q$-dimensional identity matrix, let $\bm{\Sigma}$ be a positive definite matrix and let $\bm{\theta}$ be a $p\times q$-matrix. For a given $p_{\scriptscriptstyle 1}\times p_{\scriptscriptstyle 2}$-matrix $\bm{A}$, let 
$\textnormal{Vec}(\bm{A})$ stand for a vectorization operator that transforms the matrix into a column vector by vertically stacking the columns of the matrix $\bm{A}$. 
Further, let $\bm{A}\otimes\bm{B}$ denote the Kronecker product of the matrices $\bm{A}$ and $\bm{B}$, let $U\sim \mathcal{N}_{p}\left(\mu, \bm{\Sigma}\right)$ denote a $p$-column Gaussian random vector $U$ with means $\mu$ and covariance $\bm{\Sigma}$ and, for a random matrix $\bm{U}$, let $\bm{U}\sim \mathcal{N}_{p_{\scriptscriptstyle 1}\times p_{\scriptscriptstyle 2}}\left(\bm{\mu},\bm{A}_{1}\otimes \bm{A}_{2}\right)$ to stand for $\textnormal{Vec}(\bm{U})\sim \mathcal{N}_{p_{\scriptscriptstyle 1} p_{\scriptscriptstyle 2}}\left(\textnormal{Vec}\left(\bm{\mu}\right),\bm{A}_{1}\otimes \bm{A}_{2}\right)$. We also denote
$\bm{U}\sim \mathcal{W}_{\scriptscriptstyle p}\left(m,\bm{\Sigma}\right)$ to stand for $p\times p$-random matrix that follows Wishart distribution with degrees of freedom $m$ and non-centrality parameter $\bm{\Sigma}$.
For $m \times n$ matrices $\bm{A}$ and $\bm{B}$, let 
$\nabla_{\bm{A}}=\left(\frac{\partial}{\partial A_{ij}}\right)_{1\leqslant i \leqslant m,1 \leqslant j \leqslant n}$, define
\begin{eqnarray*}
\textnormal{div}_{\bm{A}}\bm{B}=\nabla_{\bm{A}}.\bm{B}=\textnormal{div}_{\textnormal{vec}(\bm{A})}\textnormal{vec}(\bm{B})=\sum_{i,j}^{}{\frac{\partial B_{ij}}{\partial A_{ij}}},
\mbox{ and }
\left(\nabla_{\bm{A}}\bm{B}\right)_{ij}=\sum_{\alpha}^{}{\left(\nabla_{\bm{A}}\right)_{i \alpha}B_{\alpha j}}=\sum_{\alpha}^{}{\frac{\partial B_{\alpha j}}{\partial A_{i \alpha}}}.
\end{eqnarray*}
In the sequel, despite the fact that the target parameter is a matrix, we generally use the same notations/symbols as in \cite{ChetelatWells}. In particular, we consider the scenario where we observe a $p\times q$-random matrix $\bm{X}\sim \mathcal{N}_{p \times q}(\bm{\theta}, \bm{\Sigma} \otimes \bm{I}_{q})$ with $\bm{\Sigma}$ a positive definite matrix. 

%
Our objective is to estimate the mean matrix $\bm{\theta}$ in the context of high or ultra-high-dimensional data with $\bm{\Sigma}$ unknown nuisance parameter. We also consider that the random matrix $\bm{S}$ is observed along with the matrix $\bm{X}$, and we consider the scenario where $\bm{X}$ and $\bm{S}$ are independent in context  where $\bm{S}\sim \mathcal{W}_{p}\left(nq,\, \bm{\Sigma}/q\right)$, with $p>nq$. To better connect to the data collection, we consider the case where $N$ observations of $p\times q$-matrices $\bm{W}^{(1)}$, $\bm{W}^{(2)}$, \dots, $\bm{W}^{(N)}$ which are independent and identically distributed as $\mathcal{N}_{p\times q}\left(\bm{\theta},\bm{\Xi}\otimes \bm{I}_{q}\right)$. For $i=1,2,\dots, N$, let $\bm{W}^{(i)}=\left(W^{(i)}_{lk}\right)_{1\leqslant l\leqslant p, 1\leqslant k\leqslant q}$. 
Under this sampling plan, the maximum likelihood estimators (MLE) of $\bm{\theta}$ and $\bm{\Sigma}$ are respectively given by 
\begin{eqnarray}
\bm{X}\equiv\bar{W}=\ds{\frac{1}{N}\sum_{i=1}^{N}}\bm{W}^{(i)}, \quad{ }  \bm{S}\equiv\ds{\frac{1}{Nq}}\ds{\sum_{i=1}^{N}}\left(\bm{W}^{(i)}-\bar{\bm{W}}\right)\left(\bm{W}^{(i)}-\bar{\bm{W}}\right)'.\label{XbarS}
\end{eqnarray}
To rewrite $\bm{S}$ under the form of a simple quadratic form, let $\bm{W}=\left(\bm{W}^{(1)}\vdots \bm{W}^{(2)}\vdots \dots\vdots \bm{W}^{(N)}\right)'$ and let $\bm{e}_{N}$ be $N$-column vector with all components equal to 1. From these notations, the sample variance-covariance can be rewritten as
\begin{eqnarray}
Nq\bm{S}=\bm{W}'\left[\left(\bm{I}_{N}-\frac{1}{N}\bm{e}_{N}\bm{e}'_{N}\right)\otimes \bm{I}_{q}\right]\bm{W}, \label{S}
\end{eqnarray}
and then, using the properties of multivariate normal distributions along with the fact that $\left(\bm{I}_{N}-\frac{1}{N}\bm{e}_{N}\bm{e}'_{N}\right)\otimes \bm{I}_{q}$ is idempotent matrix with the rank $(N-1)q$, one can verify that
$Nq\bm{S}\sim \mathcal{W}_{p}\left((N-1)q,\, \bm{\Xi}\right)$ or equivalently $\bm{S}\sim \mathcal{W}_{p}\left((N-1)q,\, \frac{\bm{\Xi}}{Nq}\right)$. Further, there exists an orthogonal matrix $\bm{Q}$ such that
\begin{eqnarray*}
\bm{Q}'\left[\left(\bm{I}_{N}-\frac{1}{N}\bm{e}_{N}\bm{e}'_{N}\right)\otimes \bm{I}_{q}\right]
\bm{Q}=\left(
         \begin{array}{cc}
           \bm{I}_{N-1} & \bm{0} \\
           \bm{0} & 0 \\
         \end{array}
       \right)\otimes \bm{I}_{q}.
\end{eqnarray*}
Let $\bm{V}=\bm{Q}'\bm{W}$. We have $\bm{V}\sim \mathcal{N}_{Nq\times p}\left(\bm{Q}'\left(\bm{e}_{N}\otimes \bm{\theta}'\right),\,\bm{I}_{Nq}\otimes \bm{\Xi}\right)$ and
\begin{eqnarray}
Nq\,\bm{S}=\bm{V}'\left(
         \begin{array}{cc}
           \bm{I}_{(N-1)q} & \bm{0} \\
           \bm{0} & \bm{0} \\
         \end{array}
       \right)\bm{V}=\bm{V}'_{(1)}\bm{V}_{(1)},\label{S_newform}
 \end{eqnarray}
  with  \quad{ } $\bm{V}_{(1)}=\left[\bm{I}_{q(N-1)}\vdots \bm{0}\right]\bm{V}\sim \mathcal{N}_{(N-1)q\times p}\left(\bm{0},\quad{ } \bm{I}_{(N-1)q}\otimes \bm{\Xi}\right)$.
In the similar way as in \cite{ChetelatWells}, in the sequel, we let $n=N-1$ and $\bm{\Sigma}=N^{-1}\bm{\Xi}$. We also let $\bm{Y}=\bm{V}_{(1)}/\sqrt{Nq}$, i.e. we let $\bm{S}=\bm{Y}'\bm{Y}$. In classical inference where $n\geqslant pq$, the problem studied can be solved by using the results in \cite{SteinCharles}, \cite{Bilodeau}, \cite{Konno1990}, \cite{Konno} and references therein. To give another closely related reference, in the case where $\bm{\theta}$ is a vector with $n\geqslant p$, we also quote \cite{Perron} in the context of a bounded normal mean and \cite{STRAWDERMAN} in the context of elliptically symmetric distributions. We also quote \cite{Bodnar} who studied the similar estimation problem of $\bm{\theta}$ in the case where $q=1$ in the large-samples context.

 Nevertheless, more research needs to be done for the cases of high and ultra-high dimensional data i.e. the case where the total number of features $pq>n$ (or $pq>>n$).  The problem studied here is more general that the one where the total number of features $pq$ is bigger than the $n$ which does not need to tend to infinity as considered in \cite{Bodnar}. In this paper, not only we solve the problem in the context of high and ultra-high dimensional data, we push the boundaries further by considering the more general case of ultra-high dimensional data. Specifically, we consider the scenario where $p>nq$. Indeed, if $p>nq$, as this holds in our case, then $pq>nq^{2}>n$ whenever $q>1$. 
 With \eqref{XbarS} and \eqref{S_newform} in mind, to simplify the derivation of the main results of this paper, we consider the statistical model that satisfies the following assumption.
 \begin{assumption}\label{ass:model}
   We observe  two independent $p\times q$ random matrices $\bm{X}$ and $\bm{S}$ with $\bm{X}\sim \mathcal{N}_{p\times q}\left(\bm{\theta},\bm{\Sigma}\otimes \bm{I}_{q}\right)$ and $\bm{S}=\bm{Y}'\bm{Y}$ where $\bm{Y}\sim \mathcal{N}_{nq\times p}\left(\bm{0},\bm{I}_{nq}\otimes(\bm{\Sigma}/q)\right)$ with $\bm{\Sigma}$ a $p\times p$-symmetric and positive definite matrix.
 \end{assumption}
 Note that for the special case where $q=1$, the problem becomes the one studied in \cite{ChetelatWells}. Nevertheless, there is a major mistake in one of their main result. Thus, on the top of generalizing the problem studied in \cite{ChetelatWells}, we also revise their main result.  One of the main difficulty consists in the fact that, in cases where $p>nq$, the estimator of $\bm{\Sigma}$, denoted as $\bm{S}$, is singular  almost surely. Since $\bm{S}$ is singular almost surely, the traditional inverse $\bm{S}^{-1}$ does not exist, with probability one,  while for the cases where $pq\leqslant n$, such inverse of the  matrix $\bm{S}$ is used in James-Stein type estimators. Because of that, we need to adopt whenever needed the generalized inverse also known as Moore-Penrose inverse. Thus, let $A^{+}$ denote the Moore-Penrose inverse of the matrix $\bm{A}$. We also let 
  \begin{eqnarray}
   F=\mathrm{tr}(\bm{X}^{\top}\bm{S}^{+}\bm{X}), \quad{ } R=\textnormal{rank}(\bm{S}), \quad{ }  \bm{\delta}^{0}(\bm{X})=\bm{X}.\label{FRdelta0}
   \end{eqnarray}
   Further, for some real-valued function $r$ defined on $(0,\,\infty)$, let 
\begin{eqnarray}
\delta_{r}(\bm{X},\bm{S}) = \left(\bm{I}_{p} - \frac{r(F)}{F}\bm{S}\bm{S}^{+}\right)\bm{X}. \label{improvedEst}
\end{eqnarray}
Note that the classical MLE $\bm{\delta}^{0}(\bm{X})=\bm{X}=\delta_{0}(\bm{X},\bm{S})$. Further,
note that for the special case where $q=1$, the estimator in \eqref{improvedEst} yields the one given in
\cite{ChetelatWells}. In this paper, we derive the necessary and sufficient conditions for the estimator in \eqref{improvedEst}  to have a finite quadratic risk function.  As mentioned above, beyond the generalisation of the methods given in \cite{ChetelatWells}, the additional novelty consists in the fact that we also revise some main results in the above quoted paper. To this end, we also derive some important mathematical results which have interest in their own. Nevertheless, we are very pleased to acknowledge that our work was strongly
inspired by that of \cite{ChetelatWells}.  

   \begin{proposition}\label{Proposition2}
Suppose that \Cref{ass:model} holds and let $g(\bm{X},\bm{S})$ be a $p\times q$-matrix valued function defined on $\mathbb{R}^{p\times q}\times \mathbb{R}^{p\times p}$. Then\\
    $(i)$ 
    $\mathrm{E}_{\scriptscriptstyle \bm{\theta}}\biggl[\mathrm{tr}(g^{\top}(\bm{X},\bm{S})\bm{\Sigma}^{-1}(\bm{X}-\bm{\theta}))\biggr]=\mathrm{E}_{\scriptscriptstyle \bm{\theta}}
    \biggl[\mathrm{tr}(\nabla_Xg^{\top}(\bm{X},\bm{S}))\biggr]$, 
    \text{provided that $\mathrm{E}_{\scriptscriptstyle \bm{\theta}}\bigl[|\mathrm{tr}(\nabla_Xg^{\top}(\bm{X},\bm{S}))|\bigr]<\infty$;}\\
        $(ii)$ 
        $\mathrm{tr}\biggl(\nabla_Xg^{\top}(\bm{X},\bm{S})\biggr)=\ds\sum_{i,j}^{}{\frac{\partial g_{ij}(\bm{X},\bm{S})}{\partial {X}_{ij}}}$. 
\end{proposition}
The proof of this proposition is given in \Cref{sec:append}. By using \Cref{Proposition2}, we derive the following result which is very useful in deriving the main result of this paper.
   \begin{proposition}\label{Proposition1}
Suppose that \Cref{ass:model} holds along with $\textnormal{P}_{\scriptscriptstyle \bm{\theta}}(qR>2)=1$ and \eqref{FRdelta0}. Further, let $g(\bm{X},\bm{S})=r(F)\bm{S}\bm{S}^{+}\bm{X}/F$. 
Then
    \begin{flalign*}
        (i)  \quad  \mathrm{tr}(\nabla_{\bm{X}}g(\bm{X},\bm{S})^{\top})=2r^{\prime}(F)+\left[q\mathrm{tr}(\bm{S}\bm{S}^{+})-2\right]r(F)/F;&&
    \end{flalign*}
\begin{flalign*}
    (ii) \quad  \mathrm{E}_{\scriptscriptstyle \bm{\theta}}\left[\mathrm{tr}\left(g^{\top}(\bm{X},\bm{S})\bm{\Sigma}^{-1}(\bm{X}-\bm{\theta})\right)\right]
    =\mathrm{E}_{\scriptscriptstyle \bm{\theta}}\left[2r^{\prime}(F)+\left(q\mathrm{tr}(\bm{S}\bm{S}^{+})-2\right)r(F)/F\right].&&
\end{flalign*}
\end{proposition}
The proof of this result is given in the \Cref{sec:append}. It should be noticed that, for the special case where $q=1$, Part~{\em (i)} of \Cref{Proposition1} yields Lemma~2 of \cite{ChetelatWells}.
   \section{Main results}\label{Sec:mainresult}
   In this section, we present the main result of this paper. In particular, we derive a theorem which demonstrates that, for the special case where $q=1$, if the additional assumption $\textnormal{P}_{\scriptscriptstyle \bm{\theta}}(R>2)=1$ is not taken into account, Theorem~2 of \cite{ChetelatWells}  does not hold. To illustrate this point, we prsent a 
   counterexample that contradicts  Theorem~2 of \cite{ChetelatWells}. 
Moreover, we derive a result which revise Theorem~2 of \cite{ChetelatWells}. On the top of that, we generalize the corrected version  
%
of Theorem 2 in \cite{ChetelatWells}. 
   \begin{lemma}\label{Lemma3}
    Let  $g(\bm{X},\bm{S})= (\bm{S}\bm{S}^{+}\bm{X})r(F)\big/F$ and suppose that \Cref{ass:model} holds along with \eqref{FRdelta0}. Then

    \begin{flalign*}
        (i) & \quad \frac{\partial F}{\partial \bm{X}} = 2(\bm{S}^{+}\bm{X}); &&
    \end{flalign*}
    \begin{flalign*}
        (ii) & \quad \left(\frac{\partial \bm{S}\bm{S}^{+}\bm{X}}{\partial X_{ij}}\right)_{kl} = (\bm{S}\bm{S}^{+})_{ki}\delta_{lj}; &&
    \end{flalign*}
    \begin{flalign*}
        (iii) & \quad \frac{\partial g_{kl}(\bm{X},\bm{S})}{\partial X_{ij}} = \frac{2(Fr^{\prime}(F)-r(F))}{F^2}(\bm{S}^{+}\bm{X})_{ij}(\bm{S}\bm{S}^{+}\bm{X})_{kl} + \frac{r(F)}{F}(\bm{S}\bm{S}^{+})_{ki}\delta_{lj}; &&
    \end{flalign*}
    \begin{flalign*}
        (iv) & \quad \sum_{i,j} \frac{\partial g_{ij}(\bm{X},\bm{S})}{\partial X_{ij}} = 2r^{\prime}(F) + \frac{r(F)}{F}\left[q\mathrm{tr}(\bm{S}\bm{S}^{+}) - 2\right]. &&
    \end{flalign*}
\end{lemma}
The proof of this lemma is given in the \Cref{sec:appendMainres}. We also derive the following lemma which is crucial in revising Theorem~2 of \cite{ChetelatWells}.

\begin{lemma}\label{Lemma4}
Suppose that \Cref{ass:model} holds along with \eqref{FRdelta0}. 
Then,
    $\mathrm{E}_{\scriptscriptstyle \bm{\theta}}\left[1/F\right] < \infty$ \text{ if and only if } $\textnormal{P}_{\scriptscriptstyle \bm{\theta}}(qR > 2) = 1$.
    \end{lemma}
    \begin{proof}
        Assume that $\textnormal{P}_{\scriptscriptstyle \bm{\theta}}(qR>2)=1$ and let $\bm{C}(R)$ be $R \times p$-matrix of the form $\bm{C}(R)=[\bm{I}_{R}\vdots0_{R\times (p-R)}]$ and let $\bm{X}_{(1)} = \bm{C}(R)\bm{X}$.   Let $\lambda^{\dag}_{\min}(\bm{S}^{+})$ and $\lambda_{\max}^{\dag}(\bm{S}^{+})$ be the smallest and biggest nonzero eigenvalues of $\bm{S}^{+}$ respectively. Since $\bm{S}^{+}$ is semi-positive definite,  
        \begin{flalign}
           \lambda_{\min}^{\dag}(\bm{S}^{+}){ \textnormal{tr}(\bm{X}^{\top}_{(1)}\bm{X}_{(1)})}\leqslant { \textnormal{tr}(\bm{X}^{\top}\bm{S}^{+}\bm{X})} \leqslant \lambda_{\max}^{\dag}(\bm{S}^{+}){ \textnormal{tr}(\bm{X}^{\top}_{(1)}\bm{X}_{(1)})}.\label{ineq_Fmain1}
       \end{flalign}
Since $\bm{X}\sim \mathcal{N}_{p\times q}(\bm{\theta},\bm{\Sigma}\otimes \bm{I}_q)$, we have $$\bm{X}_{(1)}=\bm{C}(R)\bm{X}\big|R\sim \mathcal{N}_{R\times q}(\bm{C}(R)\bm{\theta},(\bm{C}(R)\bm{\Sigma}\bm{C}^{\top}(R))\otimes \bm{I}_q).$$ Note that, since $\bm{\Sigma}$ is positive definite matrix, $\bm{C}(R)\bm{\Sigma}\bm{C}^{\top}(R)$ is positive definite matrix with probability one. Let $\bm{A}(R)=\left(\bm{C}(R)\bm{\Sigma}\bm{C}^{\top}(R)\right)^{1/2}$  and
let $\bm{U}=\bm{A}^{-1}(R)\bm{X}_{(1)}$. Then,
\begin{flalign}
   \lambda_{\min}(\bm{A}^{2}(R))\textnormal{tr}(\bm{U}^{\top}\bm{U})\leqslant \textnormal{tr}(\bm{U}^{\top}\bm{A}^{2}(R)\bm{U}) \leqslant \textnormal{tr}(\bm{U}^{\top}\bm{U}) \lambda_{\max}(\bm{A}^{2}(R)),\label{ineq_Fmain2}
\end{flalign}
where $ \lambda_{\min}\left(\bm{A}^{2}(R)\right)$ and $\lambda_{\max}\left(\bm{A}^{2}(R)\right)$ are the smallest and biggest eigenvalues of the positive definite matrix $\bm{A}(R)$ respectively.
       Therefore, by \eqref{ineq_Fmain1} and \eqref{ineq_Fmain2}, we get
        \begin{eqnarray}
           \frac{1}{F}\leqslant \frac{1}{\lambda^{\dag}_{\min}(\bm{S}^{+})\lambda_{\min}(\bm{A}^{2}(R))\textnormal{tr}(\bm{U}^{\top}\bm{U})}
           =\frac{\lambda_{\max}^{\dag}(\bm{S})\lambda_{\max}(\bm{A}^{-2}(R))}{\textnormal{tr}(\bm{U}^{\top}\bm{U})}
           =\frac{\lambda_{\max}^{\dag}(\bm{S})\lambda_{\max}(\bm{A}^{-2}(R))}{\textnormal{vec}^{\top}(\bm{U})\textnormal{vec}(\bm{U})}\label{ineqF}
       \end{eqnarray}
       where $\lambda_{\max}^{\dag}(\bm{S})$ and $\lambda_{\max}\left(\bm{A}^{-2}(R)\right)$ are the biggest nonzero eigenvalues of $(\bm{S}^{+})^{+}=\bm{S}$ and $(\bm{A}^{2}(R))^{+}=\bm{A}^{-2}(R)$ respectively.
We also have
       \begin{flalign*}
           \textnormal{vec}(\bm{U})=\textnormal{vec}(\bm{A}^{-1}(R)\bm{X}_{(1)})=(\bm{I}_{q}\otimes \bm{A}^{-1}(R))\textnormal{vec}(\bm{X}_{(1)}).
       \end{flalign*}
       Then,
       \begin{flalign*}
           \textnormal{vec}(\bm{U})\big|R \sim \mathcal{N}_{qR}\left(\left(\bm{I}_{q}\otimes \bm{A}^{-1}(R)\right)\textnormal{vec}(\bm{C}(R)\bm{\theta}),\,(\bm{I}_{q}\otimes \bm{A}^{-1}(R))(\bm{I}_{q}\otimes \bm{A}^{2}(R))(\bm{I}_{q}\otimes \bm{A}^{-1}(R))^{\top}\right)
       \end{flalign*}
       Therefore,
       \begin{flalign}
             \textnormal{vec}(\bm{U})\big|R \sim \mathcal{N}_{qR}\left((\bm{I}_{q}\otimes \bm{A}^{-1}(R))\textnormal{vec}(\bm{C}(R)\bm{\theta}),\bm{I}_{qR}\right).\label{dvecU}
       \end{flalign}
       Note that $\lambda_{\max}^{\dag}(\bm{S})$ depends on $\bm{S}$, $\lambda_{\max}(\bm{A}^{-2}(R))$ depends on $R$, and $\bm{U}$ depends on $R$ and $\bm{X}$. Since, conditionally to $R$, $\bm{S}$ and $\bm{X}$ are independent we get
       \begin{flalign*}
         \textnormal{E}_{\scriptscriptstyle \bm{\theta}}\left[\frac{\lambda_{\max}^{\dag}(\bm{S})\lambda_{\max}(\bm{A}^{-2}(R))}
         {\textnormal{vec}^{\top}(\bm{U})\textnormal{vec}(\bm{U})}\right] &= \textnormal{E}_{\scriptscriptstyle \bm{\theta}}\left[\textnormal{E}_{\scriptscriptstyle \bm{\theta}}\left[\frac{\lambda_{\max}^{\dag}(\bm{S})\lambda_{\max}(\bm{A}^{-2}(R))}{\textnormal{vec}^{\top}(\bm{U})
         \textnormal{vec}(\bm{U})}\big|R\right]\right]\\&=\textnormal{E}_{\scriptscriptstyle \bm{\theta}}\left[\lambda_{\max}(\bm{A}^{-2}(R))\textnormal{E}_{\scriptscriptstyle \bm{\theta}}\left[\lambda_{\max}^{\dag}(\bm{S})\big|R\right]
         \textnormal{E}_{\scriptscriptstyle \bm{\theta}}\left[\frac{1}{\textnormal{vec}^{\top}(\bm{U})\textnormal{vec}(\bm{U})}\big|R\right]\right].&&
       \end{flalign*}
        Further, let $\lambda_{\max}(\bm{\Sigma})$ be the biggest eigenvalue of $\bm{\Sigma}$. Then,
       \begin{flalign*}
           \lambda_{\max}^{\dag}(\bm{S}) \leqslant \textnormal{tr}(\bm{S})=\textnormal{tr}(\bm{\Sigma}^{\frac{1}{2}}\bm{\Sigma}^{-\frac{1}{2}}\bm{S}\bm{\Sigma}^{-\frac{1}{2}}
           \bm{\Sigma}^{\frac{1}{2}})
           \leqslant &\lambda_{\max}(\bm{\Sigma})\textnormal{tr}((\bm{Y}\bm{\Sigma}^{-\frac{1}{2}})^{\top}\bm{Y}\bm{\Sigma}^{-\frac{1}{2}})\\&
           =\lambda_{\max}(\bm{\Sigma})\textnormal{vec}^{\top}(\bm{Y}\bm{\Sigma}^{-\frac{1}{2}})\textnormal{vec}(\bm{Y}\bm{\Sigma}^{-\frac{1}{2}}),&&
       \end{flalign*}
       where $\textnormal{vec}(\bm{Y}\bm{\Sigma}^{-\frac{1}{2}}) \sim \mathcal{N}_{npq}(0,(\bm{I}_{p}/q) \otimes \bm{I}_{nq})$. Therefore,  we get
       \begin{flalign*}
\textnormal{E}_{\scriptscriptstyle \bm{\theta}}[\lambda_{\max}^{\dag}(\bm{S})] \leqslant& \lambda_{\max}(\bm{\Sigma})\textnormal{E}_{\scriptscriptstyle \bm{\theta}}[\textnormal{vec}^{\top}(\bm{Y}\bm{\Sigma}^{-\frac{1}{2}})\textnormal{vec}(\bm{Y}\bm{\Sigma}^{-\frac{1}{2}})]
           \\&=\lambda_{\max}(\bm{\Sigma})\textnormal{tr}(\bm{I}_{p} \otimes \bm{I}_{nq})/q=\lambda_{\max}(\bm{\Sigma})\textnormal{tr}(I_{npq})/q=\lambda_{\max}(\bm{\Sigma})np. 
       \end{flalign*}
       Hence
       \begin{eqnarray}
        \textnormal{E}_{\scriptscriptstyle \bm{\theta}}[\lambda_{\max}^{\dag}(\bm{S})] \leqslant \lambda_{\max}(\bm{\Sigma})np. \label{boundLambdamax}
        \end{eqnarray}
Since $\textnormal{vec}(\bm{U})\big|R \sim \mathcal{N}_{qR}\left((\bm{I}_{q}\otimes \bm{A}^{-1}(R))\textnormal{vec}(\bm{C}(R)\bm{\theta}),\bm{I}_{qR}\right)$, we get
       \begin{eqnarray}
           \textnormal{vec}^{\top}(\bm{U})\textnormal{vec}(\bm{U})\big|R\sim \chi^{2}_{qR}(\delta_R)\label{dvecUvecU}
       \end{eqnarray}
       where $\delta_R = \left((\bm{I}_{q}\otimes \bm{A}^{-1}(R))\textnormal{vec}(\bm{C}(R)\bm{\theta})\right)^{\top}\left((\bm{I}_{q}\otimes \bm{A}^{-1}(R))\textnormal{vec}(\bm{C}(R)\bm{\theta})\right)$. 
       Let $Z$ be a random variable such that $Z\big|R\sim \textnormal{Poisson}(\delta_R/2)$. By \eqref{dvecUvecU} We have
       \begin{flalign*}
            &\textnormal{E}_{\scriptscriptstyle \bm{\theta}}\left[(\textnormal{vec}^{\top}(\bm{U})\textnormal{vec}(\bm{U}))^{-1}\big|R\right]=
           \textnormal{E}_{\scriptscriptstyle \bm{\theta}}\left[\textnormal{E}_{\scriptscriptstyle \bm{\theta}}[(\textnormal{vec}^{\top}(\bm{U})\textnormal{vec}(\bm{U}))^{-1}\big|R,Z]\big|R\right]
           = \textnormal{E}_{\scriptscriptstyle \bm{\theta}}\left[\textnormal{E}_{\scriptscriptstyle \bm{\theta}}[(\chi^{2}_{qR+2Z})^{-1}\big|R,Z]\big|R\right]. 
       \end{flalign*}
       Further, since $\textnormal{P}_{\scriptscriptstyle \bm{\theta}}(qR>2)=\textnormal{P}_{\scriptscriptstyle \bm{\theta}}(qR\geq 3)=1$ and $\textnormal{P}_{\scriptscriptstyle \bm{\theta}}\left(Z\geq 0\right)=1$, then 
       $qR+2Z-2\geqslant1$ with probability one and then,
       \begin{flalign*}
            &\textnormal{E}_{\scriptscriptstyle \bm{\theta}}\left[(\textnormal{vec}^{\top}(\bm{U})\textnormal{vec}(\bm{U}))^{-1}\big|R\right]
           =\textnormal{E}_{\scriptscriptstyle \bm{\theta}}\left[\frac{2^{-1}\Gamma\left(\frac{qR+2Z}{2}-1\right)}{\Gamma\left(\frac{qR+2Z}{2}\right)}\Big|R\right].
       \end{flalign*}
       and then, since $qR+2Z-2\geq 1$ with probability one,
        \begin{eqnarray*}
           \textnormal{E}_{\scriptscriptstyle \bm{\theta}}[(\textnormal{vec}^{\top}(\bm{U})\textnormal{vec}(\bm{U}))^{-1}\big |R]
           =\textnormal{E}_{\scriptscriptstyle \bm{\theta}}\left[\ds{\frac{1}{Rq+2Z-2}}\big|R\right] \leqslant 1,
       \end{eqnarray*}
       almost surely. Therefore, together with \eqref{ineqF}, we get
    \begin{flalign*}
       \textnormal{E}_{\scriptscriptstyle \bm{\theta}}\left[\frac{1}{F}\right]\leqslant \textnormal{E}_{\scriptscriptstyle \bm{\theta}}\left[\lambda_{\max}(\bm{A}^{-2}(R))\textnormal{E}_{\scriptscriptstyle \bm{\theta}}\left[\lambda_{\max}^{\dag}(\bm{S})\big|R\right]\textnormal{E}_{\scriptscriptstyle \bm{\theta}}
       \left[\frac{1}{\textnormal{vec}^{\top}(\bm{U})\textnormal{vec}(\bm{U})}\big|R\right]\right] \leqslant &\textnormal{E}_{\scriptscriptstyle \bm{\theta}}\left[\lambda_{\max}(\bm{A}^{-2}(R))\textnormal{E}_{\scriptscriptstyle \bm{\theta}}\left[\lambda_{\max}^{\dag}(\bm{S})\big|R\right]\right]
       .&&
    \end{flalign*}
    This gives
         $\ds\textnormal{E}_{\scriptscriptstyle \bm{\theta}}\left[1/F\right]\leqslant
         \textnormal{E}_{\scriptscriptstyle \bm{\theta}}\left[\textnormal{tr}(\bm{A}^{-2}(R))\textnormal{E}_{\scriptscriptstyle \bm{\theta}}\left[\textnormal{tr}(\bm{S})\big|R\right]\right]\leqslant np\lambda_{\max}(\bm{\Sigma})\textnormal{E}_{\scriptscriptstyle \bm{\theta}}\left[\textnormal{tr}(\bm{A}^{-2}(R))\right]$.
    Note that $ \textnormal{P}_{\scriptscriptstyle \bm{\theta}}\left(3\leqslant qR\leqslant qp\right)=1$ and then
    \begin{eqnarray}
    \textnormal{E}_{\scriptscriptstyle \bm{\theta}}\left[\textnormal{tr}(\bm{A}^{-2}(R))\right]=\sum_{j=\max\{\lfloor 3/q\rfloor,1\}}^{p}\textnormal{tr}(\bm{A}^{-2}(j))
    \textnormal{P}_{\scriptscriptstyle \bm{\theta}}\left(R=j\right)\leqslant \sum_{j=\max\{\lfloor 3/q\rfloor,1\}}^{p}\textnormal{tr}(\bm{A}^{-2}(j))\label{traceofinverseofA}
    \end{eqnarray}
    where $\lfloor x\rfloor$ denotes the largest integer less than or equal to $x$, for $j=\max\{\lfloor 3/q\rfloor,1\},\max\{\lfloor 3/q\rfloor,1\}+1,\dots,p-1$, $\bm{A}^{2}(j)=[\bm{I}_{j}\vdots0_{j\times (p-j)}]\bm{\Sigma}[\bm{I}_{j}\vdots0_{j\times (p-j)}]^{\top}$ and, we set $\bm{A}^{2}(p)=\bm{\Sigma}$. Thus, for $j=\max\{\lfloor 3/q\rfloor,1\},\max\{\lfloor 3/q\rfloor,1\}+1,\dots,p$, $\bm{A}^{2}(j)$ is positive definite matrix and then, $0<\textnormal{tr}(\bm{A}^{-2}(j))<\infty$. Hence,
    $0<\ds\sum_{j=\max\{\lfloor 3/q\rfloor,1\}}^{p}\textnormal{tr}(\bm{A}^{-2}(j))<+\infty$. 
Therefore, together with \eqref{boundLambdamax} and \eqref{traceofinverseofA}, we get
    \begin{flalign*}
        \textnormal{E}_{\scriptscriptstyle \bm{\theta}}\left[1/F\right]\leqslant \textnormal{E}_{\scriptscriptstyle \bm{\theta}}\left[\lambda_{\max}^{\dag}(\bm{S})\lambda_{\max}(\bm{A}^{-2}(R))\right]
        \leqslant np\lambda_{\max}(\bm{\Sigma})\sum_{j=\max\{\lfloor 3/q\rfloor,1\}}^{p}\textnormal{tr}(\bm{A}^{-2}(j))<\infty.
    \end{flalign*}
    Hence,
    $\textnormal{E}_{\scriptscriptstyle \bm{\theta}}\left[1/F\right]\leqslant np\lambda_{\max}(\bm{\Sigma})\ds\sum_{j=\max\{\lfloor 3/q\rfloor,1\}}^{p}\textnormal{tr}(\bm{A}^{-2}(j)) < \infty.$\\
   Conversely, assume that $\textnormal{E}_{\scriptscriptstyle \bm{\theta}}[1/F]<\infty$. From \eqref{ineq_Fmain1} and \eqref{ineq_Fmain2}, we have
    \begin{flalign*}
        \frac{1}{\lambda_{\max}^{\dag}(\bm{S}^{+})\lambda_{\max}(\bm{A}^{-2}(R))\textnormal{vec}^{\top}(\bm{U})\textnormal{vec}(\bm{U})}
        =\frac{\lambda_{\min}^{\dag}(\bm{S})\lambda_{\min}(\bm{A}^{-2}(R))}{\textnormal{vec}^{\top}(\bm{U})\textnormal{vec}(\bm{U})}\leqslant \frac{1}{F},
    \end{flalign*}
    where $\lambda_{\min}^{\dag}(\bm{S})$ is the smallest nonzero eigenvalue of $\bm{S}$. Again, note that $\lambda_{\min}^{\dag}(\bm{S})$ depends on $\bm{S}$, $\lambda_{\min}(\bm{A}^{-2}(R))$ depends on $R$ only,  and $\bm{U}$ depends on $R$ and $\bm{X}$. Then, since, conditionally to $R$, $\bm{S}$ and $\bm{X}$ are independent, we get
           \begin{flalign*}
         \textnormal{E}_{\scriptscriptstyle \bm{\theta}}&\left[\frac{\lambda_{\min}^{\dag}(\bm{S})\lambda_{\min}(\bm{A}^{-2}(R))}{\textnormal{vec}^{\top}(\bm{U})\textnormal{vec}(\bm{U})}\right] = \textnormal{E}_{\scriptscriptstyle \bm{\theta}}\left[\textnormal{E}_{\scriptscriptstyle \bm{\theta}}\left[\frac{\lambda_{\min}^{\dag}(\bm{S})\lambda_{\min}(\bm{A}^{-2}(R))}{\textnormal{vec}^{\top}(\bm{U})
         \textnormal{vec}(\bm{U})}\big|R\right]\right]\\
         &=\textnormal{E}_{\scriptscriptstyle \bm{\theta}}\left[\lambda_{\max}(\bm{A}^{-2}(R))\textnormal{E}_{\scriptscriptstyle \bm{\theta}}\left[\lambda_{\min}^{\dag}(\bm{S})\big|R\right]
         \textnormal{E}_{\scriptscriptstyle \bm{\theta}}\left[\frac{1}{\textnormal{vec}^{\top}(\bm{U})\textnormal{vec}(\bm{U})}\big|R\right]\right]\leqslant \textnormal{E}_{\scriptscriptstyle \bm{\theta}}\left[\frac{1}{F}\right]<\infty. 
       \end{flalign*}
       Then,
       \begin{flalign}
    \textnormal{P}_{\scriptscriptstyle \bm{\theta}}\left(\lambda_{\max}(\bm{A}^{-2}(R))\textnormal{E}_{\scriptscriptstyle \bm{\theta}}\left[\lambda_{\min}^{\dag}(\bm{S})\big|R\right]\textnormal{E}_{\scriptscriptstyle \bm{\theta}}\left[\frac{1}
    {\textnormal{vec}^{\top}(\bm{U})\textnormal{vec}(\bm{U})}\big|R\right]<\infty \right)=1.\label{ineqR1}
       \end{flalign}
       Since $\textnormal{E}_{\scriptscriptstyle \bm{\theta}}\left[\lambda_{\max}^{\dag}(\bm{S})\right]\leqslant np\lambda_{\max}(\bm{\Sigma})<\infty$, we get $0<\textnormal{E}_{\scriptscriptstyle \bm{\theta}}\left[\lambda_{\min}^{\dag}(\bm{S})\big|R\right]\leqslant\textnormal{E}_{\scriptscriptstyle \bm{\theta}}\left[\lambda_{\max}^{\dag}(\bm{S})\big|R\right]<\infty.$ Further, from \eqref{traceofinverseofA}, $\textnormal{E}_{\scriptscriptstyle \bm{\theta}}\left(\lambda_{\max}(\bm{A}^{-2}(R))\right)<\infty$.
       Hence, together with \eqref{ineqR1}, we get
       \begin{flalign*}
           \textnormal{P}_{\scriptscriptstyle \bm{\theta}}\left(\textnormal{E}_{\scriptscriptstyle \bm{\theta}}\left[\frac{1}{\textnormal{vec}^{\top}(\bm{U})\textnormal{vec}(\bm{U})}\Big|R\right]<\infty \right)=1.
       \end{flalign*}
      Therefore,
       together with \eqref{dvecUvecU}, we get
           $\textnormal{P}_{\scriptscriptstyle \bm{\theta}}\left(\textnormal{E}_{\scriptscriptstyle \bm{\theta}}\left[(\chi^{2}_{qR}(\delta_R))^{-1}\big|R\right]<\infty \right)=1$ 
           this implies that $\textnormal{P}_{\scriptscriptstyle \bm{\theta}}\left(qR>2 \right)=1$, 
       which completes the proof.
    \end{proof}
It should be noted that although Theorem~2 of \cite{ChetelatWells} does not require the condition that $\textnormal{P}_{\scriptscriptstyle \bm{\theta}}(R>2)=1$, the main step of their proof consists in proving that $\textnormal{E}_{\scriptscriptstyle \bm{\theta}}\left[1/F\right]<\infty$. However, the result of \Cref{Lemma4} shows that without requiring that $\textnormal{P}_{\scriptscriptstyle \bm{\theta}}(qR>2)=1$ (or $\textnormal{P}_{\scriptscriptstyle \bm{\theta}}(R>2)=1$ for the case studied in \cite{ChetelatWells} where $q=1$), it is impossible to have $\textnormal{E}_{\scriptscriptstyle \bm{\theta}}\left[1/F\right]<\infty$. The fact is, as mentioned in \cite{ArashSeverien}, the proof given in the quoted paper has a major mistake which resulted in the incorrect use of the Cauchy-Schwarz inequality. In particular, the misuse of this inequality on their page 3153, led to an incorrect bound for the quantity $\bm{X}'(\bm{T}^{+}\bm{T}\bm{A})^{+}(\bm{T}^{+}\bm{T}\bm{A})\bm{X}$. 
By using \Cref{Lemma4}, we derive the following theorem which represents a generalisation of the corrected version of Theorem 2 of \cite{ChetelatWells}. The established theorem is useful in deriving the domination of $\bm{\delta}_{r}\left(\bm{X},\bm{S}\right)$ over $\bm{\delta}^{0}\left(\bm{X}\right)$. 
\begin{theorem}\label{Theorem2}
Suppose that \Cref{ass:model} holds along with \eqref{FRdelta0} and $\textnormal{P}_{\scriptscriptstyle \bm{\theta}}(qR>2)=1$. 
For $\bm{A}$ the symmetric positive definite square root of $\bm{\Sigma}$, let $\Tilde{\bm{Y}}=\sqrt{q}\bm{Y}\bm{A}^{-1}$. Let $r$ be any bounded differentiable non-negative function $r:\mathbb{R}\longrightarrow [0,c_{1}]$ with bounded derivative $|r^{\prime}|\le c_{2}$. Let
       $ \bm{G}=\frac{r^2(F)}{F^2}\bm{S}^{+}\bm{X}\bm{X}^{\top}\bm{S}^{+}\bm{S}$, 
        \text{ }$g(\bm{X},\bm{S})=\frac{r(F)\bm{S}\bm{S}^{+}\bm{X}}{F}$ and let
    $\bm{H}=\bm{A}\bm{G}\bm{A}^{-1}$. Then
    \begin{flalign*}
        (i) \quad \mathrm{E}_{\scriptscriptstyle \bm{\theta}}\left[|\mathrm{div}_{\mathrm{vec}(\Tilde{\bm{Y}})}\mathrm{vec}(\Tilde{\bm{Y}}\bm{H})|\right]<\infty;&&
    \end{flalign*}
    \begin{flalign*}
        (ii) \quad \mathrm{E}_{\scriptscriptstyle \bm{\theta}}\left[|\mathrm{tr}(\nabla_Xg(\bm{X},\bm{S})^{\top})|\right] < \infty;&&
    \end{flalign*}
       \begin{flalign*}
        (iii) \quad \mathrm{E}_{\scriptscriptstyle \bm{\theta}}&\left[\mathrm{tr}\left(g^{\top}(\bm{X},\bm{S})\bm{\Sigma}^{-1}g(\bm{X},\bm{S})\right)\right]\\&=\ds{\frac{1}{q}}\times\, \mathrm{E}_{\scriptscriptstyle \bm{\theta}}\left[\frac{r^2(F)}{F}\left(nq+p-2\mathrm{tr}(\bm{S}\bm{S}^{+})-1+\frac{4\mathrm{tr}((\bm{X}^{\top}\bm{S}^{+}\bm{X})^2)}{F^2}\right)-\frac{4r(F)r^{\prime}(F)}{F^2}\right].&&
    \end{flalign*}
    \end{theorem}
    \begin{proof}
       $(i)$ \quad By Lemma \ref{Proposition5}  and triangle inequality, we get
        \begin{flalign*}
            &\left|\mathrm{div}_{\mathrm{vec}(\Tilde{\bm{Y}})}\mathrm{vec}(\Tilde{\bm{Y}}\bm{H})\right|=\left|\frac{r^2(F)}{F}\left(nq+p-2\mathrm{tr}(\bm{S}\bm{S}^{+})-1
            +\frac{4\mathrm{tr}\left((\bm{X}^{\top}\bm{S}^+\bm{X})^2\right)}{F^2}\right) 
             -\frac{4r(F)r^{\prime}(F)}{F^2}\mathrm{tr}\left((\bm{X}^{\top}\bm{S}^+\bm{X})^2\right)\right|\\
             &\leqslant \frac{r^2(F)}{F}\left|nq+p-2\mathrm{tr}(\bm{S}\bm{S}^{+})-1+\frac{4\mathrm{tr}\left((\bm{X}^{\top}\bm{S}^+\bm{X})^2\right)}{F^2}\right|
             +\left|\frac{4r(F)r^{\prime}(F)}{F^2}\right|\mathrm{tr}\left((\bm{X}^{\top}\bm{S}^+\bm{X})^2\right), 
&&
        \end{flalign*}
        then
        \begin{flalign*}
            &|\mathrm{div}_{\mathrm{vec}(\Tilde{\bm{Y}})}\mathrm{vec}(\Tilde{\bm{Y}}\bm{H})|
             \leqslant \frac{c_{1}^2}{F}\left| nq+p-2\mathrm{tr}(\bm{S}\bm{S}^{+})-1+\frac{4\mathrm{tr}\left((\bm{X}^{\top}\bm{S}^+\bm{X})^2\right)}{F^2}\right|
             +\frac{4c_{1}c_{2}}{F^2}\mathrm{tr}\left((\bm{X}^{\top}\bm{S}^+\bm{X})^2\right)
&&
        \end{flalign*}
        and then
                \begin{flalign*}
            &|\mathrm{div}_{\mathrm{vec}(\Tilde{\bm{Y}})}\mathrm{vec}(\Tilde{\bm{Y}}\bm{H})|
             \leqslant \frac{c_{1}^2}{F}\left|nq+p-2\mathrm{tr}(\bm{S}\bm{S}^{+})-1\right|+\frac{4c_{1}^2\mathrm{tr}\left((\bm{X}^{\top}\bm{S}^+\bm{X})^2\right)}{F^3}
             +\frac{4c_{1}c_{2}}{F^2}\mathrm{tr}\left((\bm{X}^{\top}\bm{S}^+\bm{X})^2\right).&&
        \end{flalign*}
        Therefore, since $\mathrm{tr}(\bm{S}\bm{S}^{+})=\mathrm{min}(nq,p)$ almost surely, we have
        \begin{flalign}
            \mathrm{E}_{\scriptscriptstyle \bm{\theta}}\left[ |\mathrm{div}_{\mathrm{vec}(\Tilde{\bm{Y}})}\mathrm{vec}(\Tilde{\bm{Y}}\bm{H})|\right]\leqslant c_{1}^2|nq+p-2\mathrm{min}(nq,p)-1|\mathrm{E}_{\scriptscriptstyle \bm{\theta}}\left[\frac{1}{F}\right]
            +4c_{1}^2\mathrm{E}_{\scriptscriptstyle \bm{\theta}}\left[\frac{\mathrm{tr}\left((\bm{X}^{\top}\bm{S}^+\bm{X})^2\right)}{F^3}\right]\nonumber\\
            +4c_{1}c_{2}\mathrm{E}_{\scriptscriptstyle \bm{\theta}}\left[\frac{\mathrm{tr}\left((\bm{X}^{\top}\bm{S}^+\bm{X})^2\right)}{F^2}\right].&&\label{exp_F}
        \end{flalign}
        Since $F$ is semi-positive definite, from \Cref{lemma:traceineq}, we have
           $\mathrm{tr}\left((\bm{X}^{\top}\bm{S}^+\bm{X})^2\right)
           \leqslant F^2$.
        Then,
        \begin{equation}
            \mathrm{tr}\left((\bm{X}^{\top}\bm{S}^+\bm{X})^2\right)/F^2\leqslant 1. \label{F2i}
        \end{equation}
        Hence,
        \begin{eqnarray}
              \mathrm{E}_{\scriptscriptstyle \bm{\theta}}\left[\mathrm{tr}\left((\bm{X}^{\top}\bm{S}^+\bm{X})^2\right)/F^2\right]\leqslant 1 \label{F2ii}
         \quad{ } \quad{ } \mbox{ and }\quad{ } \quad{ }
            \mathrm{tr}\left((\bm{X}^{\top}\bm{S}^+\bm{X})^2\right)/F^3 
            \leqslant 1/F. 
        \end{eqnarray}
        Therefore,
        \begin{equation}
            \mathrm{E}_{\scriptscriptstyle \bm{\theta}}\left[\mathrm{tr}\left((\bm{X}^{\top}\bm{S}^+\bm{X})^2\right)/F^3\right]\leqslant \mathrm{E}_{\scriptscriptstyle \bm{\theta}}\left[1/F\right].\label{F3}
        \end{equation}
        Then, by \eqref{F2ii} and \eqref{F3} together with \eqref{exp_F} we get
        \begin{flalign}
             \mathrm{E}_{\scriptscriptstyle \bm{\theta}}\left[|\mathrm{div}_{\mathrm{vec}(\Tilde{\bm{Y}})}\mathrm{vec}(\Tilde{\bm{Y}}\bm{H})|\right]\leqslant &c_{1}^2\left|nq+p-2\mathrm{min}(nq,p)-1\right|\mathrm{E}_{\scriptscriptstyle \bm{\theta}}\left[\frac{1}{F}\right]
            &+4c_{1}^2E\left[\frac{1}{F}\right]
            &+4c_{1}c_{2}.&&\label{exp_FF}
        \end{flalign}
        Further, since $\textnormal{P}_{\scriptscriptstyle \bm{\theta}}(qR>2)=1$ then by Lemma \ref{Lemma4}, we get $\mathrm{E}_{\scriptscriptstyle \bm{\theta}}\left[1/F\right]<\infty$.\\
        $(ii)$ \quad Similarly to Part $(i)$, by Part $(i)$ \Cref{Proposition3},  we get
        \begin{flalign*}
            \mathrm{E}_{\scriptscriptstyle \bm{\theta}}\left[|\mathrm{tr}(\nabla_{\bm{X}}g(\bm{X},\bm{S})^{\top})|\right]&= \mathrm{E}_{\scriptscriptstyle \bm{\theta}}\left[\left|2r^{\prime}(F)+\frac{r(F)}{F}(q\mathrm{tr}(\bm{S}\bm{S}^{+})-2)\right|\right] 
            \leqslant 2c_{2}+c_{1}|q\mathrm{min}(nq,p)-2|\mathrm{E}_{\scriptscriptstyle \bm{\theta}}\left[\frac{1}{F}\right]<\infty,&&
        \end{flalign*}
        this proves Part~$(ii)$.\\
               $(iii)$ \quad Since $\textnormal{P}_{\scriptscriptstyle \bm{\theta}}(qR>2)=1$, by Part~$(ii)$, 
               we get $\textnormal{E}_{\scriptscriptstyle \bm{\theta}}[|\mathrm{div}_{\mathrm{vec}(\Tilde{\bm{Y}})}\mathrm{vec}(\Tilde{\bm{Y}}\bm{H})|]<\infty$. Therefore, from Part $(ii)$ of Proposition~\ref{Proposition3bis}, we have
        \begin{flalign*}
            \mathrm{E}_{\scriptscriptstyle \bm{\theta}}\left[\mathrm{tr}\left(g^{\top}(\bm{X},\bm{S})\bm{\Sigma}^{-1}g(\bm{X},\bm{S})\right)\right]
            =\mathrm{E}_{\scriptscriptstyle \bm{\theta}}\left[\mathrm{div}_{\mathrm{vec}(\Tilde{\bm{Y}})}\mathrm{vec}(\Tilde{\bm{Y}}\bm{H})\right]\Big/q.
        \end{flalign*}
        Further, from \Cref{Proposition5}, we have
        \begin{flalign*}
            \mathrm{div}_{\mathrm{vec}(\Tilde{\bm{Y}})}\mathrm{vec}(\Tilde{\bm{Y}}\bm{H})=\frac{r^2(F)}{F}\left( nq+p-2\mathrm{tr}(\bm{S}\bm{S}^{+})-1+\frac{4\mathrm{tr}((\bm{X}^{\top}\bm{S}^{+}\bm{X})^2)}{F^2}\right)
         -\frac{4r(F)r^{\prime}(F)}{F^2}.
        \end{flalign*}
        Hence,
        \begin{flalign*}
            \mathrm{E}_{\scriptscriptstyle \bm{\theta}}&\left[\mathrm{tr}\left(g^{\top}(\bm{X},\bm{S})\bm{\Sigma}^{-1}g(\bm{X},\bm{S})\right)\right]\\
            &=\ds{\frac{1}{q}}\,\mathrm{E}_{\scriptscriptstyle \bm{\theta}}\left[\frac{r^2(F)}{F}\left(nq+p-2\mathrm{tr}(\bm{S}\bm{S}^{+})-1+\frac{4\mathrm{tr}((\bm{X}^{\top}\bm{S}^{+}\bm{X})^2)}{F^2}\right)-\frac{4r(F)r^{\prime}(F)}{F^2}\right],&&
        \end{flalign*}
       which completes the proof.
    \end{proof}
It should be noted that Part~{\bf(i)} of \Cref{Theorem2} corresponds to Theorem~2 of \cite{ChetelatWells} in the special case where $q=1$. Nevertheless, \cite{ChetelatWells} omitted to impose the condition, about the rank of $\bm{S}$, as required in \Cref{Theorem2}. In this section, we give a counter-example which shows that if $\textnormal{P}_{\scriptscriptstyle \bm{\theta}}\left(qR\leqslant 2\right)>0$, the above result as well as Theorem~2 of \cite{ChetelatWells} do not hold i.e. $\mathrm{E}_{\scriptscriptstyle \bm{\theta}}\left[|\mathrm{div}_{\mathrm{vec}(\Tilde{\bm{Y}})}\mathrm{vec}(\Tilde{\bm{Y}}\bm{H})|\right]=\infty$.

\begin{corollary}\label{Corollary1}
    Suppose that the conditions of \Cref{Theorem2} hold and suppose that $q\geq 3$. Then, for all $n$ and $p$,
        $\mathrm{E}_{\scriptscriptstyle \bm{\theta}}\left[\left|\mathrm{div}_{\mathrm{vec}(\Tilde{\bm{Y}})}\mathrm{vec}(\Tilde{\bm{Y}}\bm{H})\right|\right]<\infty$.
    \begin{proof}
        Since $q\geq 3$ and $R\geq 1$ a.s., we have $qR>2$ a.s. Therefore, $\textnormal{P}_{\scriptscriptstyle \bm{\theta}}(qR>2)=1$. Then, by \Cref{Theorem2}, we get
        $\mathrm{E}_{\scriptscriptstyle \bm{\theta}}\left[\left|\mathrm{div}_{\mathrm{vec}(\Tilde{\bm{Y}})}\mathrm{vec}(\Tilde{\bm{Y}}\bm{H})\right|\right]<\infty$,
       which completes the proof.
    \end{proof}
\end{corollary}
\begin{corollary}\label{Corollary2}
    Suppose that the conditions of \Cref{Theorem2} hold with
    $r(t)\in [c^{*},c_{1}]\subset (0,c_{1}]$ for some $c^{*}>0$, for all $t\geq0$ and suppose that {$|p-nq|>1$}. 
Then,\\
        $\mathrm{E}_{\scriptscriptstyle \bm{\theta}}\left[|\mathrm{div}_{\mathrm{vec}(\Tilde{\bm{Y}})}\mathrm{vec}(\Tilde{\bm{Y}}\bm{H})|\right]<\infty$
    if and only if $\textnormal{P}_{\scriptscriptstyle \bm{\theta}}(qR>2)=1$.
    \begin{proof}
    If $\textnormal{P}_{\scriptscriptstyle \bm{\theta}}(qR>2)=1$ then, by \Cref{Theorem2}, we get
        $\mathrm{E}_{\scriptscriptstyle \bm{\theta}}\left[|\mathrm{div}_{\mathrm{vec}(\Tilde{\bm{Y}})}\mathrm{vec}(\Tilde{\bm{Y}}\bm{H})|\right]<\infty$.
        Conversely, 
suppose that $\mathrm{E}_{\scriptscriptstyle \bm{\theta}}\left[|\mathrm{div}_{\mathrm{vec}(\Tilde{\bm{Y}})}\mathrm{vec}(\Tilde{\bm{Y}}\bm{H})|\right]<\infty$. We have
    \begin{flalign*}
       \mathrm{div}_{\mathrm{vec}(\Tilde{\bm{Y}})}\mathrm{vec}(\Tilde{\bm{Y}}\bm{H})=&\left( nq+p-2\mathrm{min}(p,nq)-1+4\mathrm{tr}\left((\bm{X}^{\top}\bm{S}^{+}\bm{X})^2\right)/F^2\right)r^2(F)/F\\
       &-4r(F)r^{\prime}(F)\mathrm{tr}\left((\bm{X}^{\top}\bm{S}^{+}\bm{X})^2\right)/F^2.&&
    \end{flalign*}
     Then, since $ 0 \leqslant \mathrm{tr}((\bm{X}^{\top}\bm{S}^{+}\bm{X})^2)/F^2\leqslant 1$ and $nq+p-2\mathrm{min}(p,nq)=|p-nq|$, we get\\
          $\mathrm{div}_{\mathrm{vec}(\Tilde{\bm{Y}})}\mathrm{vec}(\Tilde{\bm{Y}}\bm{H}) \geqslant 
          r^2(F)\left( |p-nq|-1\right)/F-4c_{1}c_{2}$.
     Therefore, since $|p-nq|>1$, we get
      \begin{flalign*}
          0<\frac{r^2(F)}{F} \leqslant& \frac{1}{|p-nq|-1}\left(\mathrm{div}_{\mathrm{vec}(\Tilde{\bm{Y}})}\mathrm{vec}(\Tilde{\bm{Y}}\bm{H}) + 4c_{1}c_{2}\right) 
           = \frac{1}{|p-nq|-1}\left(\left|\mathrm{div}_{\mathrm{vec}(\Tilde{\bm{Y}})}\mathrm{vec}(\Tilde{\bm{Y}}\bm{H}) + 4c_{1}c_{2}\right|\right)\\
          & \leqslant \frac{1}{|p-nq|-1}\left(\left|\mathrm{div}_{\mathrm{vec}(\Tilde{\bm{Y}})}\mathrm{vec}(\Tilde{\bm{Y}}\bm{H})\right| + 4c_{1}c_{2}\right).&&
      \end{flalign*}
      Therefore
      \begin{eqnarray*}
          \mathrm{E}_{\scriptscriptstyle \bm{\theta}}\left[r^2(F)/F\right]\leqslant \left(\mathrm{E}_{\scriptscriptstyle \bm{\theta}}\left[\left|\mathrm{div}_{\mathrm{vec}(\Tilde{\bm{Y}})}\mathrm{vec}(\Tilde{\bm{Y}}\bm{H})\right|\right] + 4c_{1}c_{2}\right)\big/(|p-nq|-1) < \infty.&&
      \end{eqnarray*}

      Further, we have $(c^{*})^2/F \leqslant r^2(F)/F$, this implies that
      \begin{eqnarray*}
           (c^{*})^2\mathrm{E}_{\scriptscriptstyle \bm{\theta}}\left[1/F\right] \leqslant \mathrm{E}_{\scriptscriptstyle \bm{\theta}}\left[r^2(F)/F\right] < \infty, &&
      \end{eqnarray*}
     which implies that
        $\mathrm{E}_{\scriptscriptstyle \bm{\theta}}\left[1/F\right] < \infty$.
    Therefore, by Lemma \ref{Lemma4}, $\textnormal{P}_{\scriptscriptstyle \bm{\theta}}(qR>2)=1$, as stated. 
    \end{proof}
\end{corollary}

In the following example, we consider a positive function $r$ such that $\mathrm{E}_{\scriptscriptstyle \bm{\theta}}[|\mathrm{div}_{\mathrm{vec}(\Tilde{\bm{Y}})}\mathrm{vec}(\Tilde{\bm{Y}}\bm{H})|]=\infty$. This emphasizes the significance of the assumption regarding $qR>2$ with probability one, where $R=\textnormal{rank}(\bm{S})$. In particular, we demonstrate that when $\textnormal{P}_{\scriptscriptstyle \bm{\theta}}(qR \leqslant 2)>0$, it is possible to have $\textnormal{E}_{\scriptscriptstyle \bm{\theta}}\left[1/F\right]=\infty$. This finding renders obsolete one of the crucial steps in the proof of Theorem~2 of \cite{ChetelatWells}.
\begin{Example}
    Let $\bm{X}\sim \mathcal{N}_{2}\left(\begin{bmatrix}
    1 \\
    1
    \end{bmatrix},\bm{I}_2\right)$ and $\bm{Y} = \begin{bmatrix}
    U \vdots V
    \end{bmatrix}$
    where $U$ and $V$ are independent random variable distributed as $\mathcal{N}\left(0,\,1\right)$ i.e.  $\bm{Y} \sim \mathcal{N}_{1\times2}(0,1\otimes \bm{I}_2)$. let $r(x) = \frac{1}{1+e^{-x}}$. Let $\bm{S}^{+}=\bm{P}\bm{D}^{+}\bm{P}^{\top}$ be the spectral decomposition of $\bm{S}^{+}$ where $\bm{D}^{+}=diag(d_1,0)$. Since
    \begin{flalign*}
        F=\bm{X}^{\top}\bm{S}^{+}\bm{X}=\bm{X}^{\top}\bm{P}\bm{D}^{+}\bm{P}^{\top}\bm{X}=(\bm{P}^{\top}\bm{X})^{\top}\bm{D}^{+}\bm{P}^{\top}\bm{X}.&
    \end{flalign*}
    Therefore
    \begin{flalign*}
        \frac{F}{d_1}=(\bm{P}^{\top}\bm{X})^{\top}\begin{bmatrix}
            1 & 0 \\
            0 & 0
        \end{bmatrix}\bm{P}^{\top}\bm{X}.
    \end{flalign*}
   Note that $d_1$ and $\bm{P}$ are  functions of $(U,V)$ and note that $\bm{P}^{\top}\bm{X}\big| U,V\sim \mathcal{N}_{2}\left(\bm{P}^{\top}\begin{bmatrix}
    1 \\
    1
    \end{bmatrix},\bm{I}_{2}\right)$. Then, 
    \begin{flalign*}
      \frac{\bm{X}^{\top}\bm{S}^{+}\bm{X}}{d_1} \big|U,V \sim \chi^2_1(\delta_{0})
    \end{flalign*}
            where $\delta_{0} = \begin{bmatrix}
        1 & 1
    \end{bmatrix}\bm{P}\begin{bmatrix}
        1 & 0 \\
        0 & 0
    \end{bmatrix}\bm{P}^{\top}\begin{bmatrix}
        1\\
        1
    \end{bmatrix}.$
    Therefore, 
     \begin{flalign*}
        \textnormal{E}_{\scriptscriptstyle \bm{\theta}}\left[\frac{d_1}{F}\big|U,V\right]= \textnormal{E}_{\scriptscriptstyle \bm{\theta}}\left[\frac{d_1}{\bm{X}^{\top}\bm{S}^{+}\bm{X}}\big|U,V\right]
        =\textnormal{E}_{\scriptscriptstyle \bm{\theta}}\left[(\chi^2_1(\delta_{0}))^{-1}\big|U,V\right]=\infty,&&
     \end{flalign*}
     almost surely.
     Then,
     \begin{eqnarray*}
        \textnormal{E}_{\scriptscriptstyle \bm{\theta}}\left[\frac{1}{F}\big|U,V\right]=\frac{1}{d_1}\textnormal{E}_{\scriptscriptstyle \bm{\theta}}\left[\frac{d_1}{F}\big|U,V\right]=\infty,
     \end{eqnarray*}
 \mbox{ with probability one.}     Hence,
 \begin{flalign}
     \textnormal{E}_{\scriptscriptstyle \bm{\theta}}\left[1/F\right] = \infty.\label{exp1surF}
 \end{flalign}

  Further, we have

       $$\mathrm{div}_{\mathrm{vec}(\Tilde{\bm{Y}})}\mathrm{vec}(\Tilde{\bm{Y}}\bm{H}) = \frac{r^2(F)}{F}(n+p-2\mathrm{min}(n,p)+3)-4r(F)r^{\prime}(F),$$
    where $\mathrm{min}(n,p)=n=1$.
    Therefore, we get
    \begin{eqnarray*}
        \mathrm{div}_{\mathrm{vec}(\Tilde{\bm{Y}})}\mathrm{vec}(\Tilde{\bm{Y}}\bm{H}) 
        = \frac{r^2(F)}{F}(1+2-2+3)-4r(F)r^{\prime}(F) 
        = \frac{4r^2(F)}{F} - 4r(F)r^{\prime}(F).
    \end{eqnarray*}
    Then,
       \begin{eqnarray*}
       0<4r^2(F)/F
       =  \mathrm{div}_{\mathrm{vec}(\Tilde{\bm{Y}})}\mathrm{vec}(\Tilde{\bm{Y}}\bm{H}) + 4r(F)r^{\prime}(F) 
        = \left|\mathrm{div}_{\mathrm{vec}(\Tilde{\bm{Y}})}\mathrm{vec}(\Tilde{\bm{Y}}\bm{H}) + 4r(F)r^{\prime}(F)\right|
        \end{eqnarray*}
    and then,
    \begin{eqnarray*}
       0<4r^2(F)/F 
       \leqslant  \left|\mathrm{div}_{\mathrm{vec}(\Tilde{\bm{Y}})}\mathrm{vec}(\Tilde{\bm{Y}}\bm{H})\right| + 4r(F)r^{\prime}(F).
    \end{eqnarray*}

    Hence, since $r(F)$ and $r^{\prime}(F)$ are bounded by 1, we get
    \begin{flalign}
       0< 4r^2(F)/F \leqslant  \left|\mathrm{div}_{\mathrm{vec}(\Tilde{\bm{Y}})}\mathrm{vec}(\Tilde{\bm{Y}}\bm{H})\right| + 4. \label{ineq1/F}
    \end{flalign}
    We also have
        $\frac{4r^2(F)}{F}=\frac{4}{F(1+e^{-F})^2}$
    and since $1<1+e^{-F} \leqslant 2$,  then $1<(1+e^{-F})^2 \leqslant 4$ and then, 
    \begin{flalign*}
        \frac{1}{F} \leqslant \frac{4}{F(1+e^{-F})^2}=\frac{4r^2(F)}{F}.
    \end{flalign*}
    Then, together with \eqref{ineq1/F}, we get
    \begin{flalign*}
        \mathrm{E}_{\scriptscriptstyle \bm{\theta}}\left[1/F\right] \leqslant \mathrm{E}_{\scriptscriptstyle \bm{\theta}}\left[4r^2(F)/F\right]\leqslant  \textnormal{E}_{\scriptscriptstyle \bm{\theta}}\left[\left|\mathrm{div}_{\mathrm{vec}(\Tilde{\bm{Y}})}\mathrm{vec}(\Tilde{\bm{Y}}\bm{H})\right|\right] + 4.
    \end{flalign*}
    But, from \eqref{exp1surF}, $\mathrm{E}_{\scriptscriptstyle \bm{\theta}}\left[1/F\right] = \infty$, then, together with these last inequalities, we get
    \begin{flalign*}
        \mathrm{E}_{\scriptscriptstyle \bm{\theta}}\left[4r^2(F)/F\right] 
        =\mathrm{E}_{\scriptscriptstyle \bm{\theta}}\left[\left|\mathrm{div}_{\mathrm{vec}(\Tilde{\bm{Y}})}\mathrm{vec}(\Tilde{\bm{Y}}\bm{H})\right|\right]=\infty,
    \end{flalign*}
    this shows the necessity of the condition $\textnormal{P}_{\scriptscriptstyle \bm{\theta}}\left(qR>p\right)=1$.
\end{Example}
By using \Cref{Theorem2} and \Cref{Proposition1}, we derive the following result which establishes the risk dominance of $\bm{\delta}_{r}(\bm{X},\bm{S})$ over $\bm{\delta}^{0}(\bm{X})$. In the special case where $q=1$, the established result yields Theorem~1 of \cite{ChetelatWells}.
\begin{theorem}\label{Main_Theorem}
     Suppose that \Cref{ass:model} along  with $\textnormal{P}_{\scriptscriptstyle \bm{\theta}}(qR>2)=1$ and suppose that
   \begin{flalign*}
       &(i)\text{ r satisfies }  0\leqslant r \leqslant \frac{2(q.\min( nq,p)-2)}{ nq+p-2\min( nq,p)+3}\\
       &(ii)\text{ } r \text{ is non-decreasing}\\
       &(iii)\text{ }r^{\prime} \text{ is bounded}.&&
   \end{flalign*}
   Then, under 
   quadratic loss
       $L(\bm{\theta},\bm{\delta})=\mathrm{tr}\left((\bm{\delta}-\bm{\theta})^{\top}\bm{\Sigma}^{-1}(\bm{\delta}-\bm{\theta})\right)$,
   $\bm{\delta}_{r}(\bm{X},\bm{S})$ dominates $\bm{\delta}^{0}(\bm{X})$.
   \begin{proof}
   Let $g(\bm{X},\bm{S})=r(F)\bm{S}\bm{S}^{+}\bm{X}/F$. Thus $\bm{\delta}_{r}(\bm{X},\bm{S})=\bm{X}-g(\bm{X},\bm{S})$. The risk difference under the quadratic loss between $\bm{\delta}_{r}(\bm{X},\bm{S})$ and $\bm{\delta}^{0}(\bm{X})$ is
       \begin{flalign*}
           \Delta_{\bm{\theta}}=&\mathrm{E}_{\scriptscriptstyle \bm{\theta}}\left[\mathrm{tr}\left(\left(\bm{X}-g(\bm{X},\bm{S})-\bm{\theta})^{\top}\right)\bm{\Sigma}^{-1}\left(\bm{X}-g(\bm{X},\bm{S})-\bm{\theta} \right)\right)\right] 
           -\mathrm{E}_{\scriptscriptstyle \bm{\theta}}\left[\mathrm{tr}\left(\left(\bm{X}-\bm{\theta}\right)^{\top}\bm{\Sigma}^{-1}\left(\bm{X}-\bm{\theta}\right)\right)\right]. 
       \end{flalign*}
       This gives
      \begin{eqnarray}
           \Delta_{\bm{\theta}} 
           =-2\mathrm{E}_{\scriptscriptstyle \bm{\theta}}\left[\mathrm{tr}\left(g^{\top}(\bm{X},\bm{S})\bm{\Sigma}^{-1}(\bm{X}-\bm{\theta})\right)\right]+\mathrm{E}_{\scriptscriptstyle \bm{\theta}}\left[\mathrm{tr}\left(g^{\top}(\bm{X},\bm{S})\bm{\Sigma}^{-1}g(\bm{X},\bm{S})\right)\right].&&\label{risk1}
       \end{eqnarray}
       From \Cref{Proposition1} and Part~(iii) of \Cref{Theorem2}, we have
       \begin{flalign*}
           \Delta_{\bm{\theta}}=\mathrm{E}_{\scriptscriptstyle \bm{\theta}}&\Big[r^2(F)\left(n+p-2\mathrm{tr}(\bm{S}\bm{S}^{+})-1+4\mathrm{tr}((\bm{X}^{\top}\bm{S}^{+}\bm{X})^2)/F^2\right)
           /(qF)-2r(F)(q\mathrm{tr}(\bm{S}\bm{S}^{+})-2)/F\\
           &-4r^{\prime}(F)\left(1+r(F)/(qF^2)\right)\Big].&&
       \end{flalign*}
       Since $r$ is non-negative and non-decreasing, $-4r^{\prime}(F)(1+\frac{r(F)}{F^2}) \leqslant 0$ and since $\frac{\mathrm{tr}((\bm{X}^{\top}\bm{S}^{+}\bm{X})^2)}{F^2}\leqslant 1$, we have
       \begin{flalign}
           \frac{r^2(F)}{qF}\left(nq+p-2\mathrm{tr}(\bm{S}\bm{S}^{+})-1
           +\frac{4\mathrm{tr}\left(\left(\bm{X}^{\top}\bm{S}^{+}\bm{X}\right)^2\right)}{F^2}\right)
           \leqslant \frac{r^2(F)}{qF}\left(nq+p-2\mathrm{tr}(\bm{S}\bm{S}^{+})+3\right).&&\label{risk4}
       \end{flalign}
       Under the condition (i) on $r$ and since $\mathrm{tr}(\bm{S}\bm{S}^{+})=\min(nq,p)$ almost surely,  we have \\ $r(F)\leqslant\left(2(q\mathrm{tr}(\bm{S}\bm{S}^{+})-2)\right)\big/ \left(nq+p-2\mathrm{tr}(\bm{S}\bm{S}^{+})+3\right)$ and equivalently
       \begin{flalign*}
          r^2(F)\left(nq+p-2\mathrm{tr}(\bm{S}\bm{S}^{+})+3\right)/(qF) \leqslant 2r(F)\left(q\mathrm{tr}(\bm{S}\bm{S}^{+})-2\right)/(qF).
       \end{flalign*}
       Therefore, by \eqref{risk4}, we get
       \begin{flalign*}
           \mathrm{E}_{\scriptscriptstyle \bm{\theta}}\left[\left(nq+p-2\mathrm{tr}(\bm{S}\bm{S}^{+})-1+4\mathrm{tr}((\bm{X}^{T}\bm{S}^{+}\bm{X})^2)/F^2\right)r^2(F)/(qF)
           -2r(F)\left(q\mathrm{tr}(\bm{S}\bm{S}^{+})-2\right)/(qF)\right]\leqslant 0,
       \end{flalign*}
       therefore
           $\Delta_{\bm{\theta}}\leqslant 0$,
       which completes the proof.
   \end{proof}
\end{theorem}
\section{Numerical study}\label{sec:simulation}
In this section, we present some numerical results which corroborate the established theoretical results. In this simulation study, we consider $F = \mathrm{tr}(\bm{X}^{\top}\bm{S}^{+}\bm{X})$ and  $r(F) = \left(1+e^{-F}\right)^{-1},$
and the proposed estimator is $\bm{\delta}_{r}\left(\bm{X},\bm{S}\right)=(\bm{I}_{p}-\bm{S}\bm{S}^{+}r(F)/F)\bm{X}.$
Further, we generate samples for various values of $p$ equals to 
$(16,24,56 \text{ and }104)$ along with 11 different matrix $\bm{\theta}$ configurations for $q=3$ fixed and 
we set  $\bm{\Sigma} = \bm{I}_{p}$ for the first case, and $\bm{\Sigma}=\bm{e}_{p}\bm{e}'_{p}+3\bm{I}_{p}$ for the second case. We obtained similar results for both cases and thus, to save the space of this paper, we only report the results of the first case. For each value of $p$, we explore four distinct sample sizes: 
$n=p/8, p/4, p-1, \text{ and }2p$. This comprehensive approach allows us to investigate the impact of different $p$, $n$ and $||\bm{\theta}||$ combinations on the results of the simulation. 
\Cref{fig:enter-label} shows the simulation results. One can see that the simulation study supports the theoretical findings. As established theoretically, the risk difference
between the proposed estimator $\bm{\delta}_{r}\left(\bm{X},\bm{S}\right)$ and the classical estimator $\bm{\delta}^{0}\left(\bm{X}\right) = \bm{X}$ is not positive. This leads to the dominance of $\bm{\delta}_{r}\left(\bm{X},\bm{S}\right)$ over $\bm{\delta}^{0}\left(\bm{X}\right)$. 

 Furthermore, as presented in \Cref{fig:enter-label}, 
 a consistent pattern becomes evident across all four cases: the risk difference between the two estimators diminishes as the norm of the mean matrix, $||\bm{\theta}||$, increases. This intriguing observation serves as a compelling incentive for potential future research. Further exploring how the mean matrix $\bm{\theta}$ is related to the difference in risk between these estimators under the invariant quadratic loss opens up an interesting path for further investigation.


\begin{figure}[htbp]
    \centering
    \includegraphics[scale=0.9]{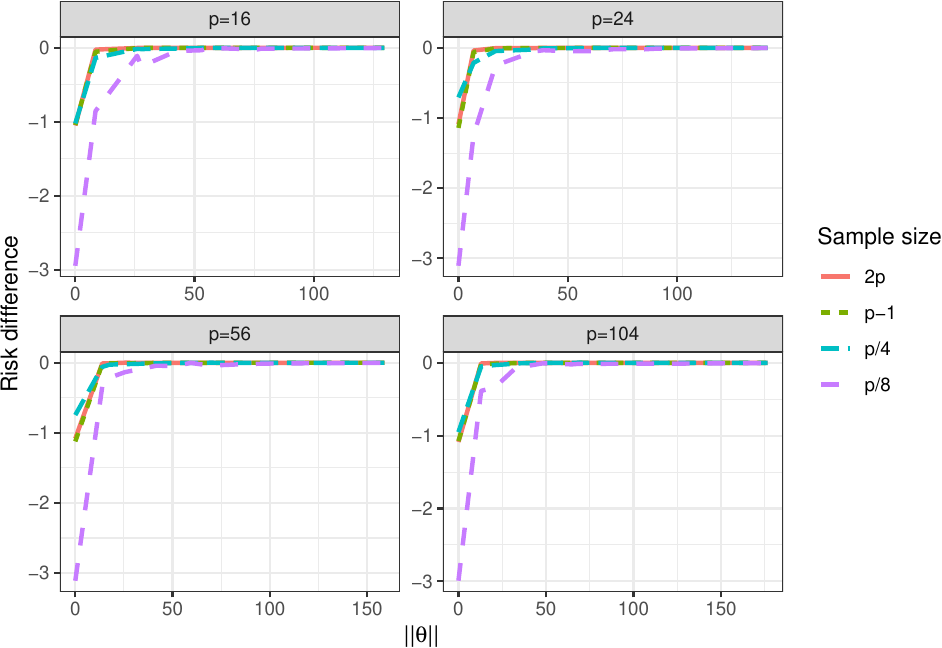}
    \caption{The risk difference between the proposed estimator $\bm{\delta}_{r}\left(\bm{X},\bm{S}\right)$ and the MLE}
    \label{fig:enter-label}
\end{figure}
\section{Concluding remark}\label{sec:conclusion}
 In this paper, we introduced a class of James-Stein type estimators for the mean parameter matrix of a multivariate normal in the context where the covariance-variance is unknown. The problem studied is more general than the one considered in \cite{ChetelatWells} and we established the revised version of their Theorem~2. In particular, we generalized the existing results in three ways. First, we studied  estimation problem concerning the parameter matrix in the context of high-dimensional data. Second, we proposed a class of matrix shrinkage  estimators and, we established the conditions for any member of the proposed class to have a finite risk. Third, we analysed the risk dominance of the proposed class of estimators versus the classical maximum likelihood estimator.  To this end, we establish some mathematical results which are interesting in their own. The additional novelty consists in the fact that, in the context of matrix estimation problem, the derived results hold in the classical case as well as in the context of high-dimensional data. We also present some numerical results which corroborate the theoretical findings. As a direction for future research, 
 it would be interesting to relax the conditions for the established \Cref{Main_Theorem} to hold.

\appendix
\section{Proofs of some fundamental results}\label{sec:append}
\begin{lemma} \label{lemma:traceineq}Let $\bm{A}$ be a symmetric and positive semi-definite matrix. Then, \\ $\textnormal{tr}\left(\bm{A}^{2}\right)\leqslant\left(\textnormal{tr}\left(\bm{A}\right)\right)^{2}$.
\end{lemma}
\begin{proof}
Let $\lambda_{1},\lambda_{2},\dots,\lambda_{n}$ be the eigenvalues of $\bm{A}$. Then, $\lambda_{1}^{2},\lambda_{2}^{2},\dots,\lambda_{n}^{2}$ are the eigenvalues of $\bm{A}^{2}$.  Therefore, since $\lambda_{i}\geq0$, $i=1,2,\dots,n$, we have
$\textnormal{tr}\left(\bm{A}^{2}\right)=\ds\sum_{i=1}^{n}\lambda_{i}^{2}\leqslant \left(\ds\sum_{i=1}^{n}\lambda_{i}\right)^{2}=\left(\textnormal{tr}\left(\bm{A}\right)\right)^{2}$,
this completes the proof.
\end{proof}

    \begin{proof}[Proof of \Cref{Proposition2}]
        $(i)$ \quad Let $\Tilde{\bm{X}}=\bm{A}^{-1}(\bm{X}-\bm{\theta})$ where $\bm{A}$ is a symmetric positive definite square root of $\bm{\Sigma}$. Then, $\Tilde{\bm{X}}\sim \mathcal{N}_{p\times q}(0,\bm{I}_{p}\otimes \bm{I}_{q})$. Therefore, $\bm{X}_{ij}\sim \mathcal{N}(0,1)$. Let $h=\bm{A}^{-1}g(\bm{X},\bm{S})$.
        Then, we have
            $\mathrm{tr}(g^{\top}(\bm{X},\bm{S})\bm{\Sigma}^{-1}(\bm{X}-\bm{\theta})) 
            =\mathrm{tr}(g^{\top}(\bm{X},\bm{S})\bm{A}^{-1}\bm{A}^{-1}(\bm{X}-\bm{\theta})) 
            =\mathrm{tr}(h^{\top}\Tilde{\bm{X}})=\ds\sum_{i}^{}{(h^{\top}\Tilde{\bm{X}})_{ii}}$ and  
  then,
  \begin{eqnarray}
  \mathrm{tr}(g^{\top}(\bm{X},\bm{S})\bm{\Sigma}^{-1}(\bm{X}-\bm{\theta}))=\sum_{i,j}^{}{h^{\top}_{ij}\Tilde{\bm{X}}_{ji}}.\label{P41}
        \end{eqnarray}
        Therefore, by \eqref{P41} we have
\begin{eqnarray*}
    \mathrm{E}_{\scriptscriptstyle \bm{\theta}}\left[\mathrm{tr}\left(g^{\top}(\bm{X},\bm{S})\bm{\Sigma}^{-1}(\bm{X}-\bm{\theta})\right)\right]= \mathrm{E}_{\scriptscriptstyle \bm{\theta}}\left[\sum_{i,j}^{}{h^{\top}_{ij}\Tilde{\bm{X}}_{ji}}\right] 
    = \sum_{i,j}^{}{\mathrm{E}_{\scriptscriptstyle \bm{\theta}}\left[h^{\top}_{ij}\Tilde{\bm{X}}_{ji}\right]} 
    = \sum_{i,j}^{}{\mathrm{E}_{\scriptscriptstyle \bm{\theta}}\left[\Tilde{\bm{X}}_{ji}h^{\top}_{ij}\right]}. 
\end{eqnarray*}
        Therefore, by Lemma~1 of \cite{SteinCharles}, we get
\[
    \sum_{i,j}^{}{\mathrm{E}_{\scriptscriptstyle \bm{\theta}}[\Tilde{\bm{X}}_{ji}h^{\top}_{ij}]}
    = \sum_{i,j}^{}{\textnormal{E}_{\scriptscriptstyle \bm{\theta}}\bigg[\frac{\partial}{\partial \Tilde{\bm{X}}_{ji}}h^{\top}_{ij}\bigg]}
    = \sum_{i,j}^{}{\textnormal{E}_{\scriptscriptstyle \bm{\theta}}\bigg[\frac{\partial}{\partial \Tilde{\bm{X}}_{ji}}h_{ji}\bigg]}
    = \textnormal{E}_{\scriptscriptstyle \bm{\theta}}\bigg[\sum_{i,j}^{}{\frac{\partial}{\partial \Tilde{\bm{X}}_{ji}}h_{ji}}\bigg]=\textnormal{E}_{\scriptscriptstyle \bm{\theta}}\bigg[\sum_{i,j}^{}{\frac{\partial}{\partial \Tilde{\bm{X}}_{ji}}(\bm{A}^{-1}g(\bm{X},\bm{S}))_{ji}}\bigg].
\]
Then,
\begin{flalign}
    \sum_{i,j}^{}{\mathrm{E}_{\scriptscriptstyle \bm{\theta}}[\Tilde{\bm{X}}_{ji}h^{\top}_{ij}]}
    = \textnormal{E}_{\scriptscriptstyle \bm{\theta}}\bigg[\sum_{i,j}^{}{\frac{\partial}{\partial \Tilde{\bm{X}}_{ji}}\sum_{k}^{}{\bm{A}^{-1}_{jk}g(\bm{X},\bm{S})_{ki}}}\bigg] 
    = \textnormal{E}_{\scriptscriptstyle \bm{\theta}}\bigg[\sum_{i,j,k}^{}{\bm{A}^{-1}_{jk}\frac{\partial}{\partial \Tilde{\bm{X}}_{ji}}g(\bm{X},\bm{S})_{ki}}\bigg]. \label{P42} &&
\end{flalign}
        Now, by applying the chain rule in \eqref{P42}, we have
\begin{eqnarray}
 \quad{ }\quad{ }  \ds{\sum_{i,j}^{}}{\mathrm{E}_{\scriptscriptstyle \bm{\theta}}[\Tilde{\bm{X}}_{ji}h^{\top}_{ij}]}
    = \textnormal{E}_{\scriptscriptstyle \bm{\theta}}\bigg[\ds{\sum_{i,j,k,l,\alpha}^{}}{\bm{A}^{-1}_{jk}\frac{\partial}{\partial \bm{X}_{l \alpha}}g(\bm{X},\bm{S})_{ki}\frac{\partial \bm{X}_{l \alpha}}{\partial \Tilde{\bm{X}}_{ji}}}\bigg]. 
    \label{P43}
\end{eqnarray}
        Since $\Tilde{\bm{X}}=\bm{A}^{-1}(\bm{X}-\bm{\theta})$, we have
            $\bm{X}_{l \alpha}=\ds\sum_{t}^{}{A_{l t}\Tilde{\bm{X}}_{t \alpha}}+\bm{\theta}_{l \alpha}$
        then,
        \begin{flalign}
            \frac{\partial \bm{X}_{l \alpha}}{\partial \Tilde{\bm{X}}_{ji}}=\sum_{t}^{}{A_{l t}
            \frac{\partial \Tilde{\bm{X}}_{t \alpha}}{\partial \Tilde{\bm{X}}_{ji}}}=\sum_{t}^{}{A_{l t}\delta_{t j}\delta_{\alpha i}}=A_{l j}\delta_{\alpha i}\label{P44}.
        \end{flalign}
        Therefore, by replacing \eqref{P44} in \eqref{P43}, we get
\begin{eqnarray*}
    \textnormal{E}_{\scriptscriptstyle \bm{\theta}}\bigg[\sum_{i,j,k,l,\alpha}^{}{\bm{A}^{-1}_{jk}\frac{\partial}{\partial \bm{X}_{l \alpha}}g(\bm{X},\bm{S})_{ki}\frac{\partial \bm{X}_{l \alpha}}{\partial \Tilde{\bm{X}}_{ji}}}\bigg] = \textnormal{E}_{\scriptscriptstyle \bm{\theta}}\bigg[\sum_{i,j,k,l,\alpha}^{}{\bm{A}^{-1}_{jk}\frac{\partial}{\partial \bm{X}_{l \alpha}}g(\bm{X},\bm{S})_{ki}A_{l j}\delta_{\alpha i}}\bigg] \\
    = \textnormal{E}_{\scriptscriptstyle \bm{\theta}}\bigg[\sum_{i,k,l}^{}{\frac{\partial}{\partial \bm{X}_{l i}}g(\bm{X},\bm{S})_{ki}\sum_{j}^{}{A_{l j}\bm{A}^{-1}_{jk}}}\bigg]
    = \textnormal{E}_{\scriptscriptstyle \bm{\theta}}\bigg[\sum_{i,k,l}^{}{\frac{\partial}{\partial \bm{X}_{l i}}g(\bm{X},\bm{S})_{ki}(\bm{A}\bm{A}^{-1})_{lk}}\bigg] 
\end{eqnarray*}
and then,
\begin{eqnarray*}
    \textnormal{E}_{\scriptscriptstyle \bm{\theta}}\bigg[\sum_{i,j,k,l,\alpha}^{}{\bm{A}^{-1}_{jk}\frac{\partial}{\partial \bm{X}_{l \alpha}}g(\bm{X},\bm{S})_{ki}\frac{\partial \bm{X}_{l \alpha}}{\partial \Tilde{\bm{X}}_{ji}}}\bigg] 
    = \textnormal{E}_{\scriptscriptstyle \bm{\theta}}\bigg[\sum_{i,k}^{}{\frac{\partial}{\partial \bm{X}_{l i}}g(\bm{X},\bm{S})_{ki}}\bigg] 
    = \textnormal{E}_{\scriptscriptstyle \bm{\theta}}\bigg[\sum_{i,k}^{}{(\nabla_{\bm{X}})_{ki}g^{\top}(\bm{X},\bm{S})_{ik}}\bigg]. 
\end{eqnarray*}
Hence,
\begin{eqnarray*}
    \textnormal{E}_{\scriptscriptstyle \bm{\theta}}\bigg[\sum_{i,j,k,l,\alpha}^{}{\bm{A}^{-1}_{jk}\frac{\partial}{\partial \bm{X}_{l \alpha}}g(\bm{X},\bm{S})_{ki}\frac{\partial \bm{X}_{l \alpha}}{\partial \Tilde{\bm{X}}_{ji}}}\bigg]
    = \textnormal{E}_{\scriptscriptstyle \bm{\theta}}\bigg[\sum_{k}^{}{(\nabla_Xg^{\top}(\bm{X},\bm{S}))_{kk}}\bigg] 
    = \textnormal{E}_{\scriptscriptstyle \bm{\theta}}\big[\mathrm{tr}(\nabla_Xg^{\top}(\bm{X},\bm{S}))\big].
\end{eqnarray*}
    $(ii)$ We have 
    $\mathrm{tr}\biggl(\nabla_Xg^{\top}(\bm{X},\bm{S})\biggr)
    = \ds\sum_{i}^{}{(\nabla_Xg^{\top}(\bm{X},\bm{S}))_{ii}}= \ds\sum_{i,j}^{}{(\nabla_{\bm{X}})_{ij}g^{\top}_{ji}(\bm{X},\bm{S})}$. Then \\ $\mathrm{tr}\biggl(\nabla_Xg^{\top}(\bm{X},\bm{S})\biggr) 
    = \ds\sum_{i,j}^{}{(\nabla_{\bm{X}})_{ij}g_{ij}(\bm{X},\bm{S})} 
    = \ds\sum_{i,j}^{}{\frac{\partial g_{ij}(\bm{X},\bm{S})}{\partial X_{ij}}}$,
which completes the proof.
    \end{proof}

    \begin{proof}[Proof of \Cref{Proposition1}]
               $(i)$ \quad From Part $(ii)$ of \Cref{Proposition2},  we get
                 $$\mathrm{tr}\biggl(\nabla_Xg^{\top}(\bm{X},\bm{S})\biggr)= \ds\sum_{i,j}^{}{\frac{\partial g_{ij}(\bm{X},\bm{S})}{\partial X_{ij}}},$$
            and by Part $(iv)$ of Lemma \ref{Lemma3},  we have
             $\ds{\sum_{i,j}} \frac{\partial g_{ij(\bm{X},\bm{S})}}{\partial X_{ij}} = 2r^{\prime}(F) + \frac{r(F)}{F}(q\mathrm{tr}(\bm{S}\bm{S}^{+}) - 2)$, this proves Part~$(i)$.\\
        $(ii)$ \quad From Part $(i)$ of \Cref{Proposition2},  we have
        \begin{flalign*}
            \mathrm{E}_{\scriptscriptstyle \bm{\theta}}\left[\mathrm{tr}(g^{\top}(\bm{X},\bm{S})\bm{\Sigma}^{-1}(\bm{X}-\bm{\theta}))\right]
            =\mathrm{E}_{\scriptscriptstyle \bm{\theta}}\left[\mathrm{tr}(\nabla_Xg^{\top}(\bm{X},\bm{S}))\right],
        \end{flalign*}
        and from Part $(i)$, we have
        \begin{flalign*}
            \mathrm{tr}(\nabla_X g^{\top}(\bm{X},\bm{S}))=2r^{\prime}(F)+(q\mathrm{tr}(\bm{S}\bm{S}^{+})-2)r(F)/F,
        \end{flalign*}
        and then,
            $\mathrm{E}_{\scriptscriptstyle \bm{\theta}}\left[\mathrm{tr}\left(g^{\top}(\bm{X},\bm{S})\bm{\Sigma}^{-1}(\bm{X}-\bm{\theta})\right)\right]
            =\mathrm{E}_{\scriptscriptstyle \bm{\theta}}\left[2r^{\prime}(F)+r(F)(q\mathrm{tr}(\bm{S}\bm{S}^{+})-2)/F\right]$,
        this proves Part~$(ii)$, and
        which completes the proof.
    \end{proof}

\begin{proposition}\label{Proposition3}
Suppose that the conditions of \Cref{Theorem2} hold.
Then


    \begin{flalign*}
        (i)\quad \mathrm{tr}\left(g^{\top}(\bm{X},\bm{S})\bm{\Sigma}^{-1}g(\bm{X},\bm{S})\right)=\mathrm{tr}\left(\bm{\Sigma}^{-1}\bm{S}\bm{G}\right);&&
    \end{flalign*}
    \begin{flalign*}
        (ii)\quad q\,\mathrm{tr}(\bm{\Sigma}^{-1}\bm{S}\bm{G})=\mathrm{tr}(\Tilde{\bm{S}}\bm{H});&&
    \end{flalign*}
    \begin{flalign*}
        (iii) \quad \mathrm{tr}(\Tilde{\bm{S}}\bm{H})=\mathrm{vec}(\Tilde{\bm{Y}}) \cdot \mathrm{vec}(\Tilde{\bm{Y}}\bm{H}).&&
    \end{flalign*}
    \end{proposition}
 The proof follows from algebraic computations.
    \begin{proposition}\label{Proposition3bis}
Suppose that the conditions of \Cref{Theorem2} hold and suppose that  $\textnormal{E}_{\scriptscriptstyle \bm{\theta}}[|\mathrm{div}_{\mathrm{vec}(\Tilde{\bm{Y}})} \mathrm{vec}(\Tilde{\bm{Y}}\bm{H})|]<\infty$.
Then

%
 \begin{flalign*}
        (i)\quad \mathrm{E}_{\scriptscriptstyle \bm{\theta}}[\mathrm{vec}(\Tilde{\bm{Y}})\cdot \mathrm{vec}(\Tilde{\bm{Y}}\bm{H})]=\mathrm{E}_{\scriptscriptstyle \bm{\theta}}[\mathrm{div}_{\mathrm{vec}(\Tilde{\bm{Y}})} \mathrm{vec}(\Tilde{\bm{Y}}\bm{H})]; &&
    \end{flalign*}
    \begin{flalign*}
        (ii) \quad \mathrm{E}_{\scriptscriptstyle \bm{\theta}}\left[\mathrm{tr}\left(g^{\top}(\bm{X},\bm{S})\bm{\Sigma}^{-1}g(\bm{X},\bm{S})\right)\right]
        =\mathrm{E}_{\scriptscriptstyle \bm{\theta}}[\mathrm{div}_{\mathrm{vec}(\Tilde{\bm{Y}})}\mathrm{vec}(\Tilde{\bm{Y}}\bm{H})]\Big/q. 
        &&
    \end{flalign*}
    \end{proposition}
          \begin{proof}
          $(i)$ 
          \text{ Since }  $\bm{Y}\sim \mathcal{N}_{nq\times p}(0,\bm{I}_{nq}\otimes (\bm{\Sigma}/q))$, we have $\Tilde{\bm{Y}}=\sqrt{q}\bm{Y}\bm{A}^{-1} \sim \mathcal{N}_{nq\times p}(0,\bm{I}_{nq}\otimes \bm{I}_{p})$. 
Then,
             $\mathrm{vec}(\Tilde{\bm{Y}}) \sim  \mathcal{N}_{npq}(0,\bm{I}_{npq})$.
        Therefore, 
            $\Tilde{Y}_{\alpha i} \sim \mathcal{N}(0,1)$.
        Also, we have
        \begin{flalign*}
            \mathrm{vec}(\Tilde{\bm{Y}})\cdot \mathrm{vec}(\Tilde{\bm{Y}}\bm{H})=\sum_{\alpha,i}^{}{\Tilde{Y}_{\alpha i}(\Tilde{\bm{Y}}\bm{H})_{\alpha i}}=\sum_{\alpha,i}^{}{\Tilde{Y}_{\alpha i}\sum_{j}^{}{\Tilde{\bm{Y}}_{\alpha j}H_{j i}}}=\sum_{\alpha,i,j}^{}{\Tilde{Y}_{\alpha i}\Tilde{\bm{Y}}_{\alpha j}H_{j i}}.&&
        \end{flalign*}
        Therefore, we have
        \begin{flalign*}
            \mathrm{E}_{\scriptscriptstyle \bm{\theta}}[\mathrm{vec}(\Tilde{\bm{Y}})\cdot \mathrm{vec}(\Tilde{\bm{Y}}\bm{H})]=\mathrm{E}_{\scriptscriptstyle \bm{\theta}}\left[\sum_{\alpha,i,j}^{}{\Tilde{Y}_{\alpha i}\Tilde{\bm{Y}}_{\alpha j}H_{j i}}\right]&=\sum_{\alpha,i,j}^{}{\mathrm{E}_{\scriptscriptstyle \bm{\theta}}\left[\Tilde{Y}_{\alpha i}\Tilde{\bm{Y}}_{\alpha j}H_{j i}\right]} 
            =\sum_{\alpha,i,j}^{}{\mathrm{E}_{\scriptscriptstyle \bm{\theta}}\left[\Tilde{Y}_{\alpha i}g_{j}(\Tilde{Y}_{\alpha i})\right]},&&
        \end{flalign*}
        where $g_{j}(\Tilde{Y}_{\alpha i})=\Tilde{Y}_{\alpha j}H_{j i}$. 
Hence, by Lemma 1 of \cite{SteinCharles}, we get
\begin{flalign*}
    \sum_{\alpha,i,j}^{}{\mathrm{E}_{\scriptscriptstyle \bm{\theta}}[\Tilde{Y}_{\alpha i}g_{j}(\Tilde{Y}_{\alpha i})]}=\sum_{\alpha,i,j}^{}{\mathrm{E}_{\scriptscriptstyle \bm{\theta}}\left[\frac{\partial}{\partial \Tilde{Y}_{\alpha i}}g_{j}(\Tilde{Y}_{\alpha i})\right]} 
    =\mathrm{E}_{\scriptscriptstyle \bm{\theta}}\left[\sum_{\alpha,i,j}^{}{\frac{\partial}{\partial \Tilde{Y}_{\alpha i}}g_{j}(\Tilde{Y}_{\alpha i})}\right] 
    =\mathrm{E}_{\scriptscriptstyle \bm{\theta}}\left[\sum_{\alpha,i,j}^{}{\frac{\partial}{\partial \Tilde{Y}_{\alpha i}}\Tilde{\bm{Y}}_{\alpha j}H_{j i}}\right].
\end{flalign*}
Then,
\begin{eqnarray*}
    \sum_{\alpha,i,j}^{}{\mathrm{E}_{\scriptscriptstyle \bm{\theta}}[\Tilde{Y}_{\alpha i}g_{j}(\Tilde{Y}_{\alpha i})]}
    =\mathrm{E}_{\scriptscriptstyle \bm{\theta}}\left[\sum_{\alpha,i}^{}{\frac{\partial}{\partial \Tilde{Y}_{\alpha i}}\sum_{j}^{}{\Tilde{\bm{Y}}_{\alpha j}H_{j i}}}\right] 
    =\mathrm{E}_{\scriptscriptstyle \bm{\theta}}\left[\sum_{\alpha,i}^{}{\frac{\partial}{\partial \Tilde{Y}_{\alpha i}}(\Tilde{\bm{Y}}\bm{H})_{\alpha i}}\right] 
    =\mathrm{E}_{\scriptscriptstyle \bm{\theta}}\left[\nabla_{\Tilde{\bm{Y}}} (\Tilde{\bm{Y}}\bm{H})\right],
\end{eqnarray*}
and then,
\begin{eqnarray*}
    \sum_{\alpha,i,j}^{}{\mathrm{E}_{\scriptscriptstyle \bm{\theta}}[\Tilde{Y}_{\alpha i}g_{j}(\Tilde{Y}_{\alpha i})]}
    =\mathrm{E}_{\scriptscriptstyle \bm{\theta}}[\mathrm{div}_{\mathrm{vec}(\Tilde{\bm{Y}})}\mathrm{vec}(\Tilde{\bm{Y}}\bm{H})].&&
\end{eqnarray*}
$(ii)$ \quad From Part $(i)$ to $(iv)$, we have
\begin{eqnarray*}
    \mathrm{E}_{\scriptscriptstyle \bm{\theta}}\left[\mathrm{tr}\left(g^{\top}(\bm{X},\bm{S})\bm{\Sigma}^{-1}g(\bm{X},\bm{S})\right)\right] 
    =\mathrm{E}_{\scriptscriptstyle \bm{\theta}}\left[\mathrm{tr}(\bm{\Sigma}^{-1}\bm{S}\bm{G})\right]
    =\mathrm{E}_{\scriptscriptstyle \bm{\theta}}\left[\mathrm{tr}(\Tilde{\bm{S}}\bm{H})\right]\Big /q 
    =\mathrm{E}_{\scriptscriptstyle \bm{\theta}}\left[\mathrm{vec}(\Tilde{\bm{Y}})\cdot \mathrm{vec}(\Tilde{\bm{Y}}\bm{H})\right]\Big /q,
\end{eqnarray*}
and then,
\begin{eqnarray*}
    \mathrm{E}_{\scriptscriptstyle \bm{\theta}}\left[\mathrm{tr}\left(g^{\top}(\bm{X},\bm{S})\bm{\Sigma}^{-1}g(\bm{X},\bm{S})\right)\right]
    =\mathrm{E}_{\scriptscriptstyle \bm{\theta}}\left[\mathrm{div}_{\mathrm{vec}(\Tilde{\bm{Y}})} \mathrm{vec}(\Tilde{\bm{Y}}\bm{H})\right]\Big / q,&&
\end{eqnarray*}
which completes the proof.
        \end{proof}

\begin{proposition}\label{Proposition4}
   Suppose that the conditions of \Cref{Theorem2} hold.
    Then,

   \begin{flalign*}
       \nabla_{\Tilde{\bm{Y}}} \cdot (\Tilde{\bm{Y}}\bm{H}) &= \mathrm{div}_{\mathrm{vec}(\Tilde{\bm{Y}})}\mathrm{vec}(\Tilde{\bm{Y}}\bm{H}) = nq\mathrm{tr}(\bm{G}) + \mathrm{tr}(\bm{Y}^{\top}(\nabla_{Y}\bm{G}^{\top}))
       = nq\mathrm{tr}(\bm{G}) + \sum_{\alpha,\beta,k}^{}{Y_{\alpha k}\frac{\partial G_{k \beta}}{\partial Y_{\alpha \beta}}}.
   \end{flalign*}
\end{proposition}
  \begin{proof} 
  We have $\nabla_{\Tilde{\bm{Y}}} \cdot (\Tilde{\bm{Y}}\bm{H})= \mathrm{div}_{\textnormal{vec}(\Tilde{\bm{Y}})} \cdot \mathrm{vec}(\Tilde{\bm{Y}}\bm{H})$, then
        \begin{flalign*}
             \nabla_{\Tilde{\bm{Y}}} \cdot (\Tilde{\bm{Y}}\bm{H}) & 
             = \sum_{\alpha,i}^{}{(\mathrm{div}_{\Tilde{\bm{Y}}})_{\alpha i}(\Tilde{\bm{Y}}\bm{H})_{\alpha i}} 
             =\sum_{\alpha,i}^{}{\frac{\partial}{\partial \Tilde{Y}_{\alpha i}}\sum_{j}^{}{\Tilde{\bm{Y}}_{\alpha j}H_{ji}}} 
             =\sum_{\alpha,i,j}^{}{\frac{\partial}{\partial \Tilde{Y}_{\alpha i}}(\Tilde{\bm{Y}}_{\alpha j}H_{ji})}. 
        \end{flalign*}
         Let $\mathfrak{J}_{1}= \ds{\sum_{\alpha,i,j}^{}}{\left(\frac{\partial}{\partial \Tilde{Y}_{\alpha i}}\Tilde{\bm{Y}}_{\alpha j}\right)H_{ji}}+\ds{\sum_{\alpha,i,j}^{}}{\Tilde{\bm{Y}}_{\alpha j}\left(\frac{\partial}{\partial \Tilde{Y}_{\alpha i}}H_{ji}\right)}$. The previous result yields
        \begin{eqnarray}
                   \quad{ } \quad{ }  \nabla_{\Tilde{\bm{Y}}} \cdot (\Tilde{\bm{Y}}\bm{H})  =\sum_{\alpha,i,j}^{}{\left(\left(\frac{\partial}{\partial \Tilde{Y}_{\alpha i}}\Tilde{\bm{Y}}_{\alpha j}\right)H_{ji}+ \Tilde{\bm{Y}}_{\alpha j}\left(\frac{\partial}{\partial \Tilde{Y}_{\alpha i}}H_{ji}\right)\right )},
             \label{p4i}
        \end{eqnarray}
        then
              \begin{eqnarray}
                   \nabla_{\Tilde{\bm{Y}}} \cdot \left(\Tilde{\bm{Y}}\bm{H}\right)
             =\sum_{\alpha,i,j}^{}{\left(\frac{\partial}{\partial \Tilde{Y}_{\alpha i}}\Tilde{\bm{Y}}_{\alpha j}\right)H_{ji}}+\sum_{\alpha,i,j}^{}{\Tilde{\bm{Y}}_{\alpha j}\left(\frac{\partial}{\partial \Tilde{Y}_{\alpha i}}H_{ji}\right)}.
             \label{p4i}
        \end{eqnarray}
      Then, by applying the chain rule in \eqref{p4i}, we get
        \begin{eqnarray*}
             & \mathfrak{J}_{1}=\ds{\sum_{\alpha,i,j}^{}}{\delta_{ij}H_{ji}}+\ds{\sum_{\alpha,i,j}^{}}{\Tilde{\bm{Y}}_{\alpha j}\ds{\sum_{k,\beta}^{}}{\frac{\partial}{\partial Y_{k \beta}}H_{ji}\frac{\partial Y_{k \beta}}{\partial \Tilde{Y}_{\alpha i}}}} 
             =\ds{\sum_{\alpha,i}^{}}{H_{ii}}+\ds{\sum_{\alpha,i,j,k,\beta}^{}}{\Tilde{\bm{Y}}_{\alpha j}\frac{\partial}{\partial Y_{k \beta}}H_{ji}\frac{\partial (\Tilde{\bm{Y}}A)_{k \beta}}{\partial \Tilde{Y}_{\alpha i}}}\Big /\sqrt{q}\nonumber\\
             &=\ds{\sum_{\alpha}^{}}{\mathrm{tr}(\bm{H})}+\sum_{\alpha,i,j,k,\beta}^{}{\Tilde{\bm{Y}}_{\alpha j}\frac{\partial}{\partial Y_{k \beta}}H_{ji}\left(\frac{\partial}{\partial \Tilde{Y}_{\alpha i}}\ds{\sum_{l}^{}}{\Tilde{\bm{Y}}_{kl}A_{l \beta}}\right)}\Big /\sqrt{q} 
             =nq\mathrm{tr}(\bm{H})+\ds{\sum_{\alpha,i,j,k,\beta,l}^{}}{\Tilde{\bm{Y}}_{\alpha j}\frac{\partial}{\partial Y_{k \beta}}H_{ji}\left(\frac{\partial \Tilde{\bm{Y}}_{kl}}{\partial \Tilde{Y}_{\alpha i}}\right)A_{l \beta}}\Big /\sqrt{q}.
        \end{eqnarray*}
        Then,
          \begin{eqnarray*}
             & \mathfrak{J}_{1}
             =nq\mathrm{tr}(\bm{A}\bm{G}\bm{A}^{-1})+\ds{\sum_{\alpha,i,j,k,\beta,l}^{}}{\Tilde{\bm{Y}}_{\alpha j}\frac{\partial}{\partial Y_{k \beta}}H_{ji}(\delta_{\alpha k}\delta_{il})A_{l \beta}}\Big /\sqrt{q} 
             =nq\mathrm{tr}(\bm{A}^{-1}AG)+\ds{\sum_{\alpha,i,j,\beta}^{}}{\Tilde{\bm{Y}}_{\alpha j}\frac{\partial}{\partial Y_{\alpha \beta}}H_{ji}A_{i \beta}}\Big /\sqrt{q}\nonumber\\
             &\nonumber=nq\mathrm{tr}(\bm{G})+\ds{\sum_{\alpha,i,j,\beta}^{}}{\Tilde{\bm{Y}}_{\alpha j}\frac{\partial}{\partial Y_{\alpha \beta}}(\bm{A}\bm{G}\bm{A}^{-1})_{ji}A_{i \beta}}\Big /\sqrt{q}
             =nq\mathrm{tr}(\bm{G})+\ds{\sum_{\alpha,i,j,\beta}^{}}{\Tilde{\bm{Y}}_{\alpha j}\frac{\partial}{\partial Y_{\alpha\beta}}\left(\ds{\sum_{k,l}^{}}{A_{jk}G_{kl}A_{li}^{-1}}\right)A_{i \beta}}\Big /\sqrt{q}
        \end{eqnarray*}
        and then, $\mathfrak{J}_{1}
             =nq\mathrm{tr}(\bm{G})+\ds{\sum_{\alpha,i,j,\beta,k,l}^{}}{\Tilde{\bm{Y}}_{\alpha j}A_{jk}\frac{\partial}{\partial Y_{\alpha\beta}}G_{kl}A_{li}^{-1}}A_{i \beta}\Big /\sqrt{q}$. 
             This gives,
                  \begin{eqnarray*}
             & \mathfrak{J}_{1}
             =nq\mathrm{tr}(\bm{G})+\ds{\sum_{\alpha,\beta,k,l}^{}}\left(\sum_{j}^{}{\Tilde{\bm{Y}}_{\alpha j}A_{jk}}\right)\frac{\partial G_{kl}}{\partial Y_{\alpha \beta}}\left(\ds{\sum_{i}^{}}{A_{li}^{-1}}A_{i \beta}\right)\Big /\sqrt{q} 
             &= nq\mathrm{tr}(\bm{G})+\ds{\sum_{\alpha,\beta,k,l}^{}}{(\Tilde{\bm{Y}}A)_{\alpha k}\frac{\partial G_{kl}}{\partial Y_{\alpha \beta}}(\bm{A}^{-1}A)_{l \beta}}\Big /\sqrt{q},
        \end{eqnarray*}
        and then,
                          \begin{eqnarray}
             & \mathfrak{J}_{1}
             = nq\mathrm{tr}(\bm{G})+\ds{\sum_{\alpha,\beta,k}^{}}{Y_{\alpha k}\frac{\partial G_{k \beta}}{\partial Y_{\alpha \beta}}}.
             \label{p4ii}
        \end{eqnarray}
        Also, we have $\mathrm{tr}(\bm{Y}^{\top}(\nabla_{Y}\bm{G}^{\top})) = \ds{\sum_{k}^{}}{(\bm{Y}^{\top}(\nabla_{Y}\bm{G}^{\top}))_{kk}}$, then
        \begin{flalign*}
            \nonumber &\mathrm{tr}(\bm{Y}^{\top}(\nabla_{Y}\bm{G}^{\top}))
            =\sum_{k,\alpha}^{}{\bm{Y}^{\top}_{k \alpha}(\nabla_{Y}\bm{G}^{\top})_{\alpha k}}
            =\sum_{k,\alpha}^{}{\bm{Y}^{\top}_{k \alpha}\sum_{\beta}^{}{(\nabla_{Y})_{\alpha \beta}\bm{G}^{\top}_{\beta k}}}
            =\sum_{\alpha,\beta,k}^{}{\bm{Y}^{\top}_{k \alpha}\frac{\partial \bm{G}^{\top}_{\beta k}}{\partial Y_{\alpha \beta}}}
        \end{flalign*}
        this gives,
              \begin{flalign}
            \mathrm{tr}(\bm{Y}^{\top}(\nabla_{Y}\bm{G}^{\top}))
            &=\sum_{\alpha,\beta,k}^{}{Y_{\alpha k}\frac{\partial G_{k \beta}}{\partial Y_{\alpha \beta}}},
            &&\label{p4iii}
        \end{flalign}
and then, the proof follows from the relations \eqref{p4ii} and \eqref{p4iii},
        this  completes the proof.
   \end{proof}
For the sake of simplicity, in the sequel, we write $\bm{I}$ to stand for $\bm{I}_{p}$.
\begin{lemma}\label{Lemma1}
 Let $\bm{Y}$, $\bm{X}$, $\bm{S}$ and $F$ be as in \Cref{Theorem2}.
 Let $\bm{A} \in \mathbb{R}^{ k\times p}$ and $\bm{B} \in \mathbb{R}^{p\times h}$. Then, 
    \begin{flalign*}
           (i)\quad \left(\bm{A}\frac{\partial \bm{S}}{\partial Y_{\alpha \beta}}\bm{B}\right)_{kl}=A_{k\beta}(\bm{Y}\bm{B})_{\alpha l}+(\bm{A}\bm{Y}^{\top})_{k\alpha}B_{\beta l};&&
    \end{flalign*}

  \begin{flalign*}
        (ii)\quad \left(\frac{\partial \bm{X}^{\top}\bm{S}^{+}\bm{X}}{\partial Y_{\alpha \beta}} \right)_{kk} =&-2(\bm{X}^{\top}\bm{S}^{+}\bm{Y}^{\top})_{k \alpha}(\bm{S}^{+}\bm{X})_{\beta k}
        +2(\bm{X}^{\top}\bm{S}^{+}\bm{S}^{+}\bm{Y}^{\top})_{k \alpha}((\bm{I}-\bm{S}\bm{S}^{+})\bm{X})_{\beta k};&&
  \end{flalign*}
    \begin{flalign*}
        (iii)\quad \frac{\partial F}{\partial Y_{\alpha \beta}}&=-2(\bm{S}^{+}\bm{X}\bm{X}^{\top}\bm{S}^{+}\bm{Y}^{\top})_{\beta \alpha}+2((\bm{I}-\bm{S}\bm{S}^{+})\bm{X}\bm{X}^{\top}\bm{S}^{+}\bm{S}^{+}\bm{Y}^{\top})_{\beta \alpha};&&
    \end{flalign*}
$(iv)$
\begin{eqnarray*}
      & \bigg\{ \frac{\partial \bm{S}^{+}\bm{X}\bm{X}^{\top}\bm{S}\bm{S}^{+}}{\partial Y_{\alpha \beta}}\bigg\}_{kl}
     =-\bm{S}^{+}_{k \beta}(\bm{Y}\bm{S}^{+}\bm{X}\bm{X}^{\top}\bm{S}\bm{S}^{+})_{\alpha l}- (\bm{S}^{+}\bm{Y}^{\top})_{k \alpha}(\bm{S}^{+}\bm{X}\bm{X}^{\top}\bm{S}\bm{S}^{+})_{\beta l}\\
     & +(\bm{I}-\bm{S}\bm{S}^{+})_{k \beta}(\bm{Y}\bm{S}^{+}\bm{S}\bm{X}\bm{X}^{\top}\bm{S}\bm{S}^{+})_{\alpha l} 
     +(\bm{S}^{+}\bm{S}^{+}\bm{Y}^{\top})_{k \alpha}((\bm{I}-\bm{S}\bm{S}^{+})\bm{X}\bm{X}^{\top}\bm{S}\bm{S}^{+})_{\beta l}\\
     &
     +(\bm{S}^{+}\bm{X}\bm{X}^{\top})_{k \beta}(\bm{Y}\bm{S}^{+})_{\alpha l}+(\bm{S}^{+}\bm{X}\bm{X}^{\top}\bm{Y}^{\top})_{k \alpha}\bm{S}^{+}_{\beta l}
     -(\bm{S}^{+}\bm{X}\bm{X}^{\top}\bm{S}\bm{S}^{+})_{k \beta}(\bm{Y}\bm{S}^{+})_{\alpha l} \\
     &
      -(\bm{S}^{+}\bm{X}\bm{X}^{\top}\bm{S}\bm{S}^{+}\bm{Y}^{\top})_{k \alpha}\bm{S}^{+}_{\beta l} 
     +(\bm{S}^{+}\bm{X}\bm{X}^{\top}\bm{S}^{+}\bm{Y}^{\top})_{k \alpha}(\bm{I}-\bm{S}\bm{S}^{+})_{\beta l}. 
\end{eqnarray*}
\end{lemma}

\begin{proof}
    $(i)$ \quad By Proposition 1 of \cite{ChetelatWells},  we have \\$\frac{\partial \bm{S}}{\partial Y_{\alpha \beta}}=\delta_{\beta i}Y_{\alpha j}+\delta_{\beta j}Y_{\alpha i}$ where $\delta_{ij}$ is Kronecker delta. Therefore, 
     \begin{flalign*}
    &\left(\bm{A}\frac{\partial \bm{S}}{\partial Y_{\alpha \beta}}\bm{B}\right)_{kl} =\sum_{j}^{}{\left(\bm{A}\frac{\partial \bm{S}}{\partial Y_{\alpha \beta}}\right)_{kj}B_{jl}}=
     \sum_{j}^{}{\left(\sum_{i}^{}{A_{ki}\left(\frac{\partial \bm{S}}{\partial Y_{\alpha \beta}} \right)_{ij}}\right )B_{jl}}\\
     &=\sum_{j}^{}{\left(\sum_{i}^{}{A_{ki}\bigg\{\delta_{\beta i}Y_{\alpha j}+\delta_{\beta j}Y_{\alpha i} \bigg\}}\right )B_{jl}}
     =\sum_{j}^{}{\left ( \sum_{i}^{}{A_{ki}\delta_{\beta i}Y_{\alpha j}}+\sum_{i}^{}{A_{ki}\delta_{\beta j}Y_{\alpha i}}\right )B_{jl}}
     \\&
     =\sum_{j}^{}{\left ( A_{k\beta}Y_{\alpha j}+\sum_{i}^{}{A_{ki}\delta_{\beta j}Y_{\alpha i}}\right )B_{jl}}
     =\sum_{j}^{}{A_{k \beta}Y_{\alpha j}B_{jl}}+\sum_{j}^{}{\left ( \sum_{i}^{}{A_{ki}\delta_{\beta j}Y_{\alpha i}}\right)B_{jl}}\\
     &
     = A_{k \beta}\sum_{j}^{}{Y_{\alpha j}B_{jl}}+\sum_{i}^{}{A_{ki}Y_{\alpha i}\sum_{j}^{}{\delta_{\beta j}B_{jl}}}
     = A_{k \beta}(\bm{Y}\bm{B})_{\alpha l}+\sum_{i}^{}{A_{ki}Y_{\alpha i}B_{\beta l}}. 
&&
 \end{flalign*}
Then, $\left(\bm{A}\frac{\partial \bm{S}}{\partial Y_{\alpha \beta}}\bm{B}\right)_{kl}
     = A_{k \beta}(\bm{Y}\bm{B})_{\alpha l}+B_{\beta l}\ds{\sum_{i}^{p}}{A_{ki}\bm{Y}^{\top}_{i \alpha}}$, and then
   \begin{flalign*}
    &\left(\bm{A}\frac{\partial \bm{S}}{\partial Y_{\alpha \beta}}\bm{B}\right)_{kl}
     = A_{k \beta}(\bm{Y}\bm{B})_{\alpha l}+B_{\beta l}(\bm{A}\bm{Y}^{\top})_{k \alpha} 
     = A_{k\beta}(\bm{Y}\bm{B})_{\alpha l}+(\bm{A}\bm{Y}^{\top})_{k\alpha}B_{\beta l},&&
 \end{flalign*}
 this proves Part~$(i)$.\\
 $(ii)$ We have
 \begin{eqnarray*}
             \left(\frac{\partial}{\partial Y_{\alpha \beta}}(\bm{X}^{\top}\bm{S}^{+}\bm{X})\right)_{kk}=\left(\bm{X}^{\top}\left(\frac{\partial \bm{S}^{+}}{\partial Y_{\alpha \beta}}\right)\bm{X} \right)_{kk}.&&
        \end{eqnarray*}
        Further, from Theorem~4.3 in \cite{GolubPereyra},  we get
       \begin{flalign*}
            &\left(\bm{X}^{\top}\left(\frac{\partial \bm{S}^{+}}{\partial Y_{\alpha \beta}}\right)\bm{X} \right)_{kk} 
            =\left(\bm{X}^{\top}\left(-\bm{S}^{+}\frac{\partial \bm{S}}{\partial Y_{\alpha \beta}}\bm{S}^{+}+(\bm{I}-\bm{S}\bm{S}^{+})\frac{\partial \bm{S}}{\partial Y_{\alpha \beta}}\bm{S}^{+}\bm{S}^{+}+\bm{S}^{+}\bm{S}^{+}\frac{\partial \bm{S}}{\partial Y_{\alpha \beta}}(\bm{I}-\bm{S}\bm{S}^{+})\right)\bm{X}\right)_{kk}\\
            &= \left(-\bm{X}^{\top}\bm{S}^{+}\frac{\partial \bm{S}}{\partial Y_{\alpha \beta}}\bm{S}^{+}\bm{X}+\bm{X}^{\top}(\bm{I}-\bm{S}\bm{S}^{+})\frac{\partial \bm{S}}{\partial Y_{\alpha \beta}}\bm{S}^{+}\bm{S}^{+}\bm{X}+\bm{X}^{\top}\bm{S}^{+}\bm{S}^{+}\frac{\partial \bm{S}}{\partial Y_{\alpha \beta}}(\bm{I}-\bm{S}\bm{S}^{+})\bm{X}\right)_{kk}\\
            &=-\left(\bm{X}^{\top}\bm{S}^{+}\frac{\partial \bm{S}}{\partial Y_{\alpha \beta}}\bm{S}^{+}\bm{X}\right)_{kk}+\left(\bm{X}^{\top}(\bm{I}-\bm{S}\bm{S}^{+})\frac{\partial \bm{S}}{\partial Y_{\alpha \beta}}\bm{S}^{+}\bm{S}^{+}\bm{X}\right)_{kk}+\left(\bm{X}^{\top}\bm{S}^{+}\bm{S}^{+}\frac{\partial \bm{S}}{\partial Y_{\alpha \beta}}(\bm{I}-\bm{S}\bm{S}^{+})\bm{X}\right)_{kk}.&&
        \end{flalign*}
        Now, by Part $(i)$, we get
        \begin{flalign*}
            &\left(\frac{\partial \bm{X}^{\top}\bm{S}^{+}\bm{X}}{\partial Y_{\alpha \beta}} \right)_{kk}=-(\bm{X}^{\top}\bm{S}^{+})_{k \beta}(\bm{Y}\bm{S}^{+}\bm{X})_{\alpha k} - (\bm{X}^{\top}\bm{S}^{+}\bm{Y}^{\top})_{k \alpha}(\bm{S}^{+}\bm{X})_{\beta k}\\
            &+(\bm{X}^{\top}(\bm{I}-\bm{S}\bm{S}^{+}))_{k \beta}(\bm{Y}\bm{S}\bm{S}^{+}\bm{X})_{\alpha k} 
            + (\bm{X}^{\top}(\bm{I}-\bm{S}\bm{S}^{+})\bm{Y}^{\top})_{k \alpha}(\bm{S}\bm{S}^{+}\bm{X})_{\beta k}\\
            &+(\bm{X}^{\top}\bm{S}^{+}\bm{S}^{+})_{k \beta}(\bm{Y}(\bm{I}-\bm{S}\bm{S}^{+})\bm{X})_{\alpha k}
            + (\bm{X}^{\top}\bm{S}^{+}\bm{S}^{+}\bm{Y}^{\top})_{k \alpha}((\bm{I}-\bm{S}\bm{S}^{+})\bm{X})_{\beta k}.
            &&
        \end{flalign*}
       Further, we have 
        \begin{eqnarray*}
            (\bm{X}^{\top}\bm{S}^{+})_{k \beta}(\bm{Y}\bm{S}^{+}\bm{X})_{\alpha k}&=&(\bm{X}^{\top}\bm{S}^{+}\bm{Y}^{\top})_{k \alpha}(\bm{S}^{+}\bm{X})_{\beta k};\\
             (\bm{X}^{\top}(\bm{I}-\bm{S}\bm{S}^{+}))_{k \beta}(\bm{Y}\bm{S}\bm{S}^{+}\bm{X})_{\alpha k}&=&(\bm{X}^{\top}\bm{S}^{+}\bm{S}^{+}\bm{Y}^{\top})_{k \alpha}((\bm{I}-\bm{S}\bm{S}^{+})\bm{X})_{\beta k};\\
             \bm{Y}(\bm{I}-\bm{S}\bm{S}^{+})=(\bm{I}-\bm{S}\bm{S}^{+})\bm{Y}^{\top}&=&0.
        \end{eqnarray*}
        Then,
        \begin{flalign*}
             \left(\frac{\partial \bm{X}^{\top}\bm{S}^{+}\bm{X}}{\partial Y_{\alpha \beta}} \right)_{kk}=-2(\bm{X}^{\top}\bm{S}^{+}\bm{Y}^{\top})_{k \alpha}(\bm{S}^{+}\bm{X})_{\beta k}+2(\bm{X}^{\top}\bm{S}^{+}\bm{S}^{+}\bm{Y}^{\top})_{k \alpha}((\bm{I}-\bm{S}\bm{S}^{+})\bm{X})_{\beta k}&&
        \end{flalign*}
        this proves Part~$(ii)$.\\
        $(iii)$ \quad By   Part $(ii)$, we get
        \begin{flalign*}
            \frac{\partial F}{\partial Y_{\alpha \beta}}&=\frac{\partial }{\partial Y_{\alpha \beta}}\sum_{k}^{}{(\bm{X}^{\top}\bm{S}^{+}\bm{X})_{kk}}
            =\sum_{k}^{}{\frac{\partial (\bm{X}^{\top}\bm{S}^{+}\bm{X})_{kk}}{\partial Y_{\alpha \beta}}}
            \\&=\sum_{k}^{}{\bigg\{-2(\bm{X}^{\top}\bm{S}^{+}\bm{Y}^{\top})_{k \alpha}(\bm{S}^{+}\bm{X})_{\beta k}+2(\bm{X}^{\top}\bm{S}^{+}\bm{S}^{+}\bm{Y}^{\top})_{k \alpha}((\bm{I}-\bm{S}\bm{S}^{+})\bm{X})_{\beta k} \bigg\}}. 
&&
        \end{flalign*}
Therefore,
       \begin{flalign*}
            \frac{\partial F}{\partial Y_{\alpha \beta}}=-2\left(\bm{S}^{+}\bm{X}\bm{X}^{\top}\bm{S}^{+}\bm{Y}^{\top}\right)_{\beta \alpha}+2\left((\bm{I}-\bm{S}\bm{S}^{+})\bm{X}\bm{X}^{\top}\bm{S}^{+}\bm{S}^{+}\bm{Y}^{\top}\right)_{\beta \alpha},&&
        \end{flalign*}
        this proves Part~(iii).\\
        $(iv)$ We have
        \begin{flalign*}
    &\left( \frac{\partial \bm{S}^{+}\bm{X}\bm{X}^{\top}\bm{S}\bm{S}^{+}}{\partial Y_{\alpha \beta}}\right)_{kl} 
    =\left(\frac{\partial \bm{S}^{+}}{\partial Y_{\alpha \beta}}\bm{X}\bm{X}^{\top}\bm{S}\bm{S}^{+}\right)_{kl}+\left(\bm{S}^{+}\bm{X}\bm{X}^{\top}\frac{\partial \bm{S}}{\partial Y_{\alpha \beta}}\bm{S}^{+}\right)_{kl}+\left(\bm{S}^{+}\bm{X}\bm{X}^{\top}\bm{S}\frac{\partial \bm{S}^{+}}{\partial Y_{\alpha \beta}}\right)_{kl}\\
    &=\left(\left(-\bm{S}^{+}\frac{\partial \bm{S}}{\partial Y_{\alpha \beta}}\bm{S}^{+}+(\bm{I}-\bm{S}\bm{S}^{+})\frac{\partial \bm{S}}{\partial Y_{\alpha \beta}}\bm{S}^{+}\bm{S}^{+}+\bm{S}^{+}\bm{S}^{+}\frac{\partial \bm{S}}{\partial Y_{\alpha \beta}}(\bm{I}-\bm{S}\bm{S}^{+})\right)\bm{X}\bm{X}^{\top}\bm{S}\bm{S}^{+}\right)_{kl}
    +\left(\bm{S}^{+}\bm{X}\bm{X}^{\top}\frac{\partial \bm{S}}{\partial Y_{\alpha \beta}}\bm{S}^{+}\right)_{kl}\\&
    +\left(\bm{S}^{+}\bm{X}\bm{X}^{\top}\bm{S}(-\bm{S}^{+}\frac{\partial \bm{S}}{\partial Y_{\alpha \beta}}\bm{S}^{+}+(\bm{I}-\bm{S}\bm{S}^{+})\frac{\partial \bm{S}}{\partial Y_{\alpha \beta}}\bm{S}^{+}\bm{S}^{+}+\bm{S}^{+}\bm{S}^{+}\frac{\partial \bm{S}}{\partial Y_{\alpha \beta}}(\bm{I}-\bm{S}\bm{S}^{+}))\right)_{kl}.
         &&
\end{flalign*}
Then,
       \begin{flalign*}
    &\left( \frac{\partial \bm{S}^{+}\bm{X}\bm{X}^{\top}\bm{S}\bm{S}^{+}}{\partial Y_{\alpha \beta}}\right)_{kl} 
         =\left( -\bm{S}^{+}\frac{\partial \bm{S}}{\partial Y_{\alpha \beta}}\bm{S}^{+}\bm{X}\bm{X}^{\top}\bm{S}\bm{S}^{+}\right)_{kl}+\left((\bm{I}-\bm{S}\bm{S}^{+})\frac{\partial \bm{S}}{\partial Y_{\alpha \beta}}\bm{S}^{+}\bm{S}^{+}\bm{X}\bm{X}^{\top}\bm{S}\bm{S}^{+}\right)_{kl}\\&
         +\left(
         \bm{S}^{+}\bm{S}^{+}\frac{\partial \bm{S}}{\partial Y_{\alpha \beta}}(\bm{I}-\bm{S}\bm{S}^{+})\bm{X}\bm{X}^{\top}\bm{S}\bm{S}^{+}\right)_{kl}
         +
         \left(\bm{S}^{+}\bm{X}\bm{X}^{\top}\frac{\partial \bm{S}}{\partial Y_{\alpha \beta}}\bm{S}^{+}\right)_{kl} 
         -
         \left(\bm{S}^{+}\bm{X}\bm{X}^{\top}\bm{S}\bm{S}^{+}\frac{\partial \bm{S}}{\partial Y_{\alpha \beta}}\bm{S}^{+}\right)_{kl}\\
         &+\left(
         \bm{S}^{+}\bm{X}\bm{X}^{\top}\bm{S}(\bm{I}-\bm{S}\bm{S}^{+})\frac{\partial \bm{S}}{\partial Y_{\alpha \beta}}\bm{S}^{+}\bm{S}^{+}\right)_{kl}
         +\left(
         \bm{S}^{+}\bm{X}\bm{X}^{\top}\bm{S}\bm{S}^{+}\bm{S}^{+}\frac{\partial \bm{S}}{\partial Y_{\alpha \beta}}(\bm{I}-\bm{S}\bm{S}^{+})\right)_{kl}.&&
\end{flalign*}
 Now, by using Part $(i)$, we get

 \begin{eqnarray}\label{L1p1}
 \quad{ } \quad{ } \left( -\bm{S}^{+}\frac{\partial \bm{S}}{\partial Y_{\alpha \beta}}\bm{S}^{+}\bm{X}\bm{X}^{\top}\bm{S}\bm{S}^{+}\right)_{kl}
 = -\bm{S}^{+}_{k \beta}(\bm{Y}\bm{S}^{+}\bm{X}\bm{X}^{\top}\bm{S}\bm{S}^{+})_{\alpha l}-(\bm{S}^{+}\bm{Y}^{\top})_{k \alpha}(\bm{S}^{+}\bm{X}\bm{X}^{\top}\bm{S}\bm{S}^{+})_{\beta l}; 
\end{eqnarray}

 \begin{flalign}\label{L1p2}
 &\left( (\bm{I}-\bm{S}\bm{S}^{+})\frac{\partial \bm{S}}{\partial Y_{\alpha \beta}}\bm{S}^{+}\bm{S}^{+}\bm{X}\bm{X}^{\top}\bm{S}\bm{S}^{+}\right)_{kl}
 = (\bm{I}-\bm{S}\bm{S}^{+})_{k \beta}(\bm{Y}\bm{S}^{+}\bm{S}^{+}\bm{X}\bm{X}^{\top}\bm{S}\bm{S}^{+})_{\alpha l}\nonumber\\&+((\bm{I}-\bm{S}\bm{S}^{+})\bm{Y}^{\top})_{k \alpha}(\bm{S}^{+}\bm{S}^{+}\bm{X}\bm{X}^{\top}\bm{S}\bm{S}^{+})_{\beta l}
 = (\bm{I}-\bm{S}\bm{S}^{+})_{k \beta}(\bm{Y}\bm{S}^{+}\bm{S}^{+}\bm{X}\bm{X}^{\top}\bm{S}\bm{S}^{+})_{\alpha l}.&&
\end{flalign}
Since $\bm{Y}(\bm{I}-\bm{S}\bm{S}^{+})=0$, we get
 \begin{flalign}\label{L1p3}
 &\left( \bm{S}^{+}\bm{S}^{+}\frac{\partial \bm{S}}{\partial Y_{\alpha \beta}}(\bm{I}-\bm{S}\bm{S}^{+})\bm{X}\bm{X}^{\top}\bm{S}\bm{S}^{+}\right)_{kl}
 = (\bm{S}^{+}\bm{S}^{+})_{k \beta}(\bm{Y}(\bm{I}-\bm{S}\bm{S}^{+})\bm{X}\bm{X}^{\top}\bm{S}\bm{S}^{+})_{\alpha l}\nonumber\\
 &+(\bm{S}^{+}\bm{S}^{+}\bm{Y}^{\top})_{k \alpha}((\bm{I}-\bm{S}\bm{S}^{+})\bm{X}\bm{X}^{\top}\bm{S}\bm{S}^{+})_{\beta l}
 = (\bm{S}^{+}\bm{S}^{+}\bm{Y}^{\top})_{k \alpha}((\bm{I}-\bm{S}\bm{S}^{+})\bm{X}\bm{X}^{\top}\bm{S}\bm{S}^{+})_{\beta l}.&&
\end{flalign}

 \begin{flalign}\label{L1p4}
 &\left( \bm{S}^{+}\bm{X}\bm{X}^{\top}\frac{\partial \bm{S}}{\partial Y_{\alpha \beta}}\bm{S}^{+}\right)_{kl}= (\bm{S}^{+}\bm{X}\bm{X}^{\top})_{k \beta}(\bm{Y}\bm{S}^{+})_{\alpha l}+(\bm{S}^{+}\bm{X}\bm{X}^{\top}\bm{Y}^{\top})_{k \alpha}\bm{S}^{+}_{\beta l}.&&
\end{flalign}

 \begin{flalign*}
 &\left( -\bm{S}^{+}\bm{X}\bm{X}^{\top}\bm{S}\bm{S}^{+}\frac{\partial \bm{S}}{\partial Y_{\alpha \beta}}\bm{S}^{+}\right)_{kl}= -(\bm{S}^{+}\bm{X}\bm{X}^{\top}\bm{S}\bm{S}^{+})_{k \beta}(\bm{Y}\bm{S}^{+})_{\alpha l}(\bm{S}^{+}\bm{X}\bm{X}^{\top}\bm{S}\bm{S}^{+}\bm{Y}^{\top})_{k \alpha}\bm{S}^{+}_{\beta l}.&&
\end{flalign*}
 Since $\bm{S}\bm{S}^{+}\bm{Y}^{\top}=\bm{Y}^{\top}$, we get

 \begin{flalign}\label{L1p5}
 \left(  -\bm{S}^{+}\bm{X}\bm{X}^{\top}\bm{S}\bm{S}^{+}\frac{\partial \bm{S}}{\partial Y_{\alpha \beta}}\bm{S}^{+}\right)_{kl}=-(\bm{S}^{+}\bm{X}\bm{X}^{\top}\bm{S}\bm{S}^{+})_{k \beta}(\bm{Y}\bm{S}^{+})_{\alpha l}(\bm{S}^{+}\bm{X}\bm{X}^{\top}\bm{Y}^{\top})_{k \alpha}\bm{S}^{+}_{\beta l}.&&
\end{flalign}
Since $\bm{S}(\bm{I}-\bm{S}\bm{S}^{+})=\bm{S}^{+}(\bm{I}-\bm{S}\bm{S}^{+})=0$ we get
 \begin{flalign} \label{L1p6}
 \left(  \bm{S}^{+}\bm{X}\bm{X}^{\top}\bm{S}(\bm{I}-\bm{S}\bm{S}^{+})\frac{\partial \bm{S}}{\partial Y_{\alpha \beta}}\bm{S}^{+}\bm{S}^{+}\right)_{kl}=0,
 &&
\end{flalign}
and letting $\bm{A}=\bm{S}^{+}\bm{X}\bm{X}^{\top}\bm{S}^{+}$, $\bm{B}=(\bm{I}-\bm{S}\bm{S}^{+})$, we have
 \begin{flalign*}
 &\left( \bm{S}^{+}\bm{X}\bm{X}^{\top}\bm{S}\bm{S}^{+}\bm{S}^{+}\frac{\partial \bm{S}}{\partial Y_{\alpha \beta}}(\bm{I}-\bm{S}\bm{S}^{+})\right)_{kl}=\left( \bm{S}^{+}\bm{X}\bm{X}^{\top}\bm{S}^{+}\frac{\partial \bm{S}}{\partial Y_{\alpha \beta}}(\bm{I}-\bm{S}\bm{S}^{+})\right)_{kl} 
 =\left(  \bm{A}\frac{\partial \bm{S}}{\partial Y_{\alpha \beta}}\bm{B}\right)_{kl}\nonumber\\
 &=A_{k \beta}(\bm{Y}\bm{B})_{\alpha l}+(\bm{A}\bm{Y}^{\top})_{k \alpha}B_{\beta l}
 = (\bm{S}^{+}\bm{X}\bm{X}^{\top}\bm{S}^{+})_{k \beta}(\bm{Y}(\bm{I}-\bm{S}\bm{S}^{+}))_{\alpha l}+(\bm{S}^{+}\bm{X}\bm{X}^{\top}\bm{S}^{+}\bm{Y}^{\top})_{k \alpha}(\bm{I}-\bm{S}\bm{S}^{+})_{\beta l}\nonumber
&&
\end{flalign*}
and then,
 \begin{flalign}\label{L1p7}
 &\left( \bm{S}^{+}\bm{X}\bm{X}^{\top}\bm{S}\bm{S}^{+}\bm{S}^{+}\frac{\partial \bm{S}}{\partial Y_{\alpha \beta}}(\bm{I}-\bm{S}\bm{S}^{+})\right)_{kl}= (\bm{S}^{+}\bm{X}\bm{X}^{\top}\bm{S}^{+}\bm{Y}^{\top})_{k \alpha}(\bm{I}-\bm{S}\bm{S}^{+})_{\beta l}.&&
\end{flalign}
Therefore, from $\eqref{L1p1}-\eqref{L1p7}$, we get
\begin{flalign*}
 \left(  \frac{\partial \bm{S}^{+}\bm{X}\bm{X}^{\top}\bm{S}\bm{S}^{+}}{\partial Y_{\alpha \beta}}\right)_{kl}&=-\bm{S}^{+}_{k \beta}(\bm{Y}\bm{S}^{+}\bm{X}\bm{X}^{\top}\bm{S}\bm{S}^{+})_{\alpha l} - (\bm{S}^{+}\bm{Y}^{\top})_{k \alpha}(\bm{S}^{+}\bm{X}\bm{X}^{\top}\bm{S}\bm{S}^{+})_{\beta l}\\
     & +(\bm{I}-\bm{S}\bm{S}^{+})_{k \beta}(\bm{Y}\bm{S}^{+}\bm{S}\bm{X}\bm{X}^{\top}\bm{S}\bm{S}^{+})_{\alpha l}+(\bm{S}^{+}\bm{S}^{+}\bm{Y}^{\top})_{k \alpha}((\bm{I}-\bm{S}\bm{S}^{+})\bm{X}\bm{X}^{\top}\bm{S}\bm{S}^{+})_{\beta l}\\
     & +(\bm{S}^{+}\bm{X}\bm{X}^{\top})_{k \beta}(\bm{Y}\bm{S}^{+})_{\alpha l}+(\bm{S}^{+}\bm{X}\bm{X}^{\top}\bm{Y}^{\top})_{k \alpha}\bm{S}^{+}_{\beta l}-(\bm{S}^{+}\bm{X}\bm{X}^{\top}\bm{S}\bm{S}^{+})_{k \beta}(\bm{Y}\bm{S}^{+})_{\alpha l}\\
     & -(\bm{S}^{+}\bm{X}\bm{X}^{\top}\bm{S}\bm{S}^{+}\bm{Y}^{\top})_{k \alpha}\bm{S}^{+}_{\beta l}+(\bm{S}^{+}\bm{X}\bm{X}^{\top}\bm{S}^{+}\bm{Y}^{\top})_{k \alpha}(\bm{I}-\bm{S}\bm{S}^{+})_{\beta l},&&
\end{flalign*}
which completes the proof.
\end{proof}
By using \Cref{Lemma1}, we establish below a lemma which is crucial in deriving the main result of this paper. As intermediate step, we derive first the following proposition. To simplify some mathematical expressions, let
\begin{eqnarray*}
A^{\alpha,k,\beta}_{1}=-\bm{S}^{+}_{k \beta}(\bm{Y}\bm{S}^{+}\bm{X}\bm{X}^{\top}\bm{S}\bm{S}^{+})_{\alpha \beta};\quad{ } \quad{ }
A^{\alpha,k,\beta}_{2}=-(\bm{S}^{+}\bm{Y}^{\top})_{k \alpha}(\bm{S}^{+}\bm{X}\bm{X}^{\top}\bm{S}\bm{S}^{+})_{\beta \beta};\\
A^{\alpha,k,\beta}_{3}=(\bm{I}-\bm{S}\bm{S}^{+})_{k \beta}(\bm{Y}\bm{S}^{+}\bm{S}\bm{X}\bm{X}^{\top}\bm{S}\bm{S}^{+})_{\alpha \beta};\quad{ }
A^{\alpha,k,\beta}_{4}=(\bm{S}^{+}\bm{S}^{+}\bm{Y}^{\top})_{k \alpha}((\bm{I}-\bm{S}\bm{S}^{+})\bm{X}\bm{X}^{\top}\bm{S}\bm{S}^{+})_{\beta \beta}\\
A^{\alpha,k,\beta}_{5}=(\bm{S}^{+}\bm{X}\bm{X}^{\top})_{k \beta}(\bm{Y}\bm{S}^{+})_{\alpha \beta}; \,
A^{\alpha,k,\beta}_{6}=(\bm{S}^{+}\bm{X}\bm{X}^{\top}\bm{Y}^{\top})_{k \alpha}\bm{S}^{+}_{\beta \beta}; \,
A^{\alpha,k,\beta}_{7}=-(\bm{S}^{+}\bm{X}\bm{X}^{\top}\bm{S}\bm{S}^{+})_{k \beta}(\bm{Y}\bm{S}^{+})_{\alpha \beta};\\
A^{\alpha,k,\beta}_{8}=-(\bm{S}^{+}\bm{X}\bm{X}^{\top}\bm{S}\bm{S}^{+}\bm{Y}^{\top})_{k \alpha}\bm{S}^{+}_{\beta \beta}; \quad{ }
A^{\alpha,k,\beta}_{9}=(\bm{S}^{+}\bm{X}\bm{X}^{\top}\bm{S}^{+}\bm{Y}^{\top})_{k \alpha}(\bm{I}-\bm{S}\bm{S}^{+})_{\beta \beta}.
\end{eqnarray*}
\begin{proposition}\label{prop:A1_t_A9} Let $\bm{Y}$, $\bm{X}$, $F$ and $\bm{G}$ be as defined in \Cref{Theorem2}. Then, \\
$(i)$ \quad{ }
            $\ds{\sum_{\alpha,k,\beta}^{}}{Y_{\alpha k}A^{\alpha,k,\beta}_{1}}
            =-F$;  
\quad{ } $(ii)$\quad{ }
           $\ds{\sum_{\alpha,k,\beta}^{}}{Y_{\alpha k}A^{\alpha,k,\beta}_{2}}
            =-F\mathrm{tr}(\bm{S}^{+}\bm{S})$;\quad{ }  
$(iii)$\quad{ }
            $\ds{\sum_{\alpha,k,\beta}^{}}{Y_{\alpha k}A^{\alpha,k,\beta}_{3}}=0$; \\ 
$(iv)$\quad{ }
             $\ds{\sum_{\alpha,k,\beta}^{}}{Y_{\alpha k}A^{\alpha,k,\beta}_{4}}
            =0$; \quad{ } 
$(v)$  \quad{ }     
            $\ds{\sum_{\alpha,k,\beta}^{}}{Y_{\alpha k}A^{\alpha,k,\beta}_{5}}=F$;
$(vi)$ \quad{ }   
           $\ds{\sum_{\alpha,k,\beta}^{}}{Y_{\alpha k}A^{\alpha,k,\beta}_{6}}
            =\mathrm{tr}(\bm{S}^{+})\mathrm{tr}(\bm{S}^{+}\bm{X}\bm{X}^{\top}\bm{S})$;\\
$(vii)$   
            $\ds{\sum_{\alpha,k,\beta}^{}}{Y_{\alpha k}A^{\alpha,k,\beta}_{7}}=-F$; \quad{ } 
$(viii)$  \quad{ } 
           $\ds{\sum_{\alpha,k,\beta}^{}}{Y_{\alpha k}A^{\alpha,k,\beta}_{8}}
             =-\mathrm{tr}(\bm{S}^{+})\mathrm{tr}(\bm{S}^{+}\bm{X}\bm{X}^{\top}\bm{S})$;\\ 
$(ix)$   
            $\ds{\sum_{\alpha,k,\beta}^{}}{Y_{\alpha k}A^{\alpha,k,\beta}_{9}}=(p-\mathrm{tr}(\bm{S}\bm{S}^{+}))F$. 
\end{proposition}
\begin{proof}
        $(i)$ \quad{ } We have $\ds{\sum_{\alpha,k,\beta}^{}}{Y_{\alpha k}A^{\alpha,k,\beta}_{1}}=-\sum_{\alpha,k,\beta}^{}{Y_{\alpha k}\bm{S}^{+}_{k \beta}(\bm{Y}\bm{S}^{+}\bm{X}\bm{X}^{\top}\bm{S}\bm{S}^{+})_{\alpha \beta}}$. Then,
        \begin{flalign*}
            \nonumber &\sum_{\alpha,k,\beta}^{}{Y_{\alpha k}A^{\alpha,k,\beta}_{1}} 
            =-\sum_{\alpha,k}^{}{Y_{\alpha k}\sum_{\beta}^{}{(\bm{Y}\bm{S}^{+}\bm{X}\bm{X}^{\top}\bm{S}\bm{S}^{+})_{\alpha \beta}\bm{S}^{+}_{\beta k}}} 
            =-\sum_{\alpha,k}^{}{Y_{\alpha k}(\bm{Y}\bm{S}^{+}\bm{X}\bm{X}^{\top}\bm{S}\bm{S}^{+}\bm{S}^{+})_{\alpha k}}\\
            &\nonumber
            =-\sum_{\alpha,k}^{}{(\bm{Y}\bm{S}^{+}\bm{X}\bm{X}^{\top}\bm{S}^{+})_{\alpha k}\bm{Y}^{\top}_{k \alpha }} 
            =-\sum_{\alpha}^{}{(\bm{Y}\bm{S}^{+}\bm{X}\bm{X}^{\top}\bm{S}^{+}\bm{Y}^{\top})_{\alpha \alpha }} 
            =-\mathrm{tr}(\bm{Y}\bm{S}^{+}\bm{X}\bm{X}^{\top}\bm{S}^{+}\bm{Y}^{\top}).
            && 
        \end{flalign*}
        This gives $\ds{\sum_{\alpha,k,\beta}^{}}{Y_{\alpha k}A^{\alpha,k,\beta}_{1}}
            =-\mathrm{tr}(\bm{X}^{\top}\bm{S}^{+}\bm{Y}^{\top}\bm{Y}\bm{S}^{+}\bm{X})$ and then,
        \begin{eqnarray}
            \sum_{\alpha,k,\beta}^{}{Y_{\alpha k}A^{\alpha,k,\beta}_{1}}
            =-\mathrm{tr}(\bm{X}^{\top}\bm{S}^{+}\bm{S}\bm{S}^{+}\bm{X})
            =-\mathrm{tr}(\bm{X}^{\top}\bm{S}^{+}\bm{X})
            =-F. 
            \label{a1}
        \end{eqnarray}
 $(ii)$ \quad{ } We have
        \begin{flalign*}
            \nonumber &\sum_{\alpha,k,\beta}^{}{Y_{\alpha k}A^{\alpha,k,\beta}_{2}}=-\sum_{\alpha,k,\beta}^{}{Y_{\alpha k}(\bm{S}^{+}\bm{Y}^{\top})_{k \alpha}(\bm{S}^{+}\bm{X}\bm{X}^{\top}\bm{S}\bm{S}^{+})_{\beta \beta}} 
            =-\sum_{\alpha,k}^{}{Y_{\alpha k}(\bm{S}^{+}\bm{Y}^{\top})_{k \alpha}\sum_{\beta}^{}{(\bm{S}^{+}\bm{X}\bm{X}^{\top}\bm{S}\bm{S}^{+})_{\beta \beta}}}\\
            &\nonumber=-\sum_{\alpha,k}^{}{Y_{\alpha k}(\bm{S}^{+}\bm{Y}^{\top})_{k \alpha}\mathrm{tr}(\bm{S}^{+}\bm{X}\bm{X}^{\top}\bm{S}\bm{S}^{+})} 
            =-\mathrm{tr}(\bm{X}^{\top}\bm{S}\bm{S}^{+}\bm{S}^{+}\bm{X})\sum_{\alpha}^{}{(\bm{Y}\bm{S}^{+}\bm{Y}^{\top})_{\alpha \alpha}}.
        \end{flalign*}
        Hence,
        \begin{flalign}
           \sum_{\alpha,k,\beta}^{}{Y_{\alpha k}A^{\alpha,k,\beta}_{2}}=-\mathrm{tr}(\bm{X}^{\top}\bm{S}^{+}\bm{X})\mathrm{tr}(\bm{Y}\bm{S}^{+}\bm{Y}^{\top}) 
            =-F\mathrm{tr}(\bm{S}^{+}\bm{Y}^{\top}Y) 
            =-F\mathrm{tr}(\bm{S}^{+}\bm{S}).&&\label{a2}
        \end{flalign}
 $(iii)$ \quad{ } We have $\ds{\sum_{\alpha,k,\beta}^{}}{Y_{\alpha k}A^{\alpha,k,\beta}_{3}}=\ds{\sum_{\alpha,k,\beta}^{}}{Y_{\alpha k}(\bm{I}-\bm{S}\bm{S}^{+})_{k \beta}(\bm{Y}\bm{S}^{+}\bm{S}\bm{X}\bm{X}^{\top}\bm{S}\bm{S}^{+})_{\alpha \beta}}$ then,
            \begin{flalign*}
            \nonumber & \sum_{\alpha,k,\beta}^{}{Y_{\alpha k}A^{\alpha,k,\beta}_{3}}
            =\sum_{\alpha,k}^{}{Y_{\alpha k}\sum_{\beta}^{}{(\bm{Y}\bm{S}^{+}\bm{S}\bm{X}\bm{X}^{\top}\bm{S}\bm{S}^{+})_{\alpha \beta}(\bm{I}-\bm{S}\bm{S}^{+})_{\beta k}}}.
&&
        \end{flalign*}
        Then
                   \begin{flalign}
            \sum_{\alpha,k,\beta}^{}{Y_{\alpha k}A^{\alpha,k,\beta}_{3}}
            &=\sum_{\alpha,k}^{}{Y_{\alpha k}(\bm{Y}\bm{S}^{+}\bm{S}\bm{X}\bm{X}^{\top}\bm{S}\bm{S}^{+}(\bm{I}-\bm{S}\bm{S}^{+}))_{\alpha k}}=0.&&\label{a3}
        \end{flalign}
 $(iv)$ \quad{ } We have $\ds{\sum_{\alpha,k,\beta}^{}}{Y_{\alpha k}A^{\alpha,k,\beta}_{4}}=\ds{\sum_{\alpha,k,\beta}^{}}{Y_{\alpha k}(\bm{S}^{+}\bm{S}^{+}\bm{Y}^{\top})_{k \alpha}((\bm{I}-\bm{S}\bm{S}^{+})\bm{X}\bm{X}^{\top}\bm{S}\bm{S}^{+})_{\beta \beta}}$. Then,
            \begin{eqnarray*}
            \nonumber &\ds{\sum_{\alpha,k,\beta}^{}}{Y_{\alpha k}A^{\alpha,k,\beta}_{4}}
            =\ds{\sum_{\alpha,k}^{}}{Y_{\alpha k}(\bm{S}^{+}\bm{S}^{+}\bm{Y}^{\top})_{k \alpha}\ds{\sum_{\beta}^{}}{((\bm{I}-\bm{S}\bm{S}^{+})\bm{X}\bm{X}^{\top}\bm{S}\bm{S}^{+})_{\beta \beta}}}\\
            &\nonumber
            =\ds{\sum_{\alpha,k}^{}}{Y_{\alpha k}(\bm{S}^{+}\bm{S}^{+}\bm{Y}^{\top})_{k \alpha}\textnormal{tr}((\bm{I}-\bm{S}\bm{S}^{+})\bm{X}\bm{X}^{\top}\bm{S}\bm{S}^{+})}. 
        \end{eqnarray*}
        Hence,
        \begin{flalign}
             \sum_{\alpha,k,\beta}^{}{Y_{\alpha k}A^{\alpha,k,\beta}_{4}}
            &
            =\mathrm{tr}(\bm{S}\bm{S}^{+}(\bm{I}-\bm{S}\bm{S}^{+})\bm{X}\bm{X}^{\top})\sum_{\alpha,k}^{}{Y_{\alpha k}(\bm{S}^{+}\bm{S}^{+}\bm{Y}^{\top})_{k \alpha}}=0.&&\label{a4}
        \end{flalign}
 $(v)$  We have, 
            \begin{flalign*}
            \nonumber \sum_{\alpha,k,\beta}^{}{Y_{\alpha k}A^{\alpha,k,\beta}_{5}}&=\sum_{\alpha,k,\beta}^{}{Y_{\alpha k}(\bm{S}^{+}\bm{X}\bm{X}^{\top})_{k \beta}(\bm{Y}\bm{S}^{+})_{\alpha \beta}}
            =\sum_{\alpha,k}^{}{Y_{\alpha k}\sum_{\beta}^{}{(\bm{S}^{+}\bm{X}\bm{X}^{\top})_{k \beta}(\bm{S}^{+}\bm{Y}^{\top})_{\beta \alpha}}}\\
            &\nonumber=\sum_{\alpha,k}^{}{Y_{\alpha k}(\bm{S}^{+}\bm{X}\bm{X}^{\top}\bm{S}^{+}\bm{Y}^{\top})_{k \alpha}}
            =\sum_{\alpha}^{}{(\bm{Y}\bm{S}^{+}\bm{X}\bm{X}^{\top}\bm{S}^{+}\bm{Y}^{\top})_{\alpha \alpha}}
=\mathrm{tr}(\bm{Y}\bm{S}^{+}\bm{X}\bm{X}^{\top}\bm{S}^{+}\bm{Y}^{\top}). 
            &&
        \end{flalign*}
        Then,
                  \begin{flalign}
            \sum_{\alpha,k,\beta}^{}{Y_{\alpha k}A^{\alpha,k,\beta}_{5}}=\mathrm{tr}(\bm{X}^{\top}\bm{S}^{+}\bm{Y}^{\top}\bm{Y}\bm{S}^{+}\bm{X})
            =\mathrm{tr}(\bm{X}^{\top}\bm{S}^{+}\bm{S}\bm{S}^{+}\bm{X}) =\mathrm{tr}(\bm{X}^{\top}\bm{S}^{+}\bm{X})=F.
            &&
            \label{a5}
        \end{flalign}
 $(vi)$ We have 
            \begin{flalign}
            \nonumber &\sum_{\alpha,k,\beta}^{}{Y_{\alpha k}A^{\alpha,k,\beta}_{6}}=\sum_{\alpha,k,\beta}^{}{Y_{\alpha k}(\bm{S}^{+}\bm{X}\bm{X}^{\top}\bm{Y}^{\top})_{k \alpha}\bm{S}^{+}_{\beta \beta}}
            =\sum_{\alpha,k}^{}{Y_{\alpha k}(\bm{S}^{+}\bm{X}\bm{X}^{\top}\bm{Y}^{\top})_{k \alpha}\sum_{\beta}^{}{\bm{S}^{+}_{\beta \beta}}}\\
            &\nonumber=\sum_{\alpha,k}^{}{Y_{\alpha k}(\bm{S}^{+}\bm{X}\bm{X}^{\top}\bm{Y}^{\top})_{k \alpha}\textnormal{tr}(\bm{S}^{+})}
            =\mathrm{tr}(\bm{S}^{+})\sum_{\alpha,k}^{}{Y_{\alpha k}(\bm{S}^{+}\bm{X}\bm{X}^{\top}\bm{Y}^{\top})_{k \alpha}}
            =\mathrm{tr}(\bm{S}^{+})\sum_{\alpha}^{}{(\bm{Y}\bm{S}^{+}\bm{X}\bm{X}^{\top}\bm{Y}^{\top})_{\alpha \alpha}}.
            &&
        \end{flalign}
        Then, $\ds{\sum_{\alpha,k,\beta}^{}}{Y_{\alpha k}A^{\alpha,k,\beta}_{6}}=\mathrm{tr}(\bm{S}^{+})\mathrm{tr}(\bm{Y}\bm{S}^{+}\bm{X}\bm{X}^{\top}\bm{Y}^{\top})$, and then,
                 \begin{flalign}
           \sum_{\alpha,k,\beta}^{}{Y_{\alpha k}A^{\alpha,k,\beta}_{6}} 
            =\mathrm{tr}(\bm{S}^{+})\mathrm{tr}(\bm{S}^{+}\bm{X}\bm{X}^{\top}\bm{Y}^{\top}Y)
            =\mathrm{tr}(\bm{S}^{+})\mathrm{tr}(\bm{S}^{+}\bm{X}\bm{X}^{\top}\bm{S}).
            &&
            \label{a6}
        \end{flalign}
(vii) We have
            \begin{flalign}
            \nonumber &\sum_{\alpha,k,\beta}^{}{Y_{\alpha k}A^{\alpha,k,\beta}_{7}}=-\sum_{\alpha,k,\beta}^{}{Y_{\alpha k}(\bm{S}^{+}\bm{X}\bm{X}^{\top}\bm{S}\bm{S}^{+})_{k \beta}(\bm{Y}\bm{S}^{+})_{\alpha \beta}}
            =-\sum_{\alpha,k}^{}{Y_{\alpha k}\sum_{\beta}^{}{(\bm{S}^{+}\bm{X}\bm{X}^{\top}\bm{S}\bm{S}^{+})_{k \beta}(\bm{S}^{+}\bm{Y}^{\top})_{\beta \alpha}}}\\
            &\nonumber=-\sum_{\alpha,k}^{}{Y_{\alpha k}(\bm{S}^{+}\bm{X}\bm{X}^{\top}\bm{S}\bm{S}^{+}\bm{S}^{+}\bm{Y}^{\top})_{k \alpha}}
            =-\sum_{\alpha,k}^{}{Y_{\alpha k}(\bm{S}^{+}\bm{X}\bm{X}^{\top}\bm{S}^{+}\bm{Y}^{\top})_{k \alpha}}
            =-\sum_{\alpha}^{}{(\bm{Y}\bm{S}^{+}\bm{X}\bm{X}^{\top}\bm{S}^{+}\bm{Y}^{\top})_{\alpha \alpha}}.
            &&
        \end{flalign}
 Then,
             \begin{flalign}
            \sum_{\alpha,k,\beta}^{}{Y_{\alpha k}A^{\alpha,k,\beta}_{7}}
            =-\mathrm{tr}(\bm{Y}\bm{S}^{+}\bm{X}\bm{X}^{\top}\bm{S}^{+}\bm{Y}^{\top})
            =-\mathrm{tr}(\bm{X}^{\top}\bm{S}^{+}\bm{S}\bm{S}^{+}\bm{X}) 
            =-\mathrm{tr}(\bm{X}^{\top}\bm{S}^{+}\bm{X})=-F.&& \label{a7}
        \end{flalign}
 $(viii)$ We have
        \begin{flalign}
            \nonumber & \sum_{\alpha,k,\beta}^{}{Y_{\alpha k}A^{\alpha,k,\beta}_{8}}=-\sum_{\alpha,k,\beta}^{}{Y_{\alpha k}(\bm{S}^{+}\bm{X}\bm{X}^{\top}\bm{S}\bm{S}^{+}\bm{Y}^{\top})_{k \alpha}\bm{S}^{+}_{\beta \beta}} 
            =-\sum_{\alpha,k}^{}{Y_{\alpha k}(\bm{S}^{+}\bm{X}\bm{X}^{\top}\bm{S}\bm{S}^{+}\bm{Y}^{\top})_{k \alpha}\sum_{\beta}^{}{\bm{S}^{+}_{\beta \beta}}}\\
            &\nonumber=-\sum_{\alpha,k}^{}{Y_{\alpha k}(\bm{S}^{+}\bm{X}\bm{X}^{\top}\bm{S}\bm{S}^{+}\bm{Y}^{\top})_{k \alpha}\textnormal{tr}(\bm{S}^{+})}
            =-\mathrm{tr}(\bm{S}^{+})\sum_{\alpha}^{}{(\bm{Y}\bm{S}^{+}\bm{X}\bm{X}^{\top}\bm{S}\bm{S}^{+}\bm{Y}^{\top})_{\alpha \alpha}} \\
            &\nonumber
            =-\mathrm{tr}(\bm{S}^{+})\mathrm{tr}(\bm{Y}\bm{S}^{+}\bm{X}\bm{X}^{\top}\bm{S}\bm{S}^{+}\bm{Y}^{\top}) 
           =-\mathrm{tr}(\bm{S}^{+})\mathrm{tr}(\bm{S}^{+}\bm{X}\bm{X}^{\top}\bm{S}\bm{S}^{+}\bm{Y}^{\top}Y).
            &&
        \end{flalign}
        Therefore,
         \begin{eqnarray}
           \sum_{\alpha,k,\beta}^{}{Y_{\alpha k}A^{\alpha,k,\beta}_{8}}
            =-\mathrm{tr}(\bm{S}^{+})\mathrm{tr}(\bm{S}^{+}\bm{X}\bm{X}^{\top}\bm{S}).
            \label{a8}
        \end{eqnarray}
 (ix) \quad{ } We have $\ds{\sum_{\alpha,k,\beta}^{}}{Y_{\alpha k}A^{\alpha,k,\beta}_{9}}=\ds{\sum_{\alpha,k,\beta}^{}}{Y_{\alpha k}(\bm{S}^{+}\bm{X}\bm{X}^{\top}\bm{S}^{+}\bm{Y}^{\top})_{k \alpha}(\bm{I}-\bm{S}\bm{S}^{+})_{\beta \beta}}$ and then,
            \begin{flalign*}
            \nonumber&\sum_{\alpha,k,\beta}^{}{Y_{\alpha k}A^{\alpha,k,\beta}_{9}} 
            =\sum_{\alpha,k}^{}{Y_{\alpha k}(\bm{S}^{+}\bm{X}\bm{X}^{\top}\bm{S}^{+}\bm{Y}^{\top})_{k \alpha}\sum_{\beta}^{}{(\bm{I}-\bm{S}\bm{S}^{+})_{\beta \beta}}}\\
            &\nonumber=\sum_{\alpha,k}^{}{Y_{\alpha k}(\bm{S}^{+}\bm{X}\bm{X}^{\top}\bm{S}^{+}\bm{Y}^{\top})_{k \alpha}\mathrm{tr}(\bm{I}-\bm{S}\bm{S}^{+})}
            =\sum_{\alpha,k}^{}{Y_{\alpha k}(\bm{S}^{+}\bm{X}\bm{X}^{\top}\bm{S}^{+}\bm{Y}^{\top})_{k \alpha}(p-\mathrm{tr}(\bm{S}\bm{S}^{+}))}\\
            &\nonumber=(p-\mathrm{tr}(\bm{S}\bm{S}^{+}))\sum_{\alpha}^{}{(\bm{Y}\bm{S}^{+}\bm{X}\bm{X}^{\top}\bm{S}^{+}\bm{Y}^{\top})_{\alpha \alpha}}
            =(p-\mathrm{tr}(\bm{S}\bm{S}^{+}))\mathrm{tr}(\bm{Y}\bm{S}^{+}\bm{X}\bm{X}^{\top}\bm{S}^{+}\bm{Y}^{\top}).
            &&
        \end{flalign*}
        Then
          \begin{flalign*}
            \nonumber&\sum_{\alpha,k,\beta}^{}{Y_{\alpha k}A^{\alpha,k,\beta}_{9}}
            =(p-\mathrm{tr}(\bm{S}\bm{S}^{+}))\mathrm{tr}(\bm{X}^{\top}\bm{S}^{+}\bm{Y}^{\top}\bm{Y}\bm{S}^{+}\bm{X})
            =(p-\mathrm{tr}(\bm{S}\bm{S}^{+}))\mathrm{tr}(\bm{X}^{\top}\bm{S}^{+}\bm{S}\bm{S}^{+}\bm{X}),
            &&
        \end{flalign*}
        and then,
            \begin{flalign}
            \sum_{\alpha,k,\beta}^{}{Y_{\alpha k}A^{\alpha,k,\beta}_{9}}
            =(p-\mathrm{tr}(\bm{S}\bm{S}^{+}))\mathrm{tr}(\bm{X}^{\top}\bm{S}^{+}\bm{X})=(p-\mathrm{tr}(\bm{S}\bm{S}^{+}))F,&&\label{a9}
        \end{flalign}
 this completes the proof.              
\end{proof}
By combining \Cref{Lemma1} and \Cref{prop:A1_t_A9}, we derive the following result. 
\begin{lemma}\label{Lemma2}
     Let $\bm{Y}$, $\bm{X}$, $F$ and $\bm{G}$ be as defined in \Cref{Theorem2}.
     Then
        \begin{flalign*}
        (i) \quad \frac{\partial G_{k l}}{\partial Y_{\alpha \beta}}&=\frac{2r(F)r^{\prime}(F)}{F^2}\left(\frac{\partial F}{\partial Y_{\alpha \beta}}\right)(\bm{S}^{+}\bm{X}\bm{X}^{\top}\bm{S}\bm{S}^{+})_{k l} 
        -\frac{2r^2(F)}{F^{3}}\left(\frac{\partial F}{\partial Y_{\alpha \beta}}\right)(\bm{S}^{+}\bm{X}\bm{X}^{\top}\bm{S}\bm{S}^{+})_{k l}\\
        &+\frac{r^2(F)}{F^2}\frac{\partial}{\partial Y_{\alpha \beta}}(\bm{S}^{+}\bm{X}\bm{X}^{\top}\bm{S}\bm{S}^{+})_{k l};&&
    \end{flalign*}
        \begin{flalign*}
        (ii)\quad \sum_{\alpha,k,\beta}^{}{Y_{\alpha k}\frac{\partial }{\partial Y_{\alpha \beta}}(\bm{S}\bm{S}^{+}\bm{X}\bm{X}^{\top}\bm{S}^{+})_{\beta k}}=F\left[p-2\mathrm{tr}(\bm{S}\bm{S}^{+})-1\right];&&
    \end{flalign*}
    \begin{flalign*}
        (iii) \quad \sum_{\alpha,\beta,k}^{}{Y_{\alpha k}\frac{\partial G_{k \beta}}{\partial Y_{\alpha \beta}}}
           =-\frac{4r(F)r^{\prime}(F)}{F^2}\mathrm{tr}\left((\bm{X}^{\top}\bm{S}^{+}\bm{X})^2\right)
           +\frac{r^2(F)}{F}\left(\frac{4\mathrm{tr}\left((\bm{X}^{\top}\bm{S}^{+}\bm{X})^2\right)}{F^2}+p-2\mathrm{tr}(\bm{S}\bm{S}^{+})-1\right).&&
    \end{flalign*}
\end{lemma}

\begin{proof}
$(i)$\quad{ }  We have $ \frac{\partial G_{k l}}{\partial Y_{\alpha \beta}}=\frac{\partial}{\partial Y_{\alpha \beta}}\bigg\{\frac{r^2(F)}{F^2}(\bm{S}^{+}\bm{X}\bm{X}^{\top}\bm{S}\bm{S}^{+})_{k l}\bigg\}$. Then,
        \begin{flalign*}
            & \frac{\partial G_{k l}}{\partial Y_{\alpha \beta}} 
          =\frac{\partial}{\partial Y_{\alpha \beta}}\left(\frac{r^2(F)}{F^2}\right)(\bm{S}^{+}\bm{X}\bm{X}^{\top}\bm{S}\bm{S}^{+})_{k l}
          +\frac{r^2(F)}{F^2}\frac{\partial}{\partial Y_{\alpha \beta}}(\bm{S}^{+}\bm{X}\bm{X}^{\top}\bm{S}\bm{S}^{+})_{k l}\\
          &=\frac{2r(F)r^{\prime}(F)(\frac{\partial F}{\partial y_{\alpha \beta}})F^2-2F(\frac{\partial F}{\partial y_{\alpha \beta}})r^2(F)}{F^4}(\bm{S}^{+}\bm{X}\bm{X}^{\top}\bm{S}\bm{S}^{+})_{k l} 
          +\frac{r^2(F)}{F^2}\frac{\partial}{\partial Y_{\alpha \beta}}(\bm{S}^{+}\bm{X}\bm{X}^{\top}\bm{S}\bm{S}^{+})_{k l}.
&&
        \end{flalign*}
Hence,
\begin{flalign*}
            & \frac{\partial G_{k l}}{\partial Y_{\alpha \beta}}
          =\frac{2r(F)r^{\prime}(F)}{F^2}\left(\frac{\partial F}{\partial y_{\alpha \beta}}\right)(\bm{S}^{+}\bm{X}\bm{X}^{\top}\bm{S}\bm{S}^{+})_{k l} 
          -\frac{2r^2(F)}{F^3}\left(\frac{\partial F}{\partial y_{\alpha \beta}}\right)(\bm{S}^{+}\bm{X}\bm{X}^{\top}\bm{S}\bm{S}^{+})_{k l} \\
          &
          +\frac{r^2(F)}{F^2}\frac{\partial}{\partial Y_{\alpha \beta}}(\bm{S}^{+}\bm{X}\bm{X}^{\top}\bm{S}\bm{S}^{+})_{k l},&&
        \end{flalign*}
        this proves Part~$(i)$.

        \begin{flalign*}
            (ii)\quad \mbox{We have}
            \sum_{\alpha,k,\beta}^{}{Y_{\alpha k}\frac{\partial }{\partial y_{\alpha \beta}}(\bm{S}\bm{S}^{+}\bm{X}\bm{X}^{\top}\bm{S}^{+})_{\beta k}}=\sum_{\alpha,k,\beta}^{}{Y_{\alpha k}\frac{\partial }{\partial y_{\alpha \beta}}(\bm{S}^{+}\bm{X}\bm{X}^{\top}\bm{S}\bm{S}^{+})_{k \beta}}.&&
        \end{flalign*}
        By Part $(iv)$ of \Cref{Lemma1},  for $A^{\alpha,k,\beta}_{1},A^{\alpha,k,\beta}_{2},\cdots,A^{\alpha,k,\beta}_{9}$ defined in \Cref{prop:A1_t_A9}, we get
        \begin{eqnarray}
            \nonumber \sum_{\alpha,k,\beta}^{}&{Y_{\alpha k}\frac{\partial }{\partial y_{\alpha \beta}}(\bm{S}^{+}\bm{X}\bm{X}^{\top}\bm{S}\bm{S}^{+})_{k \beta}} 
            =\ds{\sum_{\alpha,k,\beta}^{}}{Y_{\alpha k}(A^{\alpha,k,\beta}_{1}+A^{\alpha,k,\beta}_{2}+\cdots+A^{\alpha,k,\beta}_{9})}. 
            \label{a}
        \end{eqnarray}
        Therefore, from $\eqref{a1}-\eqref{a9}$, we get
        \begin{flalign}
            \sum_{\alpha,k,\beta}^{}{Y_{\alpha k}A^{\alpha,k,\beta}_{1}}+\sum_{\alpha,k,\beta}^{}{Y_{\alpha k}A^{\alpha,k,\beta}_{5}}=0, \quad{ } 
            \sum_{\alpha,k,\beta}^{}{Y_{\alpha k}A^{\alpha,k,\beta}_{6}}+\sum_{\alpha,k,\beta}^{}{Y_{\alpha k}A^{\alpha,k,\beta}_{8}}=0, \label{a11}\\
            \sum_{\alpha,k,\beta}^{}{Y_{\alpha k}A^{\alpha,k,\beta}_{3}}=\sum_{\alpha,k,\beta}^{}{Y_{\alpha k}A^{\alpha,k,\beta}_{4}}=0.\label{a12}
        \end{flalign}
        Then, by \eqref{a11} and \eqref{a12},  together with \eqref{a2}, \eqref{a7} and \eqref{a9} we get
        \begin{flalign*}
            \sum_{\alpha,k,\beta}^{}{Y_{\alpha k}\left(\sum_{j=1}^{9}A^{\alpha,k,\beta}_{j}\right)}=\sum_{\alpha,k,\beta}^{}{Y_{\alpha k}A^{\alpha,k,\beta}_{2}} 
            +\sum_{\alpha,k,\beta}^{}{Y_{\alpha k}A^{\alpha,k,\beta}_{7}}
            +\sum_{\alpha,k,\beta}^{}{Y_{\alpha k}A^{\alpha,k,\beta}_{9}} 
            =F\left(p-2\mathrm{tr}(\bm{S}\bm{S}^{+})-1\right),&&
        \end{flalign*}
this proves Part~(ii).\\
$(iii)$ To prove Part~$(iii)$, we first note that
        \begin{flalign*}
            &\sum_{\alpha,\beta,k}^{}{Y_{\alpha k}\frac{\partial G_{k \beta}}{\partial Y_{\alpha \beta}}}
           =\frac{2r(F)r^{\prime}(F)}{F^2}\sum_{\alpha,\beta,k}^{}{Y_{\alpha k}\left(\frac{\partial F}{\partial y_{\alpha \beta}}\right)(\bm{S}\bm{S}^{+}\bm{X}\bm{X}^{\top}\bm{S}^{+})_{\beta k}}\\
           &-\frac{2r^2(F)}{F^3}\sum_{\alpha,\beta,k}^{}{Y_{\alpha k}\left(\frac{\partial F}{\partial y_{\alpha \beta}}\right)(\bm{S}\bm{S}^{+}\bm{X}\bm{X}^{\top}\bm{S}^{+})_{\beta k}} 
           +\frac{r^2(F)}{F^2}\sum_{\alpha,k,\beta}^{}{Y_{\alpha k}\frac{\partial }{\partial y_{\alpha \beta}}(\bm{S}\bm{S}^{+}\bm{X}\bm{X}^{\top}\bm{S}^{+})_{\beta k}}\\
           &=\frac{2r(F)r^{\prime}(F)}{F^2}\left(-2\mathrm{tr}\left((\bm{X}^{\top}\bm{S}^{+}\bm{X})^2\right)\right)
           -\frac{2r^2(F)}{F^3}\left(-2\mathrm{tr}\left((\bm{X}^{\top}\bm{S}^{+}\bm{X})^2\right)\right)
           +\frac{r^2(F)}{F^2}F(p-2\mathrm{tr}(\bm{S}\bm{S}^{+})-1)\\
           &=-\frac{4r(F)r^{\prime}(F)}{F^2}\mathrm{tr}\left((\bm{X}^{\top}\bm{S}^{+}\bm{X})^2\right) 
           +\frac{4r^2(F)\mathrm{tr}\left((\bm{X}^{\top}\bm{S}^{+}\bm{X})^2\right)}{F^3}
           +\frac{r^2(F)}{F}(p-2\mathrm{tr}(\bm{S}\bm{S}^{+})-1),
           &&
        \end{flalign*}
        and this gives
        \begin{flalign*}
            \sum_{\alpha,\beta,k}^{}{Y_{\alpha k}\frac{\partial G_{k \beta}}{\partial Y_{\alpha \beta}}}
           =-\frac{4r(F)r^{\prime}(F)}{F^2}\mathrm{tr}\left((\bm{X}^{\top}\bm{S}^{+}\bm{X})^2\right)
           +\frac{r^2(F)}{F}\left(\frac{4\mathrm{tr}\left((\bm{X}^{\top}\bm{S}^{+}\bm{X})^2\right)}{F^2}+p-2\mathrm{tr}(\bm{S}\bm{S}^{+})-1\right),
           &&
        \end{flalign*}
        which completes the proof.
    \end{proof}

\begin{proposition}\label{Proposition5}
 Suppose that the conditions of \Cref{Theorem2}. Then, 
     \begin{flalign*}
         (i) \quad \mathrm{tr}(\bm{G})=r^2(F)/F;&&
     \end{flalign*}
      \begin{flalign*}
        (ii) \quad &\mathrm{tr}\left(\bm{Y}^{\top}\nabla_{\bm{Y}}\bm{G}^{\top}\right)
        =-4\frac{r(F)r^{\prime}(F)}{F^{2}}\mathrm{tr}\left((\bm{X}^{\top}\bm{S}^{+}\bm{X})^2\right)
        +\left(\frac{4\mathrm{tr}\left((\bm{X}^{\top}\bm{S}^{+}\bm{X})^2\right)}{F^{2}}+p-2\mathrm{tr}(\bm{S}\bm{S}^{+})-1\right)\frac{r^2(F)}{F};&&
    \end{flalign*}

     \begin{flalign*}
         (iii)\quad \mathrm{div}_{\mathrm{vec}(\Tilde{\bm{Y}})}\mathrm{vec}(\Tilde{\bm{Y}}\bm{H})=\left( nq+p-2\mathrm{tr}(\bm{S}\bm{S}^{+})-1+4\mathrm{tr}\left((\bm{X}^{\top}\bm{S}^{+}\bm{X})^2\right)/F^2\right)r^2(F)/F
         -4r(F)r^{\prime}(F)/F^2.&&
     \end{flalign*}
\end{proposition}


    \begin{proof}
    $(i)$\quad{ }  Part~$(i)$ follows from algebraic computations. \\
         $(ii)$ \quad From \eqref{p4iii}, 
            $\mathrm{tr}\left(\bm{Y}^{\top}(\nabla_{Y}\bm{G}^{\top})\right)
           =\ds{\sum_{\alpha,\beta,k}^{}}{Y_{\alpha k}\frac{\partial G_{k \beta}}{\partial Y_{\alpha \beta}}}$. 
         Hence, by Part $(iv)$ of 
         \Cref{Lemma2},  we get
         \begin{flalign*}
             \mathrm{tr}&\left(\bm{Y}^{\top}(\nabla_{Y}\bm{G}^{\top})\right) 
             =-\frac{4r(F)r^{\prime}(F)}{F^2}\mathrm{tr}\left((\bm{X}^{\top}\bm{S}^{+}\bm{X})^2\right)
           +\frac{r^2(F)}{F}\left(\frac{4\mathrm{tr}\left((\bm{X}^{\top}\bm{S}^{+}\bm{X})^2\right)}{F^2}+p-2\mathrm{tr}(\bm{S}\bm{S}^{+})-1\right),&&
         \end{flalign*}
         this proves Part~$(ii)$.\\
        $(iii)$ \quad By Parts $(i)$-$(ii)$ along with \Cref{Proposition4}, $\mathrm{div}_{\mathrm{vec}(\Tilde{\bm{Y}})}\mathrm{vec}(\Tilde{\bm{Y}}\bm{H})= nq\mathrm{tr}(\bm{G})+\mathrm{tr}(\bm{Y}^{\top}\nabla_YG^{\top})$. Then, 
         \begin{eqnarray*}
            \mathrm{div}_{\mathrm{vec}(\Tilde{\bm{Y}})}\mathrm{vec}(\Tilde{\bm{Y}}\bm{H})
             =nqr^2(F)/F-4r(F)r^{\prime}(F)\mathrm{tr}\left((\bm{X}^{\top}\bm{S}^{+}\bm{X})^2\right)/F^2\\
             +\left(4\mathrm{tr}\left((\bm{X}^{\top}\bm{S}^{+}\bm{X})^2\right)/F^2+p-2\mathrm{tr}(\bm{S}\bm{S}^{+})-1\right)r^2(F)/F,
         \end{eqnarray*}
         and then,
                 \begin{flalign*}
            \mathrm{div}_{\mathrm{vec}(\Tilde{\bm{Y}})}\mathrm{vec}(\Tilde{\bm{Y}}\bm{H})
             =\left[ nq+p-2\mathrm{tr}(\bm{S}\bm{S}^{+})-1+4\mathrm{tr}\left((\bm{X}^{\top}\bm{S}^{+}\bm{X})^2\right)/F^2\right]r^2(F)/F
             -4r(F)r^{\prime}(F)/F^2,&&
         \end{flalign*}

         which completes the proof.
     \end{proof}
It should be noted that Part~$(ii)$ of \Cref{Proposition5} generalizes Lemma~1 in \cite{ChetelatWells}. Indeed, in the special case where $q=1$, we  $F=\mathrm{tr}(\bm{X}^{\top}\bm{S}^{+}\bm{X})=\bm{X}^{\top}\bm{S}^{+}\bm{X}$ and then, we get the result established in Lemma~1 of \cite{ChetelatWells}.

\section{On the derivation of the main results}\label{sec:appendMainres}
   \begin{proof}[Proof of \Cref{Lemma3}]
   (i) The proof of Part~(i) follows from classical differential calculus along with some algebraic computations.\\
    (ii) \quad{ } We have $
    \left(\frac{\partial \bm{S}\bm{S}^{+}\bm{X}}{\partial X_{ij}}\right)_{kl} = \ds{\frac{\partial}{\partial X_{ij}}}\ds{\sum_{\alpha}} (\bm{S}\bm{S}^{+})_{k\alpha}\bm{X}_{\alpha l}$. Then,
\begin{eqnarray*}
    \left(\frac{\partial \bm{S}\bm{S}^{+}\bm{X}}{\partial X_{ij}}\right)_{kl} 
    = \sum_{\alpha} (\bm{S}\bm{S}^{+})_{k\alpha}\frac{\partial \bm{X}_{\alpha l}}{\partial X_{ij}} 
    = \sum_{\alpha} (\bm{S}\bm{S}^{+})_{k\alpha}\delta_{\alpha i}\delta_{lj} 
    = (\bm{S}\bm{S}^{+})_{ki}\delta_{lj},
\end{eqnarray*}
this proves Part~$(ii)$.\\
(iii) \quad \text{ By Parts $(i)$ and $(ii)$, we get} \\
\begin{eqnarray*}
    & \frac{\partial g_{kl}}{\partial \bm{X}_{i,j}} = \left(\frac{\partial}{\partial X_{ij}}\frac{r(F)}{F}\right)(\bm{S}\bm{S}^{+}\bm{X})_{kl} + \frac{r(F)}{F}\left(\frac{\partial}{\partial X_{ij}}(\bm{S}\bm{S}^{+}\bm{X})_{kl}\right) \\
    & = \frac{r^{\prime}(F)F-r(F)}{F^2}\left(\frac{\partial F}{\partial X_{ij}}\right)(\bm{S}\bm{S}^{+}\bm{X})_{kl} + \frac{r(F)}{F}\left(\frac{\partial}{\partial X_{ij}}(\bm{S}\bm{S}^{+}\bm{X})_{kl}\right) \\
    & = \frac{2(Fr^{\prime}(F)-r(F))}{F^2}(\bm{S}^{+}\bm{X})_{ij}(\bm{S}\bm{S}^{+}\bm{X})_{kl} + \frac{r(F)}{F}(\bm{S}\bm{S}^{+})_{ki}\delta_{lj} ,
\end{eqnarray*}
this proves the statement in~$(iii)$.\\
(iv) \quad \text{By Part $(iii)$} \text{, we have}
\begin{eqnarray*}
    & \ds{\sum_{i,j}^{}}{\frac{\partial g_{ij}}{\partial X_{ij}}} = \ds{\sum_{i,j}^{}}{\bigg\{\frac{2Fr^{\prime}(F)-r(F)}{F^2}(\bm{S}^{+}\bm{X})_{ij}(\bm{S}\bm{S}^{+}\bm{X})_{ij}+\frac{r(F)}{F}(\bm{S}\bm{S}^{+})_{ii}\bigg\}} \\
    & = 2\frac{Fr^{\prime}(F)-r(F)}{\mathrm{tr}^2(F)}\sum_{i,j}^{}{(\bm{S}^{+}\bm{X})_{ij}(\bm{X}^{\top}\bm{S}^{+})_{ji}}+\frac{r(F)}{F} \mathrm{tr}(\bm{S}\bm{S}^{+}) \\
    & = 2\frac{Fr^{\prime}(F)-r(F)}{F^2}\sum_{i}^{}{(\bm{S}\bm{S}^{+}\bm{X}\bm{X}^{\top}\bm{S}^{+})_{ii}}+q\frac{r(F)}{F}\mathrm{tr}(\bm{S}\bm{S}^{+}). 
\end{eqnarray*}
Then, $\ds{\sum_{i,j}^{}}{\frac{\partial g_{ij}}{\partial X_{ij}}}
    = 2\frac{Fr^{\prime}(F)-r(F)}{F^2}\mathrm{tr}(\bm{S}\bm{S}^{+}\bm{X}\bm{X}^{\top}\bm{S}^{+})+q\frac{r(F)}{F}\mathrm{tr}(\bm{S}\bm{S}^{+})$ and then,
\begin{eqnarray*}
    \sum_{i,j}^{}{\frac{\partial g_{ij}}{\partial X_{ij}}} 
    = 2\frac{Fr^{\prime}(F)-r(F)}{F^2}\mathrm{tr}(\bm{X}^{\top}\bm{S}^{+}\bm{S}\bm{S}^{+}\bm{X})+q\frac{r(F)}{F}\mathrm{tr}(\bm{S}\bm{S}^{+}) \\
    = 2\frac{Fr^{\prime}(F)-r(F)}{F^2}\mathrm{tr}(\bm{X}^{\top}\bm{S}^{+}\bm{X})+q\frac{r(F)}{F}\mathrm{tr}(\bm{S}\bm{S}^{+}), 
\end{eqnarray*}
this gives 
$\ds{\sum_{i,j}^{}}{\frac{\partial g_{ij}}{\partial X_{ij}}} =
    2
    \frac{Fr^{\prime}(F)-r(F)}{F^2}F+q\frac{r(F)}{F}\mathrm{tr}(\bm{S}\bm{S}^{+})$. Hence,
\begin{eqnarray*}
    \sum_{i,j}^{}{\frac{\partial g_{ij}}{\partial X_{ij}}} 
    = 2r^{\prime}(F)-2\frac{r(F)}{F}+q\frac{r(F)}{F}\mathrm{tr}(\bm{S}\bm{S}^{+}) 
    = 2r^{\prime}(F)+(q\mathrm{tr}(\bm{S}\bm{S}^{+})-2)r(F)/F, 
\end{eqnarray*}
this completes the proof.
    \end{proof}
\begin{proposition}\label{prop:additional}
     Let $\bm{Y}$, $\bm{X}$, $F$ and $\bm{G}$ be as defined in \Cref{Theorem2}. 
     Then
          \begin{eqnarray*}
        \sum_{\alpha,k,\beta}^{}{Y_{\alpha k}\left(\frac{\partial F}{\partial Y_{\alpha \beta}}\right)(\bm{S}\bm{S}^{+}\bm{X}\bm{X}^{\top}\bm{S}^{+})_{\beta k}}=-2\mathrm{tr}((\bm{X}^{\top}\bm{S}^{+}\bm{X})^2). 
    \end{eqnarray*}
\end{proposition}
\begin{proof}
Let $\mathcal{J}_{0}=\ds{\sum_{\alpha,k,\beta}^{}}{Y_{\alpha k}\left(\frac{\partial F}{\partial y_{\alpha \beta}}\right)(\bm{S}\bm{S}^{+}\bm{X}\bm{X}^{\top}\bm{S}^{+})_{\beta k}}$.  \text{By Part $(iii)$ of \Cref{Theorem2},  we get }
        \begin{eqnarray*}
            \mathcal{J}_{0}
            =-2\sum_{\alpha,k,\beta}^{}{Y_{\alpha k}(\bm{S}^{+}\bm{X}\bm{X}^{\top}\bm{S}^{+}\bm{Y}^{\top})_{\beta \alpha}(\bm{S}\bm{S}^{+}\bm{X}\bm{X}^{\top}\bm{S}^{+})_{\beta k}}\\
            +2\sum_{\alpha,k,\beta}^{}{Y_{\alpha k}((\bm{I}-\bm{S}\bm{S}^{+})\bm{X}\bm{X}^{\top}\bm{S}^{+}\bm{S}^{+}\bm{Y}^{\top})_{\beta \alpha}(\bm{S}\bm{S}^{+}\bm{X}\bm{X}^{\top}\bm{S}^{+})_{\beta k}}.
        \end{eqnarray*}
        Then
        \begin{eqnarray*}
        \mathcal{J}_{0} 
            =-2\sum_{\alpha,k}^{}{Y_{\alpha k}\sum_{\beta}^{}{(\bm{Y}\bm{S}^{+}\bm{X}\bm{X}^{\top}\bm{S}^{+})_{\alpha \beta}(\bm{S}\bm{S}^{+}\bm{X}\bm{X}^{\top}\bm{S}^{+})_{\beta k}}} \\
            +2\sum_{\alpha,k}^{}{Y_{\alpha k}\sum_{\beta}^{}{(\bm{Y}\bm{S}^{+}\bm{S}^{+}\bm{X}\bm{X}^{\top}(\bm{I}-\bm{S}\bm{S}^{+}))_{\alpha \beta}(\bm{S}\bm{S}^{+}\bm{X}\bm{X}^{\top}\bm{S}^{+})_{\beta k}}},
\end{eqnarray*}
and then
\begin{eqnarray*}
\mathcal{J}_{0}
            &=&-2\sum_{\alpha,k}^{}{Y_{\alpha k}(\bm{Y}\bm{S}^{+}\bm{X}\bm{X}^{\top}\bm{S}^{+}\bm{S}\bm{S}^{+}\bm{X}\bm{X}^{\top}\bm{S}^{+})_{\alpha k}} \\
            & & \quad{ } +2\sum_{\alpha,k}^{}{Y_{\alpha k}(\bm{Y}\bm{S}^{+}\bm{S}^{+}\bm{X}\bm{X}^{\top}(\bm{I}-\bm{S}\bm{S}^{+})\bm{S}\bm{S}^{+}\bm{X}\bm{X}^{\top}\bm{S}^{+})_{\alpha k}}.
\end{eqnarray*}
        Further, since $(\bm{I}-\bm{S}\bm{S}^{+})\bm{S}\bm{S}^{+}=0$, we get $\mathcal{J}_{0}
            =-2\sum_{\alpha,k}^{}{\bm{Y}^{\top}_{k \alpha}(\bm{Y}\bm{S}^{+}\bm{X}\bm{X}^{\top}\bm{S}^{+}\bm{X}\bm{X}^{\top}\bm{S}^{+})_{\alpha k}}$. Then, 
            $\mathcal{J}_{0}
            =-2\sum_{k}^{}{(\bm{Y}^{\top}\bm{Y}\bm{S}^{+}\bm{X}\bm{X}^{\top}\bm{S}^{+}\bm{X}\bm{X}^{\top}\bm{S}^{+})_{kk}} 
            =-2\sum_{k}^{}{(\bm{S}\bm{S}^{+}\bm{X}\bm{X}^{\top}\bm{S}^{+}\bm{X}\bm{X}^{\top}\bm{S}^{+})_{kk}}$ and then
        \begin{eqnarray*}
            \mathcal{J}_{0}
            =-2\mathrm{tr}(\bm{S}\bm{S}^{+}\bm{X}\bm{X}^{\top}\bm{S}^{+}\bm{X}\bm{X}^{\top}\bm{S}^{+})
            =-2\mathrm{tr}(\bm{X}^{\top}\bm{S}^{+}\bm{X}\bm{X}^{\top}\bm{S}^{+}\bm{S}\bm{S}^{+}\bm{X}).
        \end{eqnarray*}
Hence,
\begin{eqnarray*}
            \mathcal{J}_{0}=\sum_{\alpha,k,\beta}^{}{Y_{\alpha k}\left(\frac{\partial F}{\partial y_{\alpha \beta}}\right)(\bm{S}\bm{S}^{+}\bm{X}\bm{X}^{\top}\bm{S}^{+})_{\beta k}}
            =-2\mathrm{tr}(\bm{X}^{\top}\bm{S}^{+}\bm{X}\bm{X}^{\top}\bm{S}^{+}\bm{X}) 
            =-2\mathrm{tr}((\bm{X}^{\top}\bm{S}^{+}\bm{X})^2),
        \end{eqnarray*}
        this completes the proof. 
\end{proof}

\bibliographystyle{abbrvnat}
\bibliography{AoS_matrix_case_Nov23_2023}

\begin{thebibliography}{14}
\providecommand{\natexlab}[1]{#1}
\providecommand{\url}[1]{\texttt{#1}}
\expandafter\ifx\csname urlstyle\endcsname\relax
  \providecommand{\doi}[1]{doi: #1}\else
  \providecommand{\doi}{doi: \begingroup \urlstyle{rm}\Url}\fi

\bibitem[Bilodeau and Kariya(1989)]{Bilodeau}
M.~Bilodeau and T.~Kariya.
\newblock Minimax estimators in the normal manova model.
\newblock \emph{Journal Of Multivariate Analysis}, 28:\penalty0 260--270, 1989.

\bibitem[Bodnar et~al.(2019)Bodnar, Okhrin, and Parolya]{Bodnar}
T.~Bodnar, O.~Okhrin, and N.~Parolya.
\newblock Optimal shrinkage estimator for high-dimensional mean vector.
\newblock \emph{Journal of Multivariate Analysis}, 170:\penalty0 63--79, 2019.

\bibitem[Ch\'etelat and Wells(2012)]{ChetelatWells}
D.~Ch\'etelat and M.~T. Wells.
\newblock Improved multivariate normal mean estimation with unknown covariance when $p$ is greater than $n$.
\newblock \emph{The Annals of Statistics}, 40\penalty0 (6):\penalty0 3137--3160, 2012.

\bibitem[Fan et~al.(2011)Fan, Lv, and Qi]{Fan}
J.~Fan, J.~Lv, and L.~Qi.
\newblock Sparse high dimensional models in economics.
\newblock \emph{Annu Rev Econom.}, 3:\penalty0 291--317, 2011.

\bibitem[Foroushani and Nkurunziza(2023)]{ArashSeverien}
A.~A. Foroushani and S.~Nkurunziza.
\newblock A note on improved multivariate normal mean estimation with unknown covariance when $p$ is greater than $n$.
\newblock \emph{arXiv:2311.13140v1}, pages 1--5, 2023.
\newblock URL \url{https://arxiv.org/abs/2311.13140}.

\bibitem[Fourdrinier et~al.(2004)Fourdrinier, Strawderman, and Wells]{STRAWDERMAN}
D.~Fourdrinier, W.~E. Strawderman, and M.~T. Wells.
\newblock Minimax estimators in the normal manova model.
\newblock \emph{Journal Of Multivariate Analysis}, 85:\penalty0 24--39, 2004.

\bibitem[Golub and Pereyra(1973)]{GolubPereyra}
G.~H. Golub and V.~Pereyra.
\newblock The differentiation of pseudo-inverses and nonlinear least squares problems whose variables separate.
\newblock \emph{SIAM J. Numer. Anal.}, 10:\penalty0 413--432, 1973.

\bibitem[Konno(1990)]{Konno1990}
Y.~Konno.
\newblock Families of minimax estimators of matrix of normal means with unknown covariance matrix.
\newblock \emph{Journal of Japan Statist. Soc.}, 20:\penalty0 191--201, 1990.

\bibitem[Konno(1991)]{Konno}
Y.~Konno.
\newblock On estimation of a matrix of normal means with unknown covariance matrix.
\newblock \emph{Journal of Multivariate Analysis}, 36:\penalty0 44--55, 1991.

\bibitem[Marchand and Perron(2001)]{Perron}
E.~Marchand and F.~Perron.
\newblock Improving on the mle of a bounded normal mean.
\newblock \emph{The Annals of Statistics}, 29\penalty0 (4):\penalty0 1078--1093, 2001.

\bibitem[Pardy et~al.(2018)Pardy, Galbraith, and Wilson]{Pardy}
C.~Pardy, S.~Galbraith, and S.~W. Wilson.
\newblock Integrative exploration of large high-dimensional datasets.
\newblock \emph{The Annals of Applied Statistics}, 12\penalty0 (1):\penalty0 178--199, 2018.

\bibitem[Stein(1981)]{SteinCharles}
C.~M. Stein.
\newblock Estimation of the mean of a multivariate normal distribution.
\newblock \emph{Ann. Statist.}, 9\penalty0 (6):\penalty0 1135--1151, 1981.

\bibitem[Tzeng et~al.(2008)Tzeng, Lu, and Li]{Tzeng}
J.~Tzeng, H.~Lu, and W.~Li.
\newblock Multidimensional scaling for large genomic data sets.
\newblock \emph{BMC Bioinformatics}, 9\penalty0 (179):\penalty0 1--17, 2008.

\bibitem[Zhou et~al.(2013)Zhou, Li, and Zhu]{ZhouLiZhu}
H.~Zhou, L.~Li, and H.~Zhu.
\newblock Tensor regression with applications in neuroimagingdata analysis.
\newblock \emph{Journal of American Statistical Association}, 108\penalty0 (502):\penalty0 540--552, 2013.

\end{thebibliography}
\end{document}